\documentclass[a4paper,oneside,english,reqno]{amsart}
\usepackage[utf8]{inputenx}
\usepackage[english]{babel}

\usepackage{amsmath,amsfonts,amssymb,mathrsfs,amscd,comment,latexsym,bbold}
\usepackage{amsthm, mathtools, stmaryrd}
\usepackage{textcomp, url, enumerate, latexsym, graphicx, varioref, hyperref, multirow, layout, siunitx, booktabs}
\usepackage{pdflscape}
\usepackage{verbatim}

\usepackage[matrix,arrow,curve]{xy}
\usepackage{tikz-cd}

\usepackage{colortbl}
\usepackage{hyperref}

\usepackage{enumitem}

\usepackage{csquotes}

\usepackage[top = 3cm, bottom = 4cm, left = 3cm, right = 3cm]{geometry}

\usepackage{microtype}

\usepackage{cleveref}

\RequirePackage[backend    = biber,
                sortcites  = true,
                giveninits = true,
                doi        = false,
                isbn       = false,
                url        = false,
                maxnames   = 6,
                style      = numeric]{biblatex}
\DeclareFieldFormat*{title}{#1}
\renewbibmacro{in:}{}

\usepackage{amsthm}

\numberwithin{equation}{section}
\theoremstyle{plain}
\newtheorem{theorem}[equation]{Theorem}
\newtheorem{corollary}[equation]{Corollary}
\newtheorem{proposition}[equation]{Proposition}
\newtheorem{lemma}[equation]{Lemma}
\newtheorem{thmX}{Theorem}
\theoremstyle{definition}

\newtheorem{definition}[equation]{Definition}
\newtheorem*{definition*}{Definition}
\newtheorem{example}[equation]{Example}
\theoremstyle{remark}
\newtheorem{remark}[equation]{Remark}
\newtheorem*{remm}{Remark}
\newtheorem*{notations}{Notation}

\renewcommand{\geq}{\geqslant}

\RequirePackage[color       = blue!20,
                bordercolor = black,
                textsize    = tiny,
                figwidth    = 0.99\linewidth]{todonotes}

\newcommand{\A}{\mathbb A}

\newcommand{\SH}{\mathbf{SH}}

\newcommand{\HH}{\mathbf{H}}
\newcommand{\HHd}{\HH^\bullet}

\newcommand{\Hfr}{\mathbf{H}^\fr}
\newcommand{\Htr}{\mathbf{H}^\tr}

\newcommand{\Spec}{\operatorname{Spec}}
\newcommand{\Ker}{\operatorname{Ker}}
\newcommand{\Coker}{\operatorname{Coker}}

\newcommand{\colim}{\operatorname{colim}}
\newcommand{\Fib}{\operatorname{Fib}}

\newcommand{\Fr}{\mathrm{Fr}}
\newcommand{\fr}{\mathrm{fr}}

\newcommand{\tr}{\mathrm{tr}}
\newcommand{\trd}{\mathbf{tr}}

\newcommand{\nfr}{\mathrm{nfr}}
\newcommand{\pfr}{\mathrm{pfr}}

\newcommand{\Sch}{\mathrm{Sch}}

\newcommand{\Schfns}{\mathrm{Sch}^\mathrm{fns}}
\newcommand{\Sm}{\mathrm{Sm}}
\newcommand{\Aff}{\mathrm{Aff}}

\newcommand{\Afffns}{\mathrm{Aff}^\mathrm{fn}}
\newcommand{\SmAff}{\mathrm{SmAff}}
\newcommand{\EssSmAff}{\mathrm{EssSmAff}}
\newcommand{\AffSm}{\mathrm{SmAff}}

\newcommand{\Gm}{\mathbb G_m}

\newcommand{\EssSm}{\mathrm{EssSm}}

\newcommand{\Pre}{\SPre}
\newcommand{\Corr}{\mathrm{Corr}}
\newcommand{\Corrtr}{\mathrm{Corr}}
\newcommand{\CorrAtr}{\mathrm{Corr}^{\A}}

\newcommand{\nis}{\mathrm{nis}}
\newcommand{\zar}{\mathrm{zar}}
\newcommand{\tf}{\mathrm{tf}}
\newcommand{\zf}{\mathrm{zf}}
\newcommand{\stau}{{\mathrm{s}\tau}}
\newcommand{\wtau}{{\mathrm{w}\tau}}

\newcommand{\SPre}{\mathrm{PSh}} %_{S^1}}
\newcommand{\SPrerad}{\SPre_\Sigma}
\newcommand{\SPrefr}{\SPre^\fr}
\newcommand{\SPretr}{\SPre^\tr}
\newcommand{\Shtr}{\Shv^\tr}

\newcommand{\SPretrrad}{\SPre^{\mathrm{tr}}_\Sigma}

\newcommand{\SPred}{\SPre^\bullet}
\newcommand{\SPretrd}{\SPre^\trd}

\newcommand{\SPreradd}{\SPrerad^\bullet}
\newcommand{\SPreradtr}{\SPrerad^\tr}
\newcommand{\SPreradAtr}{\SPrerad^{\A\tr}}
\newcommand{\SPreradtrd}{\SPrerad^\trd}

\newcommand{\SPreradtrA}{\HH^\tr_\Sigma}%\Pre^\tr_{\A^1}}

\newcommand{\Spc}{\mathrm{Spc}}

\newcommand{\gp}{\mathrm{gp}}

\newcommand{\Htrgp}{\mathbf{H}^{\tr,\gp}}

\newcommand{\calS}{\mathcal S}
\newcommand{\sSet}{\mathrm{sSet}}
\newcommand{\triv}{\mathrm{triv}}
\newcommand{\Map}{\mathrm{Map}}

\newcommand{\cofib}{\operatorname{cofib}}
\newcommand{\fib}{\operatorname{fib}}

\renewcommand{\thesubsection}{\arabic{subsection}}

\newcommand{\Smat}{\mathrm{Sm}^\mathrm{cci}}

\newcommand{\Pretr}{\Pre^\tr}
\newcommand{\HHtr}{\HH^\tr}
\newcommand{\HHtrd}{\HH^\trd}

\newcommand{\Shv}{\mathrm{Shv}}

\newcommand{\wX}{\widetilde X}

\newcommand{\nuf}{\nu\mathrm{f}}

\newcommand{\bbS}{\mathbf S}
\newcommand{\lSz}{{S_{(z)}}}

\begin{document}

\title{Zariski-local framed $\mathbb{A}^1$-homotopy theory} %framed
% \thanks{The research is supported by the Russian Since Foundation grant 20-41-04401, except \Cref{th:tautfLoc:istartjsharp}. The second author is 
% supported by SFB 1085 Higher Invariants}

\author[A. Druzhinin]{Andrei Druzhinin}
\address{Andrei Druzhinin, 
St. Petersburg Department of Steklov Mathematical Institute of Russian Academy of Sciences,
Fontanka, 27, 191023 Saint Petersburg, Russia;
Chebyshev Laboratory, St. Petersburg State University, 
14th Line V.O., 29, Saint Petersburg 199178 Russia.
}
\email{}{andrei.druzh@gmail.com, adruzhinin@pdmi.ras.ru, a.druzhinin@spbu.ru}
% \email{\href{mailto:andrei.druzh@gmail.com}}{andrei.druzh@gmail.com}

\author[V. Sosnilo]{Vladimir Sosnilo}
\address{
Vladimir Sosnilo, M309, 
Universit{\"a}t Regensburg, Universit{\"a}tsstra{\ss}e 31, 93053 Regensburg, Germany
}
\email{}{vsosnilo@gmail.com}
% \email{\href{mailto:vsosnilo@gmail.com}}{vsosnilo@gmail.com}

\subjclass{14F42}
\keywords{stable motivic homotopy category, framed correspondences, Zariski topology, localisation theorem}
\maketitle

\begin{abstract}
% , when $S$ is a separated noetherian scheme of finite dimension
For any (not necessarily perfect) field $k$ we obtain equivalences of $\infty$-categories
\[\mathbf{H}^{\mathrm{fr},\mathrm{gp}}(k)\simeq \mathbf{H}^{\mathrm{fr},\mathrm{gp}}_{\mathrm{zf}}(k)
\text{  and  }
\mathbf{DM}(k)\simeq\mathbf{DM}_{\mathrm{zar}}(k).\]
We also construct an equivalence of $\infty$-categories 
\[
\mathbf{H}^{\mathrm{fr},\mathrm{gp}}(S) \simeq \mathbf{H}^{\mathrm{fr},\mathrm{gp}}_{\mathrm{zf}}(S)
\]
of group-like framed motivic spaces %spectra
over a separated noetherian scheme $S$ of finite Krull dimension
with respect to the Nisnevich topology at one side and the Zariski fibre topology $\mathrm{zf}$
generated by the Zariski one and the trivial fibre topology (introduced by Druzhinin, Kolderup and 
{\O}stv{\ae}r) on the other side. Over a field, the Zariski fibre topology equals the Zariski topology and the 
result follows from the previous one. To prove it in the case of a general base scheme, we prove a 
localisation theorem for $\mathbf{H}^{\mathrm{fr},\mathrm{gp}}_{\mathrm{zf}}(-)$
%Namely, we use 
%the localisation theorem for $\mathbf{H}^{\mathrm{fr},\mathrm{gp}}(S)$ was proved by 
%Hoyois, 
%and 
employing the ideas from the proof of the {\it affine localisation theorem} for the trivial fibre 
topology by the first author, Kolderup and {\O}stv{\ae}r. % to obtain the the localisation theorem for 
%$\mathbf{H}^{\mathrm{fr},\mathrm{gp}}_{\mathrm{zf}}(S)$. 
%{\O}stw{\ae}r
\end{abstract}

\section*{Introduction}

It is known that 
the Nisnevich topology on the category of smooth schemes is 
%a 
very natural in the context of motivic homotopy theory. % for the topology on the category of schemes 
In particular, the constructions of 
% of motivic homotopy categories, % such as 
Voevodsky motives \cite{Voe-motives,Voe-cancel,MVW,Cisinski-Deglise},
unstable and stable Morel-Voevodsky's motivic homotopy categories \cite{Morel-Voevodsky,Jardine-spt,morel-trieste,mot-functors},
and framed motivic categories \cite{ehksy,BigFrmotives,FramedGamma} are all based on 
Nisnevich sheaves of spaces. 
In this note,  
we show that the Zariski topology over a field and a modification of it %denoted $\zf$
over finite-dimensional separated noetherian schemes called \emph{the Zariski fibre topology}
lead to the same $\infty$-categories of 
Voevodsky motives $\mathbf{DM}(S)$ and 
of the framed motivic spectra $\SH^\fr(S)$ \cite{ehksy}. % via the usual construction 
Consequently, this allows us to improve the reconstruction theorem from \cite{Hoyois-framed-loc} that states $\SH(S)\simeq \SH^\fr(S)$, 
and write $\SH(S)\simeq \SH^\fr_\zf(S)$.

Regarding the 
categories %$\mathbf{DM}(k)$, $\mathbf{DM}_\zar(k)$ 
of Voevodsky motives over a field,
% defined with respect to the Nisnevich and the Zariski topologies
we prove 
that for any field $k$ the obvious fully faithful embedding is an equivalence
\begin{equation}\label{eq:DMniskDMzark}\mathbf{DM}(k)\simeq \mathbf{DM}_\zar(k).\end{equation}
% \begin{remark}
% \begin{remark}
As shown in \cite[Theorem 5.7]{VoevA1invCorprethories}
the strict homotopy invariance \cite[Theorem 5.6]{VoevA1invCorprethories}
and the injectivity \cite[Cor. 4.18, 4.19]{VoevA1invCorprethories} theorems 
imply 
the isomorphism \[H^n_\nis(X,F_\nis)\simeq H^n_\zar(X,F_\zar)\] 
for an $\A^1$-invariant presheaf of abelian groups with transfers. 
While the latter isomorphism implies the categorical equivalence \eqref{eq:DMniskDMzark},
and moreover, conversely, the categorical equivalence \eqref{eq:DMniskDMzark} together with the strict 
homotopy invariance theorem imply the isomorphism, 
the known proof of \cite[Theorem 5.6]{VoevA1invCorprethories} uses perfectness of $k$. 
%\end{remark}
% It is known that the equivalence \eqref{eq:DMniskDMzark} is a consequence of the injectivity on essentially smooth local schemes theorem \cite[Cor. 4.18, 4.19]{VoevA1invCorprethories}, and the strict homotopy invariance theorem provided by \cite[Theorem 5.6]{VoevA1invCorprethories} for perfect $k$.
%
%In these notes, the equivalence \eqref{eq:DMniskDMzark} is proved for any $k$ without use of the strict homotopy invariance theorem.
Our proof of \eqref{eq:DMniskDMzark} does not use the assumption; in fact our proof is self-contained apart 
from the only external ingredient -- the \'etale excision theorem, 
which is provided by \cite{VoevA1invCorprethories} for all fields $k$.

\subsubsection{} 
We expect that the direct generalization of the equivalence \eqref{eq:DMniskDMzark} to positive-dimensional 
base schemes does not hold, 
because $\mathbf{DM}_\zar(S)$ most likely does not satisfy the {\it Localisation Theorem}, 
while $\mathbf{DM}(S)$ does (see \cite{Cisinski-Deglise}). 
The Zariski fibre topology $\zf$ on $\Sm_S$ is defined to fix this problem of the Zariski topology. 
It coincides with the Zariski topology over fields and satisfies the Localisation Theorem over general bases. 
Concretely, this topology is generated by the Zariski topology and the trivial fibre topology introduced in 
\cite[Definition 3.1]{DKO:SHISpecZ} by the first author, Kolderup and {\O}stv{\ae}r. 
The latter is the topology generated by the Nisnevich covers coming from the Nisnevich squares of the form
\[%\label{eq:XprimeXtimesS-ZtoS}
\xymatrix{X^\prime\times_S(S-Z)\ar[d] \ar[r] & X^\prime\ar[d]\\
X\times_S(S-Z)\ar[r]& X
}
\]
where $Z$ is a closed subscheme of the base.

\begin{remm}%\label{rem:zf}
Zariski fibre topology is the strongest subtopology of the Nisnevich topology on $\Sch_S$ that equals the Zariski topology over the residue fields.
It is expected that the Zariski fibre topology is the weakest topology that contains Zariski topology and satisfies the Localisation Theorem for categories $\mathbf{DM}(S)$ or $\SH(S)$.
%, or \cite[Theorem 2.21]{Morel-Voevodsky}
\end{remm}

\subsubsection{}
Recall that Voevodsky motives $\mathbf{DM}_\tau(S)$, motivic spectra $\SH_\tau(S)$ and framed motivic spectra $\SH^{\mathrm{fr}}_\tau(S)$ can all be defined via a general 
construction of
taking the $\infty$-category of $\mathbb{P}^1$-spectra in $\HHtrd_\tau(S)$ -- the 
$\infty$-category of $\tau$-sheaves on a given {\it $\infty$-category of 
correspondences} $\mathrm{Corr}_S$ (see Definition~\ref{def:ShvtauHHtauA1}). 
In the case of $\mathbf{DM}(S)$ one takes $\mathrm{Corr}_S$ to be the category of finite correspondences \cite{Voe-motives} and in 
the case of motivic spectra one takes $\mathrm{Corr}_S = \mathrm{Sm}_S$. 
In the case of framed motivic spectra one takes the $\infty$-category of framed correspondences
\footnote{Framed correspondences in their original form were introduced in the unpublished notes by 
Voevodsky \cite{VoevNotesFrCor}, and the instrument was deeply studied and developed in the Garkusha-Panin theory of framed motives \cite{GP14,GP-HIVth,CancellationFrAGP,ConeTheGNP,Nesh-FrKMW,framed-MW,BigFrmotives,FramedGamma,DruzhPanin-SurjEtEx,DrKyl,DRUZHININ2022106834,SmModelSpectrumTP,Nesh:nfr_GN}} 
\cite{ehksy,framed-notes,epiga:7494}. 

We employ an axiomatic treatment of the localisation theorem, defining it as the property of a topology in the 
context of a given family of categories of correspondences (see Definition~\ref{def:correspondencesCorrS} and 
Definition~\ref{def:LocLocAff}). This allows to apply our result both in the setting of Voevodsky motives and 
motivic spectra. 

\begin{example}
The localisation property for the Nisnevich topology in the context of the category of framed correspondences 
is proved in \cite{Hoyois-framed-loc}, furthermore, it follows that it holds when 
$\mathrm{Corr}_S$ is taken to be $\mathrm{Sm}_S$ from the works of Ayoub \cite{Morel-Voevodsky,AyoubI}. In the 
context of finite correspondences the property follows from Cisinski-D\'{e}glise's work 
\cite{Cisinski-Deglise}. In \cite{DKO:SHISpecZ} the localisation property for the trivial fibre topology was proved, 
both in the context of framed correspondences and when $\mathrm{Corr}_S = \mathrm{Sm}_S$. 
\end{example}
In the present paper, based on the ideas of \cite{DKO:SHISpecZ},
we show that the localisation theorem also holds for the Zariski fibre topology. 
More generally, we prove the 
following:

\begin{thmX}[{Theorem~\ref{th:tautfLoc}}, and \ref{th:tautfLoc:istartjsharp}]%\label{th:tautfLoc}
Let $i\colon Z\to S$ be a closed immersion of affine schemes.
Let $\tau$ be a topology on $\Sch_S$ such that $\tau\supset \tf$ and satisfies the additional technical 
assumption: $\wtau=\stau$ on $\SmAff_{S,Z}$ (see \Cref{def:wtaustau}).

Denote by
$\Htr_{\Sigma,\tau}(\SmAff_{S})$ one of $\HH_{\Sigma,\tau}(\SmAff_{S})$, $\Hfr_{\Sigma,\tau}(\SmAff_{S})$.
Denote by 
$\HHtrd_{\Sigma,\tau}(\SmAff_{S})$ one of 
$\HHd_{\Sigma,\tau}(\SmAff_{S})$, $\HH^\mathrm{fr}_{\Sigma,\tau}(\SmAff_{S})$.
Then we have adjunctions $i^!\vdash i_*$, $j_*\vdash j^*$, $i^*\dashv i_*$, $j_\#\dashv j^*$
\[
\begin{tikzcd}
\HHtrd_{\Sigma,\tau}(\SmAff_{Z})\arrow[r, shift left, "i_*"] & \HHtrd_{\Sigma,\tau}(\SmAff_{S})\arrow[r, shift left, "j^*"]\arrow[l, shift left, "i^!"] & \HHtrd_{\Sigma,\tau}(\SmAff_{S-Z}),\arrow[l, shift left, "j_*"]
\end{tikzcd}
\]
\[
\begin{tikzcd}
\Htr_{\Sigma,\tau}(\SmAff_{Z})\arrow[r, shift right, swap, "i_*"] & \Htr_{\Sigma,\tau}(\SmAff_{S})\arrow[r, shift right, swap, "j^*"]\arrow[l, shift right, swap, "i^*"] & \Htr_{\Sigma,\tau}(\SmAff_{S-Z})\arrow[l, shift right, swap, "j_\sharp"]
\end{tikzcd}
\]
% by $i_*\dashv i^!$, $j^*\dashv j_*$, and the one for $i^*\dashv i^*$, $j_\#\dashv j^*$.
inducing a pullback square and, respectively, a pushout square
\[\xymatrix{
i_* i^! F\ar[r]\ar[d] & F\ar[d]\\
{*}\ar[r] & j_* j^* F,
}
\quad
\quad\quad
\xymatrix{
j_\sharp j^* G \ar[r]\ar[d] & G\ar[d]\\
{*}\ar[r] & i_* i^* G
}\]
for any 
$F\in \HHtrd_{\Sigma,\tau}(\SmAff_{S})$ and for any 
$G\in \Htr_{\Sigma,\tau}(\SmAff_{S})$.
%In other words, $\tau$ satisifes the localisation property in the sense of \Cref{def:LocLocAff}.
In particular, the claim holds for $\tau$ being the Zariski fibire topology $\zf=\zar\cup\tf$, as well as for the Nisnevich topology and the trivial fibre topology $\tf$.
\end{thmX}
% \begin{thmX}[{Theorem~\ref{th:tautfLoc}}]%\label{th:tautfLoc}
% Let $i\colon Z\to S$ be a closed immersion of affine schemes.
% Let $\tau$ be a topology on $\Sch_S$ such that $\tau\supset \tf$ and satisfying the additional technical 
% assumption: $\wtau=\stau$ on $\SmAff_{S,Z}$ (see \Cref{def:wtaustau}).

% Consider the pair of adjunctions
% \[
% \begin{array}{lclcl}
% \Hfr_{\tau}(\SmAff_{S})&\rightleftarrows& \Hfr_{\tau}(\SmAff_{S})&\rightleftarrows&\Hfr_{\tau}(\SmAff_{Z})\\
% \HH_{\tau}(\SmAff_{S})&\rightleftarrows& \HH_{\tau}(\SmAff_{S})&\rightleftarrows&\HH_{\tau}(\SmAff_{Z})
% \end{array}
% \]
% given by $i_*\dashv i^!$, $j^*\dashv j_*$.
% Then for any $F\in \Hfr_{\tau}(\SmAff_{S})$, or $F\in \HH_{\tau}(\SmAff_{S})$, there is a pullback square
% \[\xymatrix{
% i_* i^! F\ar[r]\ar[d] & F\ar[d]\\
% {*}\ar[r] & j_* j^* F.
% }\]
% %In other words, $\tau$ satisifes the localisation property in the sense of \Cref{def:LocLocAff}.
% In particular, the claim holds for $\tau$ being the Zariski fibire topology $\zf=\zar\cup\tf$, as well as for the Nisnevich topology and the trivial fibre topology $\tf$.
% \end{thmX}
%The argument is written in such a general form that provides the result for categories $\HH(\SmAff_S)$ as well.

As discussed above, using the \'etale excision theorem from \cite{GP-HIVth,DruzhPanin-SurjEtEx,DrKyl},
we prove the equivalence
\begin{equation*}
\mathbf{H}^{\fr,\gp}_{\nis}(k)\simeq \mathbf{H}^{\fr,\gp}_{\zar}(k).
\end{equation*}
%We combine the mentioned argument for $S=\Spec k$ for a field $k$based on
Combining this with the above localisation result, we obtain the following 
generalization to the relative setting. %Localisation Theorem presented above.

\begin{thmX}\label{th:SHtrzarSHtrnis}
For any noetherian separated scheme $S$ of finite Krull dimension,
there is a canonical equivalence of categories
\begin{equation*}
\mathbf{H}^{\fr,\gp}_{\nis}(S)\simeq \mathbf{H}^{\fr,\gp}_{\zf}(S).
\end{equation*}
\end{thmX}
\begin{remark}\label{ex:CorKCorGWCorMWwidetildeCorCorrfr(k)}
It is likely true that the above theorem 
holds as well in the context of other families of $\infty$-categories of correspondences, although 
we do not discuss it here. Some interesting examples to consider include
% for all preadditive $\infty$-category of correspondences equipped with the functor $\Fr_+(k)\to \Corr_k$, 
% from the framed category of framed correspondences $\Fr_+(k)$ defined in \cite{GP14}
%such as 
$\mathrm{K}$-correspondences, $\mathrm{GW}$-correspondences \cite{Walker,K-motives,DruDMGWeff}, 
%(over fields), 
Milnor-Witt correspondences \cite{Chow-W-correspondences,MW-cplx,MW-ring-spt,MWmotives} 
%appropriately defined over non-perfect fields
and finite $A$-correspondences in the sense of \cite{ccorrs, five-transfers}. 

\end{remark}

\subsection*{Structure of the text} % and proof

%\Cref{sect:Corrandtau_S}
% we introduce a formalism of $\infty$-categories of correspondences and respective motivic categories over base schemes $S\in \Sch$,
% and prove some general properties.
%
We start in \Cref{subsect:SchCorr_S} with 
a formalism of families of $\infty$-categories of correspondences $\Corr_S$ 
and respective motivic $\infty$-categories 
\[\HH_{\Sigma,\tau}^\tr(S),\quad \SH_{\Sigma,\tau}^{S^1,\tr}(S),\quad \SH^{\tr}_{\Sigma,\tau}(S)\]
% \[\HH_\tau^\Corr(S),\quad \SH_\tau^{S^1,\Corr}(S),\quad \SH^{\Corr}_\tau(S)\]
over base schemes $S\in \Sch$. 
In the rest of \Cref{sect:Corrandtau_S} we discuss various properties of families of
$\infty$-categories of correspondences and their consequences. 
In \Cref{sect:zf}, we discuss general fibre topologies, that include the Zariski fibre topology $\zf$, Nisnevich topology and the trivial fibre topologies as examples. 
In \Cref{section:LocalisationTheoremtf}, we prove the localisation theorem for $\HH_{\Sigma,\tau}^\tr(S)$ for fibre topologies. %, covering the examples provided by $\nis$, $\tf$, and $\zf$. %=\zar\cup\tf
In \Cref{sect:pointexcisive(pre)sheaf} we discuss sheaves that satisfy the property of being {\it excisive} 
with respect to varying topologies. In \Cref{lm:NexZpoints} we consider a pair of topologies $Z\subset N$, 
such that $Z$ have enough set of points, and $N$ is completely decomposable. 
We prove that a $Z$-sheaf is an $N$-sheaf whenever it is excisive with respect to $N$-squares on $Z$-points.
In \Cref{subsect:EtExFrCorr} we discuss the example given by the \'etale excision for $\mathbb{A}^1$-invariant 
framed presheaves. 
Finally, in \Cref{subsect:EtExLocA1niszar} we prove \Cref{th:SHtrzarSHtrnis} over fields and in \Cref{subsect:ResultingSummary} we deduce the result over base schemes $S\in \Schfns$. 

\begin{notations}\label{den:introductionnotation}
Throughout the paper we make use of the following conventions and notation.

%\begin{itemize}
\begin{enumerate}
\item
$\Sch$ is the category of %qcqs 
schemes, and 
$\Aff$ is the category of 
affine schemes.
\item
$\Schfns$, and $\Afffns$ 
are the categories of 
noetherian separated schemes
of finite Krull dimension, and 
affine noetherian schemes
of finite Krull dimension.
% $\Schns$ and $\Affns$ are the subcategories of noetherian separated schemes, and affine noetherian schemes.
% $\Schfns$ and $\Afffns$ are the subcategories of schemes of finite Krull dimension.

% \item
% $\Sch_S$ and $\Aff_S$ are the categories of schemes in $\Sch$, and $\Aff$
% over a given noetherian separated scheme $S\in \Sch$, or $S\in\Aff$.
% Writing symbols $\Sch_S$, or $\Aff_S$, we always suppose that $S$ is noetherian separated.
\item
$\Sch_S$ and $\Aff_S$ are the categories of %qcqs 
schemes % in $\Sch$, and $\Aff$
over a given scheme $S$. %$S\in \Sch$, or $S\in\Aff$.
% % Writing symbols $\Sch_S$, or $\Aff_S$, we always suppose that $S$ is noetherian separated.

% \item
% $\Sch_S$ and $\Aff_S$ are the categories of schemes in $\Sch$, and $\Aff$
% over a given noetherian separated scheme $S\in \Sch$, or $S\in\Aff$.
% Writing symbols $\Sch_S$, or $\Aff_S$, we always suppose that $S$ is noetherian separated.
% \todo{commented alternative text is here}
% \item
% $\Sch$ is the category noetherian separated schemes and $\Aff$ is the category of affine schemes.
% $\Sch_S$ and $\Aff_S$ are the categories of arrows over $S$ in $\Sch$, and $\Aff$ for a given $S\in \Sch$, or $S\in\Aff$.
% Speaking about limits of diagrams in $\Sch$, or $\Sch_S$, %of schemes or $S$-schemes
% we always mean the diagrams that have limits in the category $\Sch$, or $\Sch_S$.
\item
$\Sm_S$ and $\SmAff_S$ are the full subcategories of $\Sch_S$ and $\Aff_S$ spanned by smooth $S$-schemes;
% For $S\in \Sch$ we denote by 
$\Smat_{S}$ is the subcategory of $\SmAff_{S}$ 
spanned by schemes whose tangent bundle is stably trivial.
\item 
For a scheme $S\in\Sch$, 
we identify a point $z\in S$ with 
% the $0$-dimensional scheme given by 
the spectrum of its residue field. 
We write $\lSz=\Spec\mathcal O_{S,z}$, and $S^h_z=\Spec\mathcal O_{S,z}^h$. 
% \todo{commented alternative text is here}
% \item
% $\EssSm_S$ is the subcategory in $\Sch_S$ 
% of essentially smooth schemes over $S$,
% namely, 
% filtered limits in the category of qcqs schemes %$\Sch_S$ 
% of diagrams in $\Sm_S$ 
% with \'etale affine transition maps
% that belong to $\Sch_S$.
% For a functor $F$ on $\Sm_S$, denote by the same symbol 
% the continuous functor on $\EssSm_S$ 
% given by $F(\varprojlim_\alpha X_\alpha)\cong \varinjlim_\alpha F(X_\alpha)$.
\item
For any $X\in \Sch_S$, $z\in S$, and $S^\prime\to S$,
we write $X_z=X\times_S z$, and $X_{S^\prime}=X\times_S S^\prime$.
\item
$\EssSm_S$ is the category of essentially smooth schemes over $S$,
namely, filtered limits in $\Sch_S$ of diagrams in $\Sm_S$ 
with \'etale affine transition maps.
For a functor $F$ on $\Sm_S$, denote by the same symbol 
the continuous functor on $\EssSm_S$ 
given by $F(\varprojlim_\alpha X_\alpha)\cong \varinjlim_\alpha F(X_\alpha)$.
\item
For any presheaf on the small Zariski site over a scheme $X$, 
and a closed immersion $i\colon Z \hookrightarrow X$, 
$F(X_{(Z)})$ equals the global sections of the presheaf $i^{-1}F$ on $Z$. 
%TODO add: In other words it's the colimit of $F(U)$ over all open subschemes of $X$ containing $Z$. 
% If $W$ is 
% an essentially smooth scheme obtained as 
% a localisation of a smooth scheme $X$ at 
% a closed subset $Z$, 
% which are clear from the context, 
% we write $F(W)$ for $F(X_Z)$.

\item 
Staring from \Cref{def:correspondencesCorrS},
we write $\mathcal S_{(-)}$ for $\mathcal S_\bbS$, when $\bbS\subset\Sch$ is clear from the context.

\item
We use the standard language of $\infty$-categories following
\cite{Lurie}. 
We denote by $\mathrm{Cat}^{\operatorname{padd}}_{\infty}$ the $\infty$-categories of small preadditive $\infty$-categories and $\mathrm{Cat}^{\operatorname{add}}_{\infty}$ is its subcategory consisting of additive $\infty$-categories. 
We denote by $(-)^\gp$ the left adjoint to the obvious embedding functor
$\mathrm{Cat}^{\operatorname{add}}_{\infty}\to \mathrm{Cat}^{\operatorname{padd}}_{\infty}$.
% % takes a preadditive $\infty$-category $C$ 
% % to the additive $\infty$-category $C^\gp$ whose mapping spaces are group completions of the corresponding $\mathbb{E}_{\infty}$-monoids that are mapping spaces of $C$, see \Cref{def:Prepaddadd}.

%$\mathrm{Cat}^{\operatorname{padd}}_{\infty}$, $\mathrm{Cat}^{\operatorname{add}}_{\infty}$ are the $\infty$-categories of preadditive and additive  $\infty$-categories.
% For $\mathcal C\in\mathrm{Cat}^{\operatorname{padd}}_{\infty}$, 
% $\mathcal C^\gp\in \mathrm{Cat}^{\operatorname{add}}_{\infty}$ is
% the the associated additive $\infty$-category.

\item
$\SPre(\mathcal S)$ and $\SPred(\mathcal S)$ 
are the $\infty$-categories of presheaves of spaces or pointed spaces 
on %the $\infty$-category 
$\mathcal S$. 
$\SPrerad(\mathcal S)$ and $\SPreradd(\mathcal S)$ are 
the subcategories of radditive presheaves. %, whenever $\mathcal S$ has coproducts

\item
% Let $S$ be a scheme.
Given an $\infty$-category denoted by %$\Corr_S$ or 
$\Corr(\mathcal S)$, %the base scheme,
% and equipped with a functor
% $\Sm_S\to \Corr_S$, 
% we use the following notation:
we write 
\[
\SPretr(\mathcal S)=\SPre(\Corr(\mathcal S)),\quad 
\SPretrd(\mathcal S)=\SPred(\Corr(\mathcal S))
\]
when there is no confusion.
%where the latter one stands for pointed presheaves.
%
% $\SPretr(\mathcal S_*)=\SPre(\Corr(\mathcal S_*))$, 
% and
% $\SPretrd(\mathcal S_*)=\SPred(\Corr(\mathcal S_*))$.
%
%
%
\item
We denote by $h^\tr_*(X)\in\SPre^\tr(\mathcal S_*)$ 
the presheaf represented by $X\in \Corr(\calS_*)$.
We write $h^\tr(X)$ for $h^\tr_*(X)$, when the base scheme is clear from the context.
\item $h^\trd_*(X)=h^\tr_*(X)_+\in\SPretrd(\mathcal S_*)$.
% We 
% denote by $h^\tr_S(X)\in\SPre^\tr(\mathcal S_*)$ the presheaf represented by $X\in \Corr(\calS_*)$,
% and $h^\trd_S(X)=h^\tr_S(X)_+\in\SPre^\trd(\mathcal S_*)$.
%and the associated sheaf $h^\tr_S(X)\in\Shv_\tau^\tr(\mathcal S_*)$.
% We write $h^\tr(X)$ for $h^\tr_*(X)$, when the base scheme $S$ is defined by the context.
%Denote by $h^\tr(X)$ the presheaf represented by $X\in \Corr_S$,
%and denote by $h^\tr_\tau(X)$ the sheafification with respect to a topology $\tau$.\todo{added}

\item
In \Cref{subsect:SZA1Z}, we use notation 
$\Pre^{\A\tr}(\mathcal S_*)=\SPre(\CorrAtr(\mathcal S_*))$ in sense of \Cref{def:CorrA}.
% $\Pre^{\A\trd}(\mathcal S_*)=\SPred(\CorrAtr(\mathcal S_*))$ in sense of \Cref{def:CorrA}.
%$\Pre^{\A\tr}(\mathcal S_*)=\Pre(\Corr^\A(\mathcal S_*))$.

% \item
% Let $S$ be a scheme.
% Given a category denoted by $\Corr_S$ or $\Corr(S)$ and equipped with a functor
% $\Sm_S\to \Corr_S$, 
% % that induces isomorphism 
% % $\pi_0(\Sm_S)\cong \pi_0(\Corr_S)$,
% we use the following notation:
% $\SPretr(S)=\SPre(\Corr_S)$,
% $\SPretrrad(S)$ is the subcategory of radditive presheaves in $\SPretr(S)$.
% % and for a topology $\tau$ on $\Sm_S$.
% % $\Sh^\tr_\tau(S)$ is the subcategory of $\tau$-local objects.
% % \item
% % Given a family of categories $\Corr(S)$,
% % $\SPretr(S)$ is the $\infty$-category of additive presheaves of spaces on $\Corr_S$,
% % $\SPretr(\SmAff_S)$ is the $\infty$-category of additive presheaves of spaces on $\mathrm{Corr}(\SmAff_S)$.
% Denote by $h^\tr(X)$ the presheaf represented by $X\in \Sm_S$.
% % the subcategory of $\mathrm{Corr}^{\mathrm{fr}}_S$ spanned by affine schemes
% % \item 

\item 
% a category denoted by $\Corr_S$ or $\Corr(S)$
% equipped with 
% Given a subcategory $\mathcal S$ in $\Sch_S$ over the base scheme $S$ closed with respect to the endofunctor $-\times\A^1$, and an essentially surjective functor $\Sm_S\to \Corr_S$,
% the subcategory of $\SPre^\tr_\Sigma(S)$ spanned by 
% $\A^1$-invariant objects is denoted by $\HH^\tr(S)$; 
% the subcategory of $\nu$-sheaves in $\SPre^\tr_\Sigma(S)$, for a topology $\nu$ on $\Sm_S$,
% is denoted by $\Shv^\tr_\nu(S)$;
% $\HH^\tr_\nu(S)=\HH^\tr(S)\cap\Shv^\tr_\nu(S)$.
% The latter two subcategories and their intersection are reflective,
% we denote the corresponding localisation functors by 
% $L_{\A^1}$, $L_{\nu}$, and $L_{\A^1,\nu}$ respectively.

% Given a subcategory $\mathcal S$ in $\Sch_S$ over the base scheme $S$ closed with respect to the endofunctor $-\times\A^1$, and an essentially surjective functor $\Sm_S\to \Corr_S$,
% the subcategory of $\SPre^\tr(S)$ spanned by 
% $\A^1$-invariant objects is denoted by $\HH^\tr_\triv(S)$.

% The subcategory of $\SPre^\tr(-)$ or $\SPre^\trd(-)$
% spanned by 
% $\A^1$-invariant objects is 
% denoted by $\HH^\tr_\triv(-)$ or $\HH^\trd_\triv(-)$.

$\HH^\tr_\triv(-)$ or $\HH^\trd_\triv(-)$
denotes 
the subcategory of 
$\SPre^\tr(-)$ or $\SPre^\trd(-)$
spanned by 
$\A^1$-invariant objects.

\item
Throughout the text we consider subcanonical topologies $\tau$ on categories $\mathcal S_*$; 
$\Shv^\tr_\tau(\mathcal S_*)=\Shv_\tau(\Corr(\mathcal S_*))$,
that is the subcategory in $\SPre^\tr(\mathcal S_*)$ 
spanned by 
the objects that go to 
$\tau$-sheaves in $\SPre(\mathcal S_*)$,
and similarly for $\Shv^\trd_\tau(\mathcal S_*)$.
%Similarly we define the subcategory $\Shv^\trd_\tau(\mathcal S_*)$ in $\SPre^\trd(\mathcal S_*)$.

\item
%Denote 
$\HH^\tr_\tau(-)=\HH^\tr_\triv(-)\cap\Shv^\tr_\tau(-)$,
and similarly for $\HH^\trd_\tau(-)$.
%The latter two subcategories and their intersection are reflective,
%We define the $\infty$-category $\HH^\trd_\tau(-)$ in a similar way.
\item 
We write 
$\SH^{S^1,\tr}_\tau(-) = \SH^{S^1}_\tau(\Corr(-))$,
$\SH^{\tr}_\tau(-) = \SH_\tau(\Corr(-))$.

\item 
We denote the corresponding localisation functors 
$L_{\A^1}\colon \SPre^\tr(-)\to \HH^\tr_\triv(-)$, 
$L_{\tau}\colon \SPre^\tr(-)\to \Shv^\tr_\tau(-)$, and 
$L_{\A^1,\tau}\colon \SPre^\tr(-)\to \HHtr_\tau(-)$,
and similarly for pointed categories.

\item 
We write $\SPre(S)$ for $\SPre(\Sm_S)$, where $S\in\Sch$,
and write $\SPre(R)$ for $\SPre(\operatorname{Spec} R)$ for a ring $R$.
We use similar notation for 
$\Shtr_\nu(-)$, $\HHtr(-)$, $\HHtrd(-)$, $\SH^{S^1,\tr}_\tau(-)$, $\SH^{\tr}_\tau(-)$.
% $\Shtr_\nu(S)$, $\HHtr(S)$, $\HHtrd(S)$, $\SH^{S^1,\tr}_\tau(S)$, $\SH^{\tr}_\tau(S)$.

\item
We denote by $\tau\cup\nu$ the topology generated by topologies $\tau$ and $\nu$.

\item
%is equipped over $\mathbb{E}_{\infty}$-monoids, 
We write
$\SPre^{\tr,\gp}(\mathcal S) = \SPre(\Corr_*^\gp(\mathcal S))$,
when 
the $\infty$-category $\Corr(\mathcal S_*)$ as above is preadditive.
Note that any $F \in \SPretr_\Sigma(\mathcal S)$ admits a canonical structure of an $\mathbb{E}_{\infty}$-monoid;
so 
$\SPretrd_\Sigma(S)\simeq \SPretr_\Sigma(S)$,
and
$\SPre^{\tr,\gp}_\Sigma(\mathcal S)$
is equivalent to the subcategory of group-like objects in
$\SPre^\tr_\Sigma(\mathcal S)$
by \cite[Lemma 1.6]{UnivMultSp}.

\item
$\mathrm{Corr}^{\mathrm{fr}}_S=\mathrm{Corr}^\fr(S)$ 
is the $\infty$-category of framed correspondences over $S$ (constructed in \cite{ehksy}).
We write $\fr$ for $\tr$ in the above notation when $\Corr_S=\Corr^\fr_S$.
%In particular, $\SPrefr(S)=\SPre(\Corr^\fr_S)$. 
%$\SPrefrd(S)=\SPred(\Corr^\fr_S)$.
% Note that 
% $\SPrefr_\Sigma(S)\simeq \SPrefrd_\Sigma(S)$ since $\Corr^\fr$ is preadditive.
%$\SPrefrd(S)=\SPred(\Corr^\fr_S)$.

% \item
% $\SPrefr(S)$, $\SPrefr(\SmAff_S)$, $\SPrefrrad(S)$, $\Sh^\fr_\nu(S)$, $\SPre^{\fr,\gp}(S)$,
% $\mathbf{H}^{\fr}(S)$,
% $\mathbf{H}^{\fr,\gp}(S)$,
% $\SH^{S^1,\fr}(S)$,
% $\SH^{\fr}(S)$,
% denote
% % $\SPretr(S)$, $\SPretr(\SmAff_S)$, $\SPretrrad(S)$, $\Sh^\tr_\nu(S)$, $\SPre^{\tr,\gp}(S)$,
% % $\mathbf{H}^{\tr}(S)$,
% % $\mathbf{H}^{\tr,\gp}(S)$,
% % $\SH^{S^1,\tr}(S)$,
% % $\SH^{\tr}(S)$,
% the categories defined above for the case of $\mathrm{Corr}^{\mathrm{fr}}_S$.

% \item 
% For any $S\in\Sch$,
% we consider that functor 
% \[\sSet\to \Aff_S; K\mapsto K_S\] that is
% the left can extension of the functor $\Delta\to \Aff_S$ given by
% the cosimplicial object $\Delta^\bullet_S$.
% For a simplicial set $K$,
% we denote by $K_S\in\Aff_S$ 
% the image of $K$ 
% with respect to the latter functor.
%functor $\sSet\to \Aff_S$ that 
% the left can extension $\sSet\to \Aff_S$ of the functor $\Delta\to \Aff_S$ given by
% the cosimplicial object $\Delta^\bullet_S$.
%and denote by thesame symbol.
% $K_S\in\Delta\SmAff_S$ 
% the smooth simplicial scheme over $S$ 
% that is the geometric realisations of $K$ 
% with respect to the cosimplicial object $\Delta^\bullet_S$ in $\Sch_S$,
% and denote by thesame symbol.

\end{enumerate}
%\end{itemize}
\end{notations}

\subsection{Acknowledgement}
The article combines two parts of research with separate support: 
(1) one part is dedicated to   
the localisation theorem for 
motives with preadditive transfers %correspondences $\infty$-categories, 
% motivic $\infty$-categories 
% with respect to preadditive correspondences $\infty$-categories, 
and
the equivalence $\SH_{\zf}^\fr(S)\simeq \SH^\fr_\nis(S)$, 
see \Cref{th:tautfLoc}, \Cref{prop:relaitve:SHfrS1ZarSHfrSNis}; 
(2) the second part is dedicated to  
the localisation theorem for unpointed motivic homotopy categories with or without transfers, 
see \Cref{sect:Loc_SZSS-Z:istarjsharp}, \Cref{th:tautfLoc:istartjsharp} and \Cref{cor:tautfLoc}. 
% \todo{write or delete}

The research except \Cref{sect:Loc_SZSS-Z:istarjsharp} and \Cref{th:tautfLoc:istartjsharp} 
is supported by the Russian Science Foundation grant 20-41-04401 only. 
The first author is supported for \Cref{sect:Loc_SZSS-Z:istarjsharp} and \Cref{th:tautfLoc:istartjsharp}
by a Young Russian Mathematics award,
%T
the second author is 
supported by the SFB 1085 ``Higher Invariants''  
funded by the Deutsche Forschungsgesellschaft (DFG).
% \todo{separate SFB 1085 Higher Invariants from 20-41-04401, or not}

% is supported by the Russian Since Foundation grant 20-41-04401. 
% The first author is supported for \Cref{th:tautfLoc:istartjsharp}
% by a Young Russian Mathematics award, and
% %T
% the second author is 
% supported by SFB 1085 Higher Invariants.

\renewcommand{\thesubsection}{\thesection.\arabic{subsection}}

\section{Families of categories and topologies}\label{sect:Corrandtau_S}
\subsection{Categories of correspondences.}\label{subsect:SchCorr_S}

%We denote the $\infty$-category of preadditive or additive $\infty$-categories by $ \mathrm{Cat}^{\operatorname{padd}}_\infty$ or $\mathrm{Cat}^{\operatorname{add}}_\infty$.

\begin{definition}\label{def:preaddcorrespondencescategoryCorrS}
Let $S\in \Sch$,
$\mathcal S_S$ be a monoidal subcategory of $\Sch_S$ 
with respect to the monoidal structure $\times_S$.
An \emph{$\infty$-category of correspondences} $\Corr_{S}$
on $\mathcal S_S$ is
a %pointed %\todo{pointed addad} 
$\infty$-category $\Corr_S\in \mathrm{Cat}_\infty$
with the pointed object $\emptyset\in \Corr_S$, 
and a given object $\mathrm{pt}_S\in \Corr_S$,
and an action of the monoidal category $\mathcal S_S$,
\[\mathcal S_S\times \Corr_S \to \Corr_S; \quad
(X, T) \mapsto (X \times_S T),\]
such that the induced functor
\begin{equation}\label{eq:SmStoCorrS}
\mathrm{corr}_S\colon \mathcal S_S\to \Corr_S; \quad
X\mapsto X\times_S  \mathrm{pt}_S 
\end{equation}
is
essentially surjective.
% an essentially surjective functors
% \begin{equation}\label{eq:SmStoCorrS}r\colon \Sm_S\to \Corr_S, S\in \Sch\end{equation}
% and
% the endofunctor $-\times{\Gm}$ on $\Corr_S$
% with the natural equivalence
% $r(X)\times{\Gm}\simeq r(X\times{\Gm})$ for $X\in \Sm_S$,
% and 
% the natural morphism .
% a natural monoidal functor,
% \[\Sm_S\to \mathrm{Func}(\Corr_S, \Corr_S^\prime), \mathrm{pt}\colon *\to \Corr_S,\]
% where $*$ denotes the category with one object and trivial mapping space,
% such that the induced functor
% \begin{equation}\label{eq:SmStoCorrS}\Sm_S\to \Corr_S, S\in \Sch\end{equation}
% and such that $\pi_0(\Sm_S)\cong \pi_0(\Corr_S)$, for each $S$.
% (2) any presheaf on $\Corr_S$ is equipped canonically with the structure of $\mathbb{E}_{\infty}$-monoid.
% The latter condition equivalently claims that \eqref{eq:SmStoCorrS} passes through
% \[\Sm^{\mathbb{E}_{\infty}}_S\to \Corr_S, S\in \Sch,\]
% where $\Sm_S^{\mathbb{E}_{\infty}}$ denotes the universal $\infty$-category that mapping spaces 
% are naturally equivalent to the $\mathbb{E}_{\infty}$-monoids associated with mapping sets in $\Sm_S$.
% If the functor \eqref{eq:CorrSchCataddpadd} lands in $\mathrm{Cat}^{\operatorname{padd}}_\infty$ or $\mathrm{Cat}^{\operatorname{add}}_\infty$ the categories of correspondences are called additive or preadditive.
% For a family of preadditive categories of correspondences $\Corr_S$ we denote by $\Corr^\gp_S$ the corresponding additive one.
The correspondences $\Corr_{S}$ are called \emph{radditive}
if the presheaves $h^\fr_S(X)=\Corr_S(-,X)$ on $\Sm_S$ are radditive for all $X\in\Corr_S$. 
\end{definition}
\begin{example}
(0)
The category $\Sch_{S}$ for $S\in\Sch$ is a category of correspondences over $S\in\Sch$ itself.
%The family of the categories $\Sm_S$ as a categories of correspondences satisfies (Embed).
%
%\todo{do we need preadditivisation}
(1)
The $\infty$-category of framed correspondences $\Corr^\fr_{S}$ from \cite{ehksy}.
\end{example}

Denote by $-\times \Gm\colon \SPre(\Corr_S)\to \SPre(\Corr_S)$ the direct image
endofunctor % on $\SPre(\Corr_S)$ is given by $(-\times\Gm)^*$
induced by the endofunctor $-\times\Gm$ on $\Corr_S$.
%Furthermore, 
Then the endofunctor $-\times{\Gm}$ on $\Corr_S$ induces the endofunctors 
$\Omega_{\Gm}$ and $\Sigma_{\Gm}$ on $\Pre(\Corr_S)$
\begin{equation}\label{eq:OmegaGmSigmaGm}
\Omega_{\Gm}F(X)\simeq \fib(F(X\times\Gm)\to F(X\times\{1\}), \;\;
\Sigma_{\Gm}F = \cofib(F\times\{1\}\to F\times\Gm),
\end{equation}
where $ F\times\{1\}\simeq F$, 
%the functor $-\times \Gm$ on $\SPre(\Corr_S)$ is given by $(-\times\Gm)^*$,
and the right-side morphisms in \eqref{eq:OmegaGmSigmaGm} are induced by the embedding morphism $\{1\}\to\Gm$. 
%vie the left Ka\SPre($
Functors \eqref{eq:OmegaGmSigmaGm} 
agree with the $\Gm$-loop and $\Gm$-suspension on $\SPre(\Sm_S)$
in the sense of natural equivalences
\begin{equation}\label{eq:SigmaGmOmegaGmgamma}\Sigma_{\Gm} r^*\simeq r^* \Sigma_{\Gm}, \Omega_{\Gm} r_*\simeq r_* \Omega_{\Gm}.\end{equation}

\begin{definition}\label{def:ShvtauHHtauA1}
%Let $\mathcal S$ be a subcategory of the category $\Sch_S$ over $S\in \Sch$ .
Let
$\Corr_{S}$
be an $\infty$-category of correspondences on 
$\mathcal S_S$ over $S\in \Sch$ 
such that $\A^1_S\in\mathcal S_S$,
and let $\tau$ be a topology on $\mathcal S_S$.
A presheaf $F\in \Pre(\Corr_S)$ is called 
\emph{$\tau$-sheaf}, for a topology $\tau$ on $\mathcal S_S$ (resp. \emph{$\A^1$-invariant}),
if it goes to the object of such type along the forgetful functor
$\Pre(\Corr_S)\to \Pre(\mathcal S_S)$.
%
% In particular, $F\in \SPre(\Corr_S)$,
% is $\A^1$-invariant, if for any $X\in \Sm_S$ the map $F(X)\to F(\A^1\times X)$ is an equivalence. 
%
%Denote by $\Shv_\tau(\Corr_S)$ the subcategory of $\Pre(\Corr_S)$ spanned by the $\tau$-sheaves.
Define $\HH^\tr_{\Sigma,\tau}(\calS_S)=\HH_{\Sigma,\tau}(\Corr_S)$ as 
the subcategory spanned by $\A^1$-invariant $\tau$-sheaves in $\Pre_{\Sigma}(\Corr_S)$,
and $\HH^\trd_{\Sigma,\tau}(\calS_S)=\HHd_{\Sigma,\tau}(\Corr_S)$ for pointed ones,
\[\SH^{S^1}_{\Sigma,\tau}(\Corr_S)=\HH_{\Sigma,\tau}(\Corr_S)[(S^1)^{\wedge -1}], \;\;
\SH_{\Sigma,\tau}(\Corr_S)=\SH^{S^1}_{\Sigma,\tau}(\Corr_S)[\Gm^{\wedge -1}].\]
\end{definition}

% Given a preadditive family of $\infty$-categories of correspondences $\Corr_{(-)}$,
% a presheaf $F\in \Pre(\Corr_S)$ for $S\in \Sch$ is called $\tau$-sheaf, for a topology $\tau$ on $\Sm_S$ (resp.  $\A^1$-invariant)
% if and only if it goes to the object of such type along the forgetful functor
% $\Pre(\Corr_S)\to \Pre(\Sm_S)$.
%
% Denote 
% $\HH^\tr_\tau(S)=\HH^\tr_\tau(\Sm_S)=\HH_\tau(\Corr_S)$, 
% and similarly for 
% $\SH^{S^1,\tr}_\tau(S)$, 
% and $\SH^\tr_\tau(S)$.
%
% $\SH^{S^1,\tr}_\tau(S)=\SH^{S^1,\tr}_\tau(\Sm_S)=\SH^{S^1}_\tau(\Corr_S)$, 
% $\SH^\tr_\tau(S)=\SH^\tr_\tau(\Sm_S)=\SH_\tau(\Corr_S)$.
%
%For a functor
%Similarly we define the $\infty$-categories
% $\HH^\tr(\SmAff_S)=\HH(\Corr(\SmAff_S))$, $\SH^{S^1,\tr}(\SmAff_S)=\SH^{S^1}(\Corr(\SmAff_S))$, $\SH^\tr(\SmAff_S)=\SH(\Corr(\SmAff_S))$
% with respect to the $\infty$-categories $\SmAff_S$ and their images $\Corr(\SmAff_S)$ in $\Corr_S$ instead of $\Sm_S$ and $\Corr_S$.

A family of $\infty$-categories over a subcategory $\bbS$ in the category $\Sch$,
is a functor $\bbS\to \mathrm{Cat}_\infty$.
Denote by $\Sm_{\bbS}$ the contravariant functor 
\[\Sm_{\bbS}\colon \bbS\to \mathrm{Cat}; S\mapsto \Sm_S.\]
% For a subcategory $\bbS$ of $\Sch$, 
% we denote by $\Sm_\bbS$ the restriction of $\Sm_{\Sch}$ on $\bbS$,
% We write $\Sm_{(-)}$ for $\Sm_{\bbS}$ sometimes.
% Denote by $\Sm_{\Sch}$ the contravariant functor 
% \[\Sm_{\Sch}\colon \Sch\to \mathrm{Cat}; S\mapsto \Sm_S\]
% that takes a morphism $S^\prime\to S$ to the base change functor $\Sch_S\to \Sch_{S^\prime}$.
% For a subcategory $\bbS$ of $\Sch$, 
% we denote by $\Sm_\bbS$ the restriction of $\Sm_{\Sch}$ on $\bbS$,
% and we write $\Sm_{(-)}$ for $\Sm_{\bbS}$ sometimes.

% Denote by $\Sm_{(-)}$ or by $\Sm_{\Sch}$ the contravariant functor 
% \[\Sm_{(-)}\colon \Sch\to \mathrm{Cat}; S\mapsto \Sm_S\]
% that takes a morphism $S^\prime\to S$ to the base change functor $\Sch_S\to \Sch_{S^\prime}$.
% Let $\bbS$ be a subcategory of $\Sch$.
% We denote the restriction of $\Sm_{(-)}$ on $\bbS$ by the same symbol $\Sm_{(-)}$ or by $\Sm_{\bbS}$.

\begin{definition}\label{def:correspondencesCorrS}
Let $\bbS$ be a subcategory of $\Sch$, and 
\[\mathcal S_{\bbS}\colon \bbS\to \mathrm{Cat}\] 
be a functor equipped with a natural embedding $\mathcal S_{\bbS}\hookrightarrow\Sch_{\bbS}$%
%be a subfunctor of $\Sm_{(-)}$
.
%A \emph{family of subcategories} $\mathcal S_{(-)}$ in $\Sm_{(-)}$ over $\Sch$ is a subfunctor.
A \emph{family of $\infty$-categories of correspondences} $\Corr_{\bbS}$ on $\mathcal S_{\bbS}$
is
a contravariant functor 
\begin{equation}\label{eq:CorrSchCataddpadd}\Corr_{\bbS}\colon \bbS\to \mathrm{Cat}_\infty; S\mapsto \Corr_S; f\mapsto f^*\end{equation}
%from the category of noetherian separated schemes to the $\infty$-category of $\infty$-categories
equipped with a natural structure of an $\infty$-category of correspondences over $S$ for each $S\in\bbS$,
and a natural isomorphism 
\[f^*(-)\times_{S^\prime} (X\times_S S^\prime)\simeq f^*(-\times_S X),\]
for each morphism
$f\colon S^\prime\to S$, and $X\in \mathcal S_S$.
%todo
We write $\mathcal S_{(-)}$ for $\mathcal S_{\bbS}$, when $\bbS$ is clear from the context,
and %We
write $\Corr_{\bbS}$, or $\Corr_{(-)}$,
for $\Corr(\mathcal S_{\bbS})$.
%\todo{added}
\end{definition}
%there is a natural equivalence
% and such that 
% for any 
% morphism \[f\colon S^\prime\to S,\] and $X\in \mathcal S_S$
% %todo
%
% %there is a natural equivalence
% \[f^*(-)\times_{S^\prime} (X\times_S S^\prime)\simeq f^*(-\times_S X).\]
%
% \begin{definition}
% If the functor \eqref{eq:CorrSchCataddpadd} lands in $\mathrm{Cat}^{\operatorname{padd}}_\infty$ or $\mathrm{Cat}^{\operatorname{add}}_\infty$ the family of $\infty$-categories of correspondences will be called \emph{additive} or \emph{preadditive}.
% For a preadditive family of $\infty$-categories of correspondences $\Corr_S$, we denote by $\Corr^\gp_{(-)}$ the corresponding additive one.
% \end{definition}
%
%\todo{add about $\HH(\mathcal S_{)-})$}
%Let $\mathcal S_{(-)}$ be a subfunctor of $\Sm_{(-)}$ as above.
Given a family of topologies $\tau$ on $\mathcal S_{\bbS}$ 
as above 
that is compatible with the base change functors,
\Cref{def:ShvtauHHtauA1} gives the respective families of $\infty$-categories. 
% $\Shv_{\Sigma,\tau}(\mathcal S_{(-)})$, 
% $\HH_{\Sigma,\tau}(\mathcal S_{(-)})$,
% $\Shtrd_{\Sigma,\tau}(\mathcal S_{(-)})$, 
% $\HH^\trd_{\Sigma,\tau}(\mathcal S_{(-)})$.

\begin{definition}\label{def:Prepaddadd}
An $\infty$-category $A$ is called \emph{preadditive} if it admits a zero 
object and the map 
$$X \coprod Y \to X \times Y$$
is an equivalence for any objects $X,Y \in A$. In this case, it makes sense to talk about the direct sum of objects, which we denote by $X \oplus Y$.

Preadditive categories are canonically enriched over $\mathbb{E}_\infty$-monoids (see \cite[Proposition 2.3]{UnivMultSp}). %TODO link
An $\infty$-category $A$ is called \emph{additive} if all mapping spaces are group-like with respect to the $\mathbb{E}_\infty$-monoid structure.
\end{definition}

\begin{example}
\begin{enumerate}
\item
The $2$-category of spans of morphisms of $G$-sets $\operatorname{Span}_G$ is 
preadditive,
\cite[Def. 5.7]{Gspans}. %TODO link to Barwick
\item 
Various $\infty$-categories of correspondences in motivic homotopy theory are 
preadditive. In particular, $\Corr_S^\fr$ is preadditive, see \cite{ehksy}, as well as the 
discrete category $\operatorname{Cor}_S$ from \cite{Voe-motives}. %TODO link EHKSY, Voevodsky
\end{enumerate}
\end{example}
\begin{definition}
If the functor \eqref{eq:CorrSchCataddpadd} lands in $\mathrm{Cat}^{\operatorname{padd}}_\infty$ or $\mathrm{Cat}^{\operatorname{add}}_\infty$ the family of $\infty$-categories of correspondences will be called \emph{additive} or \emph{preadditive}.
For a preadditive family of $\infty$-categories of correspondences $\Corr_S$, we denote by $\Corr^\gp_{S}$ the corresponding additive one.
\end{definition}

%\subsection{Finiteness property}\label{subsect:FinE}
\subsection{Continuous families}\label{subsect:FinE}

Let $\bbS\subset \Sch$, $\calS_\bbS\subset\Sch_{\bbS}$ 
be subcategories closed with respect to filtered limits and coproducts
and such that
for any $S\in\bbS$, $z\in S$, $\lSz\in\bbS$, and there is a filtered system of Zariski neighbourhoods $U_\alpha\in\bbS$ of $z$ in $S$ such that $\varprojlim_\alpha U_\alpha=\lSz$, and
any scheme in $\calS_{\lSz}$ equals $X_\lSz=X\times_{U_{\alpha}} \lSz$ of a scheme $X\in \calS_{U_\alpha}$ for some $U_\alpha$.
%todo reread
\begin{definition}\label{def:EmbedCorr}
A family of correspondences $\Corr(\calS_{\bbS})$ %over $\bbS\subset\Sch$ %schemes $S\in \Sch$
satisfies \emph{the property (Embed)}, if for any open immersion $U\to S$ in $\bbS$, there is a fully faithful functor $j_{\#}\colon \Corr_U\to \Corr_S$ that is left adjoint to $j^*\colon \Corr_S\to \Corr_U$.
\end{definition}
\begin{definition}\label{def:FinECorr}
Given a point $z\in S$, for $S\in\bbS$, and 
$X,Y\in \calS_S$,
consider the morphism in 
$\Spc$,
\begin{equation}\label{eq:CorretatodlCorrUalpha}\Corr_{\lSz}(X\times_S \lSz,Y\times_S\lSz)\leftarrow 
\varinjlim_{\alpha} 
\Corr_{U_\alpha}(X\times_S U_\alpha,Y\times_S U_\alpha),\end{equation}
induced by inverse image functors along the morphisms $\lSz\to U_\alpha$,
where $U_\alpha$ runes over the filtered set of Zariski neighbourhoods of $z$. %,
%todo
%and $h^\tr_Y(V)=\Corr_{S}(V,Y)$.
We say that
%a \emph{family of $\infty$-categories of correspondences} 
$\Corr_{\bbS}$
%schemes $S\in \Sch$
is \emph{continuous}, 
%or satisfies \emph{the property (FinE)}, 
if it satisfies (Embed) and
for any $S\in\bbS$, and a point $z\in S$, for any $X,Y\in \calS_S$, 
%todo
the morphism \eqref{eq:CorretatodlCorrUalpha} is 
%an equivalence of spaces,
%i.e. the 
an isomorphism %equivalence 
in the $\infty$-category $\Spc$.
% ,
% \[\Corr_\eta(X\times_S\eta,Y\times_S\eta)\simeq 
% %todo
% %\varprojlim_{\alpha}
% \varinjlim_{\alpha} 
% \Corr_{U_\alpha}(X\times_S U_\alpha,Y\times_S U_\alpha),\]
% %there is the equivalence for the representable presheaves
% % \[h^\tr_\eta(X\times_S\eta)\simeq 
% % %todo
% % %\varprojlim_{\alpha}
% % \varinjlim_{\alpha} 
% % h^\tr_{U_\alpha}(X\times_S U_\alpha),\]
% where $U_\alpha$ runes over the filtered set of Zariski neighbourhoods of $\eta$,
% %todo
% and $h^\tr_Y(V)=\Corr_{S}(V,Y)$.
\end{definition}

\begin{example}
(0)
The %preadditivisation of the 
family of the categories $\Sch_{(-)}$ over $\Sch$ satisfies (Embed).
%The family of the categories $\Sm_S$ as a categories of correspondences satisfies (Embed).
%
%\todo{do we need preadditivisation}
(1)
The family of the $\infty$-category of correspondences $\Corr^\fr_{(-)}$ satisfies (Embed).
\end{example}

%todo
\begin{definition}\label{def:Embed}
Let $\tau$ be a family of topologies on the categories 
$\calS_\bbS$. % for $S\in \Sch$.
We say that $\tau$ satisfies the property (Embed), 
if for any open immersion $U\hookrightarrow S\in \bbS$, 
the canonical embedding functor \[\calS_U\to \calS_S\] preserves and detects $\nu$-coverings.
In other words, 
a morphism $\widetilde X\to X$ in $\calS_U$ is a $\tau$-covering, if 
its image $\widetilde X\to X$ in $\calS_S$ is a $\tau$-covering.
\end{definition}
% \begin{example}\label{ex:embedZarNistf}
% The property (Embed) holds for the Zariski, Nisnevich, and the trivial fibre topologies.
% \end{example}

\begin{lemma}\label{lm:Embednu}
Let $\tau$ be a family of topologies on $\calS_\bbS$ that has the property (Embed).
Then for any open immersion of schemes $j\colon U\to S$ in $\bbS$,
the functor 
\begin{equation*}\label{eq:jusHtriv}
j^*\colon \Pre(\calS_S)\to \Pre(\calS_U)
\end{equation*}
%\mathbf{H}_{\mathrm{triv}}
preserves $\tau$-sheaves.
\end{lemma}
\begin{proof}
Let
$F\in \Shv_\tau(\calS_S)$.
Let
$v\colon \widetilde X\to X$ be a $\tau$-covering in $\calS_U$.
Then $v$ is a $\tau$-covering in $\calS_S$.
Denote by $h(X)$ the representable presheaves, 
and by $h_{\widetilde X}(X)$ the covering sieves
in both categories 
$\calS_S$ and $\calS_U$.
Then $j^*(h(X))=h(X)$, and $j^*h_{\widetilde X}(X)=h_{\widetilde X}(X)$.
%todo sieve
The sequence of isomorphisms 
\[\begin{array}{lcl}
\mathrm{Map}_{\Pre(\calS_U)}(h_{\widetilde X}(X),j^*F) &\simeq &
\mathrm{Map}_{\Pre(\calS_U)}(h_{\widetilde X}(X),F)\\
&\simeq& 
\mathrm{Map}_{\Pre(\calS_U)}(h(X),F)\\
&\simeq& 
\mathrm{Map}_{\Pre(\calS_S)}(h(X),j^*F)
\end{array}\]
implies that $j^* F$ is a $\tau$-sheaf.
\end{proof}

\begin{definition}\label{def:FinProp}
Let $\tau$ be a family of topologies on $\calS_\bbS$. %\todo{$\SmAff_{(-)}$}%noetherian separated todo base change
We say that $\tau$ is \emph{continuous} %property (FinE) 
if $\tau$ satisfies (Embed), and 
for any 
$S\in\bbS$, $X\in \calS_S$, 
point $z\in S$, and 
$\nu$-covering $\widetilde{X_\lSz}\to X_\lSz$, where $X_\lSz=X\times_S \lSz$, 
there is 
a Zariski neighbourhood $U$ of $z$ and 
a $\nu$-covering $\widetilde X_U\to X\times_S U$ 
such that 
the induced morphism $\widetilde X_U\times_U \lSz\to X_\lSz$ is 
a refinement of $\widetilde{X_\lSz}\to X_\lSz$.
\end{definition}
% \begin{example}\label{ex:FinProprtyfortf}
% The property (FinE) holds for any family of topologies $\tau$ on the categories $\SmAff_S$ for $S\in \Sch$ satisfying (Embed) and such that $\tau$ is trivial on the categories $\SmAff_k$ for all fields $k$,
% in particular, the trivial fibre topology from \Cref{ex:tf} has the property.
% \end{example}
% \begin{example}\label{ex:FinProprtyforZarNis}
% The Zariski and Nisnevich topologies over a noetherian separated scheme $S$ have the property (FinE). Note that if (FinE) holds for topologies $\tau$ and $\nu$ then it holds for the topology $\tau\cup\nu$.
% \end{example}
\begin{example}\label{ex:embedZarNistfFinEmbed}
%The properties (Embed) and (FinE) 
%hold for the Zariski, Nisnevich, and the trivial fibre topologies.
The Zariski, Nisnevich, and the trivial fibre topologies are
continuous in sense of \Cref{def:FinProp}.
\end{example}
% \begin{example}\label{ex:FinProprtyfortf:ex:FinProprtyforZarNis}
% The property (FinE) holds for the trivial fibre topology from \Cref{ex:tf},
% the Zariski and 
% the Nisnevich topologies over a noetherian separated scheme $S$.
% \end{example}

\begin{lemma}\label{lm:ContinuityCov}
Let $\tau$ be a continuous family of topologies on $\calS_\bbS$. %satisfies the property (FinE).
Then the functor $j^*\colon \Pre(\calS_S)\to \Pre(\calS_\lSz)$ 
preserves $\tau$-sheaves for each point $z$ of a scheme $S\in\bbS$.
\end{lemma}
\begin{proof}
Any scheme in $\calS_{\lSz}$ equals $X_\lSz=X\times_S \lSz$ of a scheme $X\in \calS_S$.
Given a $\tau$-covering $\widetilde{X_\lSz}\to X_\lSz$, by \Cref{def:FinProp} %(FinE) 
there is a $\tau$-covering $\widetilde X_V\to X\times_S V$, for an Zariski neighbourhood $V$ of $z$ such that  $\widetilde X_\lSz=\widetilde X_V\times_V\lSz\to X_\lSz$ is a refinement of $\widetilde{X_\lSz}\to X_\lSz$.
If $F\in \Pre(\calS_S)$ is a $\nu$-sheaf, then for any Zariski neighbourhood $U$ of $z$ in $V$ 
\[F(X_V\times_V U)\simeq F(\check C(\widetilde X_V\times_V U)).\]
Hence for a filtered system of $U_\alpha$ such that $\varprojlim_\alpha U_\alpha=\lSz$, we have %todo s
\[
j^*F(X_\lSz)\simeq
\varinjlim_\alpha F(X_V\times_V U_\alpha)\simeq 
\varinjlim_\alpha F(\check C(\widetilde X_V\times_V U))\simeq
j^*F(\check C(\widetilde X_\lSz)).
\]
Thus $L_\nu j^*F\simeq F$, and consequently $j^*F$ is a $\tau$-sheaf.
\end{proof}

\begin{lemma}\label{lm:Continuity}
Let $S\in \bbS$, and $S^{(0)}$ denote the union of generic points of $S$,
and $U_\alpha$ be a filtered system of Zariski neighbourhoods, such that $\varprojlim_\alpha U_\alpha=S^{(0)}$.
Denote $j_\alpha\colon U_\alpha\to U$, $j\colon S^{(0)}\to S$.
% Let $U_\alpha$ be the filtered system of dense open subschemes of a scheme $S$.
% Let $S\in \Sch$, $S^{(0)}$ denote the union of generic points of $S$,
% and $U_\alpha$ be a filtered system, such that $\varprojlim_\alpha U_\alpha=S^{(0)}$.
% Denote $j_\alpha\colon U_\alpha\to U$, $j\colon S^{(0)}\to S$.
% % Let $U_\alpha$ be the filtered system of dense open subschemes of a scheme $S$.
The following natural equivalence of the endofunctors on $\Pre(\calS_U)$ holds 
\begin{equation}\label{eq:jjeqjjalpha}\varinjlim_{\alpha} (j_\alpha)_*(j_\alpha)^*\simeq j_*j^*\end{equation}
for the functors $(j_\alpha)_*\colon \Pre(\calS_{U_\alpha})\to \Pre(\calS_U)$ and $j_*\colon \Pre(\calS_{S^{(0)}})\to \Pre(\calS_U)$.

Let 
$\tau$ be a continuous family of topologies on $\calS_\bbS$ %satisfies (FinE) 
in sense of \Cref{def:FinProp},
and
a continuous family of $\infty$-categories of correspondences $\Corr_S$ over $S\in \Sch$ 
%satisfying (FinE) 
in sense of \Cref{def:FinECorr}.
Then equivalence \eqref{eq:jjeqjjalpha}
holds for 
the $\infty$-categories $\Shv_{\tau}(-)$, 
and $\Shv_{\tau}^\tr(-)$. 
\end{lemma}
\begin{proof}
% The cofiltered limit $\varprojlim_\alpha U_\alpha$ equals to the union of generic points $S^{(0)}$ of $S$.
% Consider the continuous functor on $\EssSm_S$ defined by $F$ and denote it by the same symbol.
Since $\varprojlim_\alpha U_\alpha=S^{(0)}$,
\[\varinjlim_{\alpha} (j_\alpha)_*(j_\alpha)^*(F)(X)\simeq \varinjlim_{\alpha}F(X\times_S U_\alpha) \simeq F(X\times_S S^{(0)})\simeq j_*j^*(F)(X).\]
So equivalence \eqref{eq:jjeqjjalpha} for $\infty$-categories $\Pre(-)$ follows.
By \Cref{lm:ContinuityCov}
the functors $j^*$ and $j_*$ on $\Pre(-)$
induces the ones on $\Shv_\tau(-)$
by the restriction along the embedding $\Shv_\tau(-)\to \Pre(-)$.
So equivalence \eqref{eq:jjeqjjalpha}
for $\infty$-categories $\Shv_\tau(-)$ follows
from the one for $\Pre(-)$.
% The similar equivalence for $\Shv_\tau(-)$ follows 
% in view of \Cref{lm:ContinuityCov};
Equivalence \eqref{eq:jjeqjjalpha} 
for $\Shv_\tau^\tr(-)$ 
holds as well 
because of the canonical equivalence 
\[r_*j^*\simeq j^*r_*,\] 
where $r\colon \calS_S\to \Corr(\calS_S)$, $r_*\colon \Shv_{\tau}^\tr(\calS_S)\to \Shv_{\tau}(\calS_S)$,
for any $\Corr_{(-)}$ that is continuous in sense of \Cref{def:FinECorr}. %satisfying (FinE).
%similar isomorphism holds becuase of (FinE) for $\Corr_{(-)}$.
\end{proof}

%\todo{new text : begin}

\subsection{Lifting properties with respect to the affine henselian pairs.}

We discuss the lifting property for discrete presheaves on $\EssSmAff_S$, $S\in\Sch$, with respect to affine henselian pairs, and apply this to framed correspondences.

\begin{definition}\label{def:RLwAHP}
We say that a discrete presheaf of sets $c$ on $\Aff_S$ %\Aff\cap\Sch_S
has \emph{the lifting property with respect to affine henselian pairs},
% or say that $c$
% \emph{satisfies (AHP)},
whenever for any $X\in \Aff_S$ and a closed subscheme $Y\subset X$,
the morphism 
\begin{equation}\label{eq:RLwAHP:c_surj} c(X^h_Y)\to c(Y)\end{equation}
is surjective. 
\end{definition}

\begin{lemma}\label{lm:XYWh}
For any $X\in \Aff_S$, a closed subscheme $Y\subset X$, and a closed subscheme $W\subset X$,
the closed immersion $Y\amalg_{Y\cap W} W^h_{Y\cap W}\leftarrow Y$ is a henselian pair,
and
the closed immersion $X^h_Y\leftarrow Y\amalg_{Y\cap W} W^h_{Y\cap W}$ is a henselian pair.
\end{lemma}
\begin{proof}
The claims follow from the universal property of the henselisation.
\end{proof}

\begin{lemma}\label{lm:henshairliftstructuredcofsmooth}
%\label{lm:henshairliftFrsmooth}
Suppose a discrete presheaf of sets $c$ on $\Aff_S$
has the lifting property with respect to affine henselian pairs.
%satisfies (RLwAHP).
Then for any $X\in \Aff_S$, a closed subscheme $Y\subset X$, and a closed subscheme $W\subset X$, the morphism 
\[ c(X^h_Y)\to c(Y\amalg_{Y\cap W} W^h_{Y\cap W})\]
is surjective. 
\end{lemma}
\begin{proof}
%Applying \Cref{def:RLwAHP} to 
Applying the surjectivity of \eqref{eq:RLwAHP:c_surj} to %from 
$X$ being $X^h_Y$, and $Y$ being $Y\amalg_{Y\cap W} W^h_{W\cap Y}$
by \Cref{lm:XYWh} we get the claim.
\end{proof}
% For any $S\in\Sch$ and a closed subscheme $Z$ in $S$ 
% denote
% %\[L_{\A^1_{S,Z}}\colon \PSh(\Aff_{S,Z})\to \PSh(\Aff_{S,Z}); F\mapsto F^{\Delta^\bullet_{B,Z}},\]
% where $\Delta^\bullet_{B,Z}$ is an .
For any $U\in\Sch$, we consider the functor \[\sSet\to \Aff_U; K\mapsto K_U\] that is
the left can extension of the functor $\Delta\to \Aff_U$ given by
the cosimplicial object $\Delta^\bullet_U$.

\begin{corollary}\label{cor:DeltatrivfibhenspairsmoothLRwAHP}
Suppose a discrete presheaf of sets $c$ on $\Aff_S$
has the lifting property with respect to affine henselian pairs in sense of \Cref{def:RLwAHP},
%satisfies (RLwAHP) 
and satisfies closed gluing on $\Aff_S$.
Let $U\in \AffSm_{S}$, and $Y\subset U$ be a closed subscheme.
Then the morphism of simplicial sets
\[c((\Delta^\bullet_U)^h_{Y})\to c(\Delta^\bullet_Y).\]
is a trivial fibration, where $(\Delta^\bullet_U)^h_{Y}$ denotes the cosimplicial object with terms $(\Delta^n_U)^h_{(\Delta^n_Y)}$.
\end{corollary}
\begin{proof}
Let $K\to N$ be an injection of simplicial sets.
Then there are the induced closed immersions $K_U\to N_U$, $N_Y\to N_U$, $K_Y\to K_U$, see \Cref{den:introductionnotation}.
%Recall notation $Y=U\times_S Z$.
% Note also that in view of our notation there are isomorphisms of cosimplicial objects in $\Aff_Y$
% \[
% U\times_S\Delta^\bullet_S\cong \Delta^\bullet_U,
% U\times_S\Delta^\bullet_Y\cong \Delta^\bullet_Y.
% \]
% \[U\times_S\Delta^\bullet_S\cong \Delta^\bullet_U,
% U\times_S\Delta^\bullet_Y\cong \Delta^\bullet_Y.\]
Then since $c$ satisfies closed gluing,
the morphism of sets
\[ 
c((\Delta^\bullet_U)^h_Y)^{N} \to 
c((\Delta^\bullet_U)^h_Y)^{K}\times_{
(c(\Delta^\bullet_Y)^{K})}
c(\Delta^\bullet_Y)^{N}
\]
equals
\[
c((N_U)^h_{Y})\to 
c((K_U)^h_{Y})\times_{
c(K_Y)}
c(N_Y)\cong
c((K_U)^h_{Y} \amalg_{K_Y} (N_Y)),
\]
where
$(N_U)^h_Y = (N_U)^h_{N_Y}$, $(K_U)^h_Y = (K_U)^h_{K_Y}$.
The last morphism is surjective 
by \Cref{lm:henshairliftstructuredcofsmooth} applied to 
$X=N_U$, 
$Y=N_Y$,
$W= K_U$.
\end{proof}

\begin{example}\label{ex:LRwAHP}
The following presheaves satisfy the lifting property with respect to affine henselian pairs:
\begin{itemize}
\item[0)] The representble presheaves of the category $\SmAff_{S}$, $S\in \Aff$.
\item[1)] The presheaves $\Fr_+(-,X)$ for $X\in\SmAff_S$ of framed correspondences from \cite{VoevNotesFrCor,GP14}, see \cite[Lemma A11]{DKO:SHISpecZ}. 
\item[2)] The presheaves of \emph{normally} framed correspondences $h^\nfr(X)$ for $X\in \SmAff_S$ defined in \cite{ehksy} and \cite{Nesh:nfr_GN} because of the representability by the smooth affine scheme for smooth affine $X$ \cite{ehksy},
\item[3)] The presheaves $h^\pfr(X)=\Fr^{\mathrm{st:id}}(-,X)$ for $X\in\SmAff_S$ defined in \cite[Definition 7]{SmModelSpectrumTP} by the respective representability result. 
\end{itemize}
\end{example}

\begin{definition}\label{def:CorrAHP_first}
A family of correspondences $\Corr(\EssSmAff_S)$
satisfies the property (AHP)(1), %\footnote{AHP is a shorten for affine henselian pairs} 
if
for any $S\in\Aff$, and $X\in \EssSmAff_S$
there is a discrete presheaf of sets $c$ on $\EssSmAff_S$
such that 
$c$ satisfies %\todo{has} 
the lifting property with respect to affine henselian pairs in sense of \Cref{def:RLwAHP},
and
there is an isomorphism in $\Pre(\EssSmAff_S)$ %weak equivalence
\[L_{\A^1}h^\tr(X)\simeq L_{\A^1}c\]
 in $\Pre(\EssSmAff_S)$.
%where $L_{\A^1} = (-)^{\Delta^\bullet_S}.$
\end{definition}

\begin{example}
(0)
The %preadditivisation of the 
family of the categories $\EssSmAff_{\Aff}$ satisfies (AHP)(1) by \Cref{ex:LRwAHP}(0).

(1)
The family of the $\infty$-category of correspondences $\Corr^\fr(\EssSmAff_\Aff)$ satisfies (AHP)(1),
because by \cite[Corollary 2.3.25]{ehksy} it follows the equivalence 
$L_{\A^1}h^\fr(E)\simeq L_{\A^1}h^\nfr(E)$ on the category $\EssSmAff_S$,
while $h^\nfr(E)$ satisfies the lifting property with respect to affine henselian pairs by \Cref{ex:LRwAHP}(2).
% The then claim follows by \Cref{cor:LA1trivfibhenspairsmoothLRwAHP} and \eqref{ex:LRwAHP}(2).
Similar argument holds with the use any one of \Cref{ex:LRwAHP}(1), and \Cref{ex:LRwAHP}(3) as well.
% , applying \cite[]{}\todo{reference 2.2.20; more general for smooth affine} for $h^\nfr(E)$, or \cite[Proposition 3]{SmModelSpectrumTP} for $h^{\pfr}(E)$. 
% because by \cite[Corollary 2.3.25]{ehksy} there is an equivalence 
% $L_{\A^1}h^\fr(E)\simeq L_{\A^1}\Fr(-,E)$ on the category $\SmAff_S$,
% while $\Fr(-,E)$ satisfies the lifting property with respect to affine henselian pairs by \Cref{ex:LRwAHP}(1).
% % The then claim follows by \Cref{cor:LA1trivfibhenspairsmoothLRwAHP} and \eqref{ex:LRwAHP}(2).
% Similar argument holds with the use each one of \Cref{ex:LRwAHP}(2),(3), 
% applying \cite[]{}\todo{reference 2.2.20; more general for smooth affine} for $h^\nfr(E)$, or \cite[Proposition 3]{SmModelSpectrumTP} for $h^{\pfr}(E)$. 
%
% By \cite[Corollary 2.3.25]{ehksy} and by \cite[Proposition 3]{SmModelSpectrumTP} there are equivalences 
% $L_{\A^1}h^\tfr(E)\simeq L_{\A^1}h^\nfr(E)\simeq L_{\A^1}\Fr(E,-)\simeq L_{\A^1}h^{\pfr}(E)$ on the category $\SmAff_S$.
% So summarising $L_{\A^1}h^\tfr(E)\simeq L_{\A^1}h^\fr(E)$ on the category $\SmAff_S$, for any $h^\fr\in\{\Fr,h^\nfr,h^\pfr\}$.
% The then claim follows by \Cref{cor:LA1trivfibhenspairsmoothLRwAHP}.
\end{example}

\begin{definition}\label{def:CorrAHP_second}
A family of $\infty$-categories of correspondences $\Corr(\EssSmAff_{\Aff})$ 
satisfies the property (AHP)(2), 
if
for any 
$S\in \Aff$, and a closed subscheme $Z$,
such that $S^h_Z\simeq S$,
\[\Corr_S(S,X^h_{X\times_S Z})\simeq \Corr_S(S,X)\]
for
any 
$X\in \EssSmAff_S$.
%for all $U,X\in\EssSmAff_S$.
\end{definition}

\begin{example}
(0)
The %preadditivisation of the 
family of the categories $\EssSmAff_{\Aff}$ satisfies (AHP)(2).

(1)
The family of the $\infty$-category of correspondences $\Corr^\fr(\EssSmAff_\Aff)$ satisfies (AHP)(2).
Indeed, the spaces $\Corr^\fr_S(U,X)$ are the spaces of the data
\begin{equation}\label{eq:tfrspanandframing}
U\xleftarrow{f}Z\xrightarrow{v} X, \quad \tau\colon \mathcal L_f\simeq 0,
\end{equation}
where $f$ is finite flat lci, and $\mathcal L_f$ is the cotangent complex of $f$ in the K-theory $\infty$-groupoid $K(Z)$.
% \todo{write the span}
So the claim follows by point (0) applied to morphisms $v$ in $\EssSmAff_S$ from \eqref{eq:tfrspanandframing}.
% , because
% the role of $X$ in \eqref{} is contained in the right side arrow that 
\end{example}

\begin{definition}\label{def:CorrAHP}
A family of $\infty$-categories of correspondences $\Corr(\EssSmAff_\Aff)$ %\todo{$\Aff$ or $\Sch$}
satisfies the property (AHP)\footnote{affine henselian pairs}, %\footnote{AHP is a shorten for affine henselian pairs} 
if
it satisfies both (AHP)(1), and (AHP)(2).
\end{definition}

\subsection{Localisation property}\label{subsect:LocProp}
%and grouplike sheaves with transfers

%Let $S$ be a scheme, $Z$ be a closed subscheme.

In this section, \begin{itemize}
\item $\bbS$ is a category of schemes 
that contains all open and all closed subschemes of any $S\in\bbS$.  
\item $\Corr(\calS_\bbS)$ is 
a continuous family of $\infty$-categories of radditive correspondences in sense of \Cref{def:FinECorr},
\item $\tau$ is a continuous family of topologies on 
$\calS_{\bbS}$ %EssSmAff
in sense of \Cref{def:FinProp}.
\end{itemize}

% Let $\bbS$ be a category of schemes 
% that contains all open and all closed subschemes of any $S\in\bbS$.  
% % ny one from the list 
% % $\Sch$, $\Aff$, $\Schns$, $\Affns$
% Let $\Corr(\calS_\bbS)$ be 
% a family of an $\infty$-categories 
% % where $\bbS$ is any one from the list 
% % $\Sch$, $\Aff$, $\Schns$, $\Affns$,
% that 
% is continuous
% %satisfies (FinE) 
% in sense of \Cref{def:FinECorr}, and
% % and 
% % satisfies (AHP) in sense of \Cref{def:CorrAHP}.
% %
% %Let $i\colon Z\to S$ be a closed immersion of affine schemes.
% % Let $\tau$ be a topology on $\Sch_S$ such that 
% % $\tau$ has finite type
% % %satisfies (FinE) 
% % in sense of \Cref{def:FinProp},
% % $\tau\supset \tf$, and $\wtau=\stau$ on $\SmAff_{S,Z}$, see \Cref{def:wtaustau}.
% %
% % Let $\Corr_{(-)}$ be 
% % a preadditive family of $\infty$-categories of correspondences on $\SmAff_{(-)}$ over $\Sch$,
% % and
% % $\tau$ 
% %  be family of topologies on $\SmAff_{(-)}$,
% %  that has finite type
% % % satisfy 
% % %%Suppose $\Corr_S$ and $\tau$ satisfy the property 
% % %(FinE)
% % in sense of \Cref{def:FinECorr} and \Cref{def:FinProp} respectively,
% % %\todo{necessary or not}
% % and $\Corr_{(-)}$ satisfy (HU) in sense of \Cref{def:CorrHU}. 
% %
% $\tau$ be a family of topologies on 
% $\calS_{\bbS}$ %EssSmAff
% that is continuous in sense of \Cref{def:FinProp}.
% that induces topologies on the categories $\SmAff_S$ over a 
%of 
% separated noetherian scheme $S$
% such that the base change functors are continuous. 

\begin{definition}\label{def:LocLocAff}
%Given 
% For a scheme $S\in \Sch$, consider
% a closed immersion $i\colon Z\to S$, and the open embedding $j\colon S-Z\to S$.

We say that $\tau$ satisfies \emph{the localisation property} with respect to $\Corr(\calS_\bbS)$, 
if 
for any $S\in\bbS$, 
a closed immersion $i\colon Z\to S$, and the open embedding $j\colon S-Z\to S$%$i\colon Z\to S$, and $j\colon S-Z\to S$ as above
,
for any $F\in \HHtrd_{\Sigma,\tau}(\calS_{S})$, 
the canonical sequence 
\[i_* i^! F\to F \to j_* j^* F,\]
provided by the adjunctions $i_*\dashv i^!$, $j^*\dashv j_*$, 
\[\HHtrd_{\Sigma,\tau}(\calS_{Z})\rightleftarrows \HHtrd_{\Sigma,\tau}(\calS_{S})\rightleftarrows\HHtrd_{\Sigma,\tau}(\calS_{S-Z}),\]
is a fibre sequence, 
see \Cref{section:LocalisationTheoremtf} for the definitions of $i_*,i^!,j_*,j^*$. %sect:LocThConclusion%section:LocalisationTheoremtf
\end{definition}

\begin{lemma}\label{lm:Loc(Aff)detectssheaves}
Suppose $\Corr(\calS_\bbS)$ is preadditive, and
suppose that $\tau$ satisfies the localisation property with respect to $\Corr(\calS_\bbS)$.
Let
$\nu$ be a continuous 
family of topologies that contains $\tau$.
Then for any %noetherian 
$S\in \bbS$ of finite Krull dimension,
the functor
\[\SH^{S^1,\tr}_{\Sigma,\tau}(\calS_S)\to \prod_{z\in S}\SH^{S^1,\tr}_{\Sigma,\tau}(\calS_z); F\mapsto (i^!_z j^*_{\lSz}F)_{z\in S},\]
%; F\mapsto (j^*i^!F)
% \[\Htrgp_{\nu\cup\tau}(\SmAff_S)\to \prod_{z\in S}\Htrgp_\tau(\SmAff_z)\]
where $j\colon \lSz\to S$, $i_z\colon z\to\lSz$,
detects $\nu$ sheaves,
and similarly for $\Htrgp_{\Sigma,\tau}(-)$, and $\SH^{\tr}_{\Sigma,\tau}(-)$.
\end{lemma}
\begin{proof}
The claim is tautological for fields.
Assume the claim for all base schemes of dimension less than $\dim S$.
Let $F\in \SH^{S^1,\tr}_{\tau}(\calS_S)$ whose image in $\SH^{S^1,\tr}_\tau(\calS_z)$ is a $\nu$-sheaf for each $z\in S$.
Consider the cofiltered system of closed subschemes $Z_\alpha$ of positive codimension in $S$.
Then 
\[(i_\alpha)_*(i_\alpha)^! F \simeq \fib(F \to (j_\alpha)_*(j_\alpha)^* F),\]
where $j_\alpha\colon S-Z_\alpha\to S$, and $i_\alpha\colon Z_\alpha\to S$ are the canonical embeddings.
Then
\begin{equation}\label{eq:iFfibFjF}\varinjlim_\alpha (i_\alpha)_*(i_\alpha)^! F \simeq \fib(F \to \varinjlim_\alpha(j_\alpha)_*(j_\alpha)^* F),\end{equation}
The projective limit of the schemes $S-Z_\alpha$ equals the union of generic points of $S$. Denote it by $S^{(0)}$ and by $j\colon S^{(0)}\to S$ the canonical embedding,
then by \Cref{lm:Continuity}
\[\varinjlim_\alpha(j_\alpha)_*(j_\alpha)^* F\simeq j_*j^* F.\]
By the assumption it follows that $(i_\alpha)_*(i_\alpha)^! F$, and $j_*j^* F$ are $\nu$-sheaves.
By \eqref{eq:iFfibFjF} it follows that
\begin{equation*}\label{eq:FfibjFiF[1]}F \simeq \fib( j_*j^* F \to \varinjlim_\alpha i_*i^! F[1] ).\end{equation*}
Hence $F$ is a $\nu$-sheaf.
Thus the claim for $\SH^{S^1,\tr}_{\Sigma,\tau}(-)$ is proved.
The claim on $\Htrgp_{\Sigma,\tau}(-)$ follows, 
because of the embedding $\Htrgp_{\Sigma,\tau}(-)\hookrightarrow\SH^{S^1,\tr}_{\Sigma,\tau}(-)$. 
The claim on $\SH^{\tr}_{\Sigma,\tau}(-)$ follows, 
because the functor 
\[\prod_{l\in\mathbb Z}\Omega^{l}\colon\SH^{\tr}_{\Sigma,\tau}(\calS_S)\to \prod_{l\in\mathbb Z}\SH^{S^1,\tr}_{\Sigma,\tau}(\calS_S)\] detects $\nu$-sheaves.
\end{proof}

\section{Zariski fibre topology}\label{sect:zf} 

%Throughout the rest of the paper we fix a category of base schemes $\mathbf S$.
% \todo{Is the previous phrase needed: "Throughout the rest of the paper we fix a category of base schemes $\mathbf S$."}
Throughout this section \begin{itemize}
\item $\mathcal S_{(-)}$ is any one from the list $\Sch_{(-)},\Aff_{(-)},\Sm_{(-)},\SmAff_{(-)}$, 
\item $\mathbf S=\Schfns,\Afffns$.
\end{itemize}
Let $\nu$ be a family of subtopologies of the \'etale topology on 
the family of categories $\mathcal S_{\bbS}$. % over the category of base schemes $\mathbf S$,
% where $\mathcal S_{(-)}=\Sch_{(-)},\Aff_{(-)},\Sm_{(-)},\SmAff_{(-)}$, $\mathbf S=\Schfns,\Afffns$.
%$\Sch_{(-)}$ over $\Sch$, or $\Aff_{(-)}$ over $\Aff$.
%We write for $\Sch_{(-)}$ over $\Sch$ in what follows.

\begin{definition}\label{def:nuf}
Given a family of subtopologies $\nu$ of the \'etale topology on $\calS_\bbS$,
define the family of subtopologies $\nuf$ of the \'etale topology on $\calS_{\bbS}$ 
as the strongest one
such that for each reduced zero-dimensional scheme $z\in\bbS$ the topology $\nuf$ on $\calS_z$ %on\Sch_ %is equal or weaker then 
is contained in $\nu$. 
\end{definition}
\begin{remark}\label{rem}
%Let $\tau$ be a family of subtopologies of the Nisnevich topology on $\Sch_{(-)}$ over $\Sch$.
A $\nuf$-covering $\wX\to X$ in $\calS_{S}$ for $S\in\bbS$
is a Nisnevich covering that base change along the morphism $z\to S$ for each point $z\in S$ is a $\nu$-covering.
\end{remark}
% \begin{remark}
% The correctness of \Cref{} requirfollows because $\tauf$-coverings from \Cref{} define a topology.
% \end{remark}
\begin{example}
Let $\zf$ be the topology $\nuf$ for $\nu$ being Zariski topology.
We call it the Zariski fibre topology.
\end{example}
%\subsection{}

%Let us recall the definition of the trivial fibre topology, from \cite[Definition 3.1]{DKO:SHISpecZ}.
Let us recall the definition of so called $\tf$-topology, the trivial fibre topology
introduced in \cite[Definition 3.1]{DKO:SHISpecZ}.
\begin{definition} %[\protect{\cite[Definition 3.1]{DKO:SHISpecZ}}]
(1)
We say that a morphism $v\colon \widetilde X\to X$ in $\calS_S$ is a $\tf$-covering, if it is \'etale affine, and for each $z\in S$, there is a morphism $X\times_S z\to \widetilde X$ that composite with $w$ equals the morphism $X\times_S z\to X$.
% We say that a morphism $\widetilde X\to X$ in $\calS_S$ is a $\tf$-covering, if it is \'etale affine, and for each $z\in S$, there is a lifting $X\times_S z\to \widetilde X$.

(2)
We define a $\tf$-square to be a pullback square of the form 
\[\xymatrix{
U^\prime\ar@{^(->}[r]\ar[d] & X^\prime\ar[d]^{f}\\
U\ar@{^(->}[r]^j & X
}\]
where 
$f$ is \'etale affine, and there is a closed subscheme $Z$ in $S$ such that
$U=X-Y$,
and
$Y\times_X X^\prime\simeq Y$,
where 
$Y=Z\times_S X$,
and
$j$ is the open immersion.
\end{definition}
% \begin{proposition}
\begin{remark}
It is shown in \cite[\S 3]{DKO:SHISpecZ}
$\tf$-squares from point (2) define a regular bounded cd-structure 
on $\Aff_S$ or $\Sch_S$. 
%and %generate a completely decomposable topology on 
For $S\in\Schfns$, 
the induced completely decomposable topology coincides with the topology defined by $\tf$-coverings from point (1)
similarly to Nisnevich topology \cite[Prop. 3.4.1]{Morel-Voevodsky}. 
%Since any $\tf$-suqare is a Nisnevich square, $\tf$-topology is a subtopology of the Nisnevich one.
\end{remark}
% Recall that \cite{DKO:SHISpecZ} introduces so called trivial fibre topology, $\tf$-topology, that is the complete decomposible topology on $\Sch_S$ 
% generated by the squares of the form \eqref{eq:NisSqSnieghbourhhdZ} such that $Y=Z\times_S V$, $Y^\prime=V^\prime\times_S Z$, and $e$ is \'etale affine, see \cite[Defintion 3.1]{DKO:SHISpecZ}.
%
According to \cite[Remark 3.7]{DKO:SHISpecZ} the $\tf$-topology on $\Aff_S$ over affine $S$ 
is the strongest subtopology of the Nisnevich topology that is trivial over the residue fields.
We prove the following general statement.
\begin{lemma}\label{lm:nuf=nucuptf}
% Let $\nu$ be
% such that
Suppose that $\nu$ is continuous, %\todo{continuous}
and furthermore,
for any $S\in\bbS$, $X\in\calS_S$, $z\in S$, and a 
$\nu$-covering $u\colon Y\to X_z$ over $z$,
there is 
%an \'etale neighbourhood $\tilde S$ of $z$ in $S$,
%and a $\nu$-covering $v\colon \wX\to X_{\tilde S}$ over $\tilde S$
a $\nu$-covering $v\colon \wX\to X_{S^h_z}$ over $S^h_z$
that base change $v_z\colon \wX_z\to X_z$ along the morphism $z\to S^h_z$ is a $\nu$-refinement of $u$.
% Suppose that 
% for any $S\in\Sch$, $X\in\Sch_S$, $z\in S$, and a 
% $\nu$-covering $v_z^\prime\colon \wX_z^\prime\to X\times_S z$ in $\Sch_z$,
% there is $\nu$-covering $v\colon \wX\to X$ 
% that base change $v_z\colon \wX\times_S z\to X\times_S z$ along the morphism $z\to S$ is a refinement of $v_z^\prime$.
%
Then \[\nuf = \nu\cup \tf,\] and $\nuf$ is continuous.
\end{lemma}
%\todo{induction of $\dim S$, finite type continuous}
\begin{proof}
By \Cref{def:nuf}, we have $\nu\cup\tf\subset\nuf$.
We prove the converse implication by the induction 
on $\dim S$. The claim is trivial for $S=\emptyset$, and suppose that claim for all because schemes of dimension less then $\dim S$.
Given a $\nuf$-covering $v\colon\wX\to X$ over $S$, 
we are going to show that there is %to construct 
a $\nuf$-refinement of $v$ that is a $\nu\cup\tf$-covering.
%Then the claim follows because $
%
Since and Nisnevich covering of $S$ is a $\tf$-covering over $S$, and since $\nu$ is continuous,
without loss of generality we can assume that $S=S^h_z$ for $z\in S$.
Consider the $\nu$-covering 
$v_z\colon \wX_z\to X_z$. %over $z$.
By assumption, there is 
a $\nu$-covering $\wX^\prime\to X$ such that the morphism
$\wX^\prime_z\to X_z$ is a $\nu$-refinement of $v_z$.
Then the morphism 
\begin{equation}\label{eq:wXpwX->wXp}\wX^\prime\times_X \wX\to \wX^\prime\end{equation}
is a $\nuf$-covering over $\tilde S$, and consequently it is \'etale.
The base change of \eqref{eq:wXpwX->wXp} along $z\to \tilde S$ equals $\wX^\prime_z\times_X\wX_z\to \wX^\prime_z$
and has a section provided by the morphism $\wX^\prime_z\to \wX_z$.
Then there is an open subscheme $V\subset \wX^\prime\times_X \wX$ such that the morphism $V\to \wX^\prime$ is an \'etale neighbourhood of $\wX^\prime_z$.
So the morphism
\[w\colon W = V\amalg (\wX^\prime\times_S (S-z))\to \wX^\prime\]
is a $\tf$-covering.
The morphism $\wX\times_S V\to V$ is a $\tf$-covering, because it is \'etale and admits a section given by $V\hookrightarrow \wX^\prime\times_X \wX\to \wX$.
The morphism $\wX\times_{S}(\wX^\prime\times_S (S-z))\to (\wX^\prime\times_S (S-z))$ is a $\nu\cup\tf$-covering by the inductive assumption since $\dim(S-z)<\dim S$.
Thus the morphisms in the sequence 
\begin{equation}\label{eq:seq:nucuptf}
\wX\times_S W
\to
W
%V\amalg (\wX^\prime\times_S (S-z))
\to 
\wX^\prime
\to 
X
\end{equation}
are $\nu\cup\tf$-coverings,
while the morphisms in the sequence
\begin{equation}\label{eq:seq:nuf}
\wX\times_S W
\to 
\wX
\to 
X
\end{equation}
are $\nuf$-coverings.
So the claim follows because the composites of \eqref{eq:seq:nucuptf} and \Cref{eq:seq:nuf} are equal.
% %\[v_z\colon \wX_z=\wX\times_S z\to X\times_S z=X_z\] over $z$.
% %eq:wXpwX->wXp
% By the assumption there is 
% an \'etale neighbourhood $\tilde S$ of $z$, %S^\prime
% and a $\nu$-covering $\wX^\prime\to X_{\tilde S}$ such that the morphism
% $\wX^\prime_z\to X_z$ is a $\nu$-refinement of $v_z$.
% Then the morphism 
% \begin{equation}\label{eq:wXpwX->wXp}\wX^\prime\times_X \wX\to \wX^\prime\end{equation}
% is a $\nuf$-covering over $\tilde S$, and consequently it is \'etale.
% The base change of \eqref{eq:wXpwX->wXp} along $z\to \tilde S$ equals $\wX^\prime_z\times_X\wX_z\to \wX^\prime_z$
% and has a section provided by the morphism $\wX^\prime_z\to \wX_z$.
% Then there is an open subscheme $V\subset \wX^\prime\times_X \wX$ such that the morphism $V\to \wX^\prime$ is an \'etale neighbourhood of $\wX^\prime_z$.
% %the morphism
% Thus the composite morphism 
% \[\coprod_{z\in S}(V\amalg (\wX^\prime-\wX^\prime_z))\to \coprod_{z\in S}\wX^\prime\to \coprod_{z\in S}X_{\tilde S}%\to \wX
% \to X\]
% \[\xymatrix{
% \coprod_{z\in S}(V\amalg (\wX^\prime-\wX^\prime_z))\to \coprod_{z\in S}\wX^\prime\to 
% \coprod_{z\in S}X_{\tilde S}%\to \wX
% \to X\]
%
% is the required $\nu\cup\tf$-covering.
\end{proof}

\begin{example}
% Let $\zf$ be the topology $\nuf$ for $\nu$ being Zariski topology, then $\zf=\zar\cup\tf$.
% We call $\zf$ Zariski fibre topology.
$\zf=\zar\cup\tf$.
\end{example}

\section{Localisation theorem for \texorpdfstring{$\SH^{\fr}_{\zf}(\SmAff_S)$}{SHfrzf(SmAffS)}.}\label{section:LocalisationTheoremtf} %todo check ,S^1

In the section, 
generalizing the localisation theorems from 
\cite{Morel-Voevodsky,AyoubI,Hoyois-framed-loc,DKO:SHISpecZ}, %Morel-Voevodsky,Ayo-Loctheorem
we prove a localisation theorem for $\infty$-categories 
$\HH_\tau^\bullet(\SmAff_S)$ and $\HH^\fr_\tau(\SmAff_S)$ 
for some class of topologies $\tau$ on $\Sm_S$ that includes $\tf$, 
and 
%the the Zarisky fibre topology 
$\zf$, and $\nis$.  
% and
% for the $\infty$-categories $\HHd_\tau(S)$ and $\HH^\fr_\tau(S)$ for the second two topologies.
% The topology $\zf$ is the main example for our application.
%in particular
% For our application, the main example is the Zarisky fibre topology$\tau=\tf\cup\zar$.

% The proof is based on the principles of \cite{DKO:SHISpecZ} that covers $\tau=\tf$, 
% though this general argument is formally independent from the most complicated results of \cite{DKO:SHISpecZ}.
% % Indeed, \cite{DKO:SHISpecZ} provides for $\tau=\tf$ 
% % some results that are stronger and more complicated
% % than Localisation Theorem,
% % and we consider that they can be generalized as well for all topologies considered here.
% % While the mentioned improvements of Localisation Theorem are necessary for the aims of \cite{} 
% % %Nevertheless, since 
% % Localisation Theorem is enough for the current application, 
% % that is why we do not deal with that, 
% % and this helps us to keep the text being short. 
% Indeed, for $\tau=\tf$, \cite{DKO:SHISpecZ} provides 
% the pair of results \cite[Theorem 11.2, Theorem 11.3]{DKO:SHISpecZ},
% that combination is much stronger and more complicated
% than the discussed here Localisation Theorem.
%,and are necessary for the aims of \cite{DKO:SHISpecZ}

The proof is based on the principles of \cite{DKO:SHISpecZ} that covers the case $\tau=\tf$. 
In fact, in this case \cite{DKO:SHISpecZ} proves two results \cite[Theorem 11.2, Theorem 11.3]{DKO:SHISpecZ},
that in combination are much stronger and more involved 
than the discussed here Localisation Theorem.
Though we think that a similar enhancement of the Localisation Theorem 
holds for all topologies $\tau$ considered here,
we do not study it here, %mentioned iprovements 
%,
since classical form of Localisation Theorem is enough for the purposes of the text. 
This restriction helps us keep the text short and self-contained at the same time. 
The presented here general argument is formally independent at least from the most complicated results of \cite{DKO:SHISpecZ} 
including the proof of the Localisation Theorem, and from more classical proofs of Localisation Theorems.
%
% our
% It allows to avoid proving and studying of rigidity properties of presheaves, that occupies large part of the original reasoning,
% and replace it by the checking of the fibrancy property of $\infty$-categorical equivalences
% that appears being formally easier and is equivalent.
% Actually, we translate to the setting of the category $\Corr^\fr$ and $\Hfr(S)$ defined in \cite{ehksy} %recall the proof of 
% the Localisation Theorem for $\mathbf H_\tf(B)$ from \cite{DKO:SHISpecZ}.
% , i.e.
% the motivic homotopy category with respect to trivial fibre topology 
% from 
% \cite[Section 14.2]{DKO:SHISpecZ},
% combining with some results from the previous sections,
% and translate it to the setting of the category $\Corr^\fr$ and $\Hfr(S)$ defined in \cite{ehksy}.
% The argument is written in general form and recovers Localisation Theorems for $\mathbf H(B)$ and $\Hfr(B)$ and $\mathbf H_\tf(B)$. % at the same time
%
% \item[]
%
% Throughout this section $\bbS$ is subcategory of $\Sch$ such that for any $S\in\bbS$ all closed and open subschemes of $S$ are in $\bbS$.
Throughout this section 
\begin{itemize}
\item $\bbS$ is 
%one of the categories 
$\Aff$, or 
%$\Affns$, 
$\Afffns$.
% $\tau$ is a topology on $\EssSmAff_{\bbS}$,
\item $\Corr_{(-)}=\Corr(\Sch_\bbS)$ is a family of $\infty$-categories of radditive correspondences
that satisfies (AHP) in sense of \Cref{def:CorrAHP_second},
and satisfies the closed gluing on affine schemes.
\item $S\in \bbS$, and $Z$ is a closed subscheme of $S$. 
\end{itemize}

\subsection{}\label{subsect:SmSZ}
%$\SmAff^\fr_{S,Z}$ $\SmAff_{S}$ $\SmAff_{S-Z}$
Following \cite{DKO:SHISpecZ}, for a 
given scheme $X\in \SmAff_S$, 
denote $X_Z=X\times_S Z$, 
denote $X_{S-Z}=X\times_S (S-Z)$,
and denote by $X^h_Z=X^h_{X_Z}$ the henselisation of $X$ along $X_Z$.
% Following \cite{DKO:SHISpecZ}, for a 
% given $S\in\bbS$, 
% a closed subscheme $Z\subset S$,
% and a scheme $X\in \SmAff_S$, denote 
% $X_Z=X\times_S Z$, denote by
% $X^h_Z=X^h_{X_Z}$ the henselisation 
% of $X$ along $X_Z$, and by
% $X_{S-Z}=X\times_S (S-Z)$ the open complement.
%
% Following \cite{DKO:SHISpecZ}, for a 
% given $S\in\bbS$, 
% a closed subscheme $Z\subset S$,
% and a scheme $X\in \SmAff_S$, denote by
% $X_Z=X\times_S Z$ the fibre product with $Z$, by
% $X^h_Z=X^h_{X_Z}$ the henselisation 
% of $X$ along $X_Z$, and by
% $X_{S-Z}=X\times_S (S-Z)$ the open complement.
% Define the subcategory $\SmAff^\fr_{S,Z}$ in $\Aff_S$ spanned by schemes of the form $X^h_Z=X^h_Z$ for $X\in \SmAff_S$.

As in \cite{DKO:SHISpecZ},
we consider the subcategory $\Aff_{S,Z}$ in $\Aff_S$
spanned by the schemes of the form $X^h_Z$ for all $X\in \Aff_S$. 
% Then there is a functor 
% \begin{equation*}\label{eq:Afffunctor_XhZ}
% \Aff_S\to \Aff_{S,Z};X\mapsto X^h_Z.
% \end{equation*}
%
% The schemes of the form $X^h_Z$ for $X\in \Aff_S$
% form the subcategory in $\Aff_S$ denoted by $\Aff_{S,Z}$. 
% The subcategory $\Aff_{S,Z}$ in $\Aff_S$ is 
% spanned by schemes of the form $X^h_Z=X^h_Z$ for $X\in \Aff_S$,
For any subcategory $\mathcal S_S$ in $\Aff_S$,
define $\mathcal S_{S,Z}$ in $\Aff_{S,Z}$ 
as the subcategory in $\Aff_{S,Z}$ spanned by the objects $X^h_Z$, where $X\in\mathcal S_S$.
%\todo{changed}
% Similarly define $\mathcal S_{S,Z}$ for any subcategory $\mathcal S_S$ in $\Aff_S$.
For any $X,Y\in \mathcal S_{S,Z}$, we write $X\times_{S,Z}Y$ for $(X\times_S Y)^h_Z$.

\begin{definition}\label{def:wtaustau}
Let $\mathcal S_\bbS$ be a family of subcategories of $\Aff_\bbS$.% for $\bbS=\Schfns,\Afffns$.

For a topology $\tau$ on $\mathcal S_S$, 
define \emph{the topology $\wtau$} on $\mathcal S_{S,Z}$ 
as the weakest topology such that 
the functor $\mathcal S_S\to \mathcal S_{S,Z}\colon X\mapsto X^h_Z$ is continuous.

For a topology $\tau$ on $\mathcal S_Z$, 
define \emph{the topology $\stau$} on $\mathcal S_{S,Z}$ 
as the strongest topology such that 
the functor $\mathcal S_{S,Z}\to \Sm_Z\colon X\mapsto X_Z$ is continuous.
\end{definition}
%$\widetilde X^h_Z\to X^h_Z$ $\widetilde X\to X$ 
\begin{remark}
Note that the topology $\stau$ above is stronger then $\wtau$,
since $\Sch_S$ the base change functor $\mathcal S_S\to \mathcal S_Z$ is continuous with respect to $\tau$.
\end{remark}
\begin{lemma}\label{lm:stau}
Let $\mathcal S_\bbS$ be $\SmAff_\bbS$ or $\Smat_\bbS$, see \Cref{den:introductionnotation}. %, see \Cref{subsect:SZA1Z}
Let $\tau$ be a topology on $\mathcal S_Z$.
A morphism $\widetilde X\to X$ is a $\stau$-covering 
if and only if
$\widetilde X_Z\to X_Z$ is a $\tau$-covering.
\end{lemma}
\begin{proof}
The property of coverings in $\mathcal S_{S,Z}$ 
described in the lemma 
defines a pretopology $\stau^\prime$ on $\Sch_Z$, 
that defines a topology on $\Aff_Z$,
and moreover, the latter topology restricts to $\Smat_Z$.
%\mathcal S_{S,Z}$ for the given topology $\tau$ on $\mathcal S_{Z}$. 
%
% The property of coverings in $\mathcal S_{S,Z}$ 
% described in the lemma defines a topology on $\mathcal S_{S,Z}$ for the given topology $\tau$ on $\mathcal S_{Z}$. 
%
Any topology on $\mathcal S_{S,Z}$ such that 
the functor $\mathcal S_{S,Z}\to \mathcal S_{Z}$ is continuous is contained in $\stau^\prime$.
Hence $\stau^\prime=\stau$.
%this topology coincides with $\stau$.
\end{proof}
\begin{lemma}\label{lm:niszftfstauwtau}
Let $\bbS=\Afffns$, $\tau=\nuf$ be provided by \Cref{def:nuf}
for a subtopology $\nu$ of the \'etale topology on $\Aff_\bbS$. %\bbS=Aff
%, or $\Sch_{\Sch}$.
%
%and $Z$ be a closed subscheme of $S$.
%Given $\stau$-covering $$ in $\Sch_{S,Z}$, 
Then any $\stau$-covering in $\Aff_{S,Z}$ is a $\wtau$-covering.
\end{lemma}
\begin{proof}
We repeat the argument of \cite[Proposition 4.2]{DKO:SHISpecZ} that covers $\tau=\nu$ being Nisnevich topology. 
%We repeat the argument covering $\zf$ and $\tf$ as well.
Every $\stau$-covering $\wX^h_Z\to X^h_Z$ in $\Aff_{S,Z}$ 
is obtained from a morphism $f\colon \wX\to X$ in $\Sm_S$ 
such that $f_Z\colon \wX_Z\to X_Z$ is a $\tau$-covering in $\Aff_{Z}$.
Since $\wX,X\in \Sm_S$ and $f_Z$ is \'etale, the morphism $f$ is \'etale over $\wX_Z$.
Moreover, 
since $f_Z$ is a $\tau$-covering, it follows that $\wX\amalg X_{S-Z}\to X$ is a $\tau$-covering,
where $X_{S-Z}=X\times_S({S-Z})$,
% .
% Furthermore, if $f_Z$ is a $\tau$-covering, then $\wX\amalg X_{S-Z}$ is a $\tau$-covering
since $\tau$ is the strongest subtopology of the \'etale topology  
that coincides with $\nu$ on $\Aff_z$ for each $z\in S$.
Summarising, 
if $\wX^h_Z\to X^h_Z$ is an $\stau$-covering, 
then ${\wX}^h_Z\amalg (X_{S-Z})^h_Z\to X^h_Z$ is a $\wtau$-covering in $\Aff_{S,Z}$. 
This implies ${\wX}^h_Z\to X^h_Z$ is a $\wtau$-covering in $\Aff_{S,Z}$ since $(X_{S-Z})^h_Z=\emptyset$. 
\end{proof}

\subsection{}
For any category of correspondences $\Corr_S$ on $\EssSmAff_S$, we define 
$\Corr_{S,Z}$ as the subcategory of $\Corr_S$ spanned by the objects of $\EssSmAff_{S,Z}$.

\begin{definition}
A presheaf $F\in\Pre^\tr(\SmAff_{S,Z})$ %\in \Pre(\SmAff_{B,Z})$ or $ 
is called \emph{$\A^1$-invariant}
if $F((\A^1\times X)^h_Z)\simeq F(X^h_Z)$ for all $X\in \SmAff_S$.
\end{definition}
% Similarly, to the case of $\SmAff_S$ the subcategory, 
As mentioned in \cite[\S 4]{DKO:SHISpecZ} and \cite[\S 2.3.4]{DKO:SHISpecZ}
the equality $X^h_Z\times_{S,Z}Y^h_Z=(X\times Y)^h_Z$ defines the monoidal structure on $\Aff_{S,Z}$,
and the subcategory %$\Pre_{\A^1}(\SmAff_{B,Z})$ 
of $\A^1$-invariant presheaves in $\Pre^\tr(\SmAff_{S,Z})$ %\mathcal S
is reflective, %localising
where the localisation functor 
$L_{\A^1}$
is given by
$F\mapsto F((\Delta^\bullet_S\times_S -)^h_Z)$. As in \cite[\S 2.3.4]{DKO:SHISpecZ}
we use the notation $\Delta^\bullet_{S,Z}=(\Delta^\bullet_S)^h_Z$, so $F((\Delta^\bullet_S\times_S -)^h_Z)=F^{\Delta^\bullet_{S,Z}}$.
\begin{remark}
We note that 
$F^{\Delta^\bullet_{S,Z}}$ is not equivalent to the restriction of $F^{\Delta^\bullet_{S}}$ 
along the embedding functor $\EssSmAff_{S,Z}\to \EssSmAff_{S}$, because $(\Delta^\bullet_S)^h_Z\not\simeq \Delta^\bullet_S$.
Precisely, 
$F^{\Delta^\bullet_{S,Z}}(U^h_Z)\not\simeq F^{\Delta^\bullet_{S}}(U^h_Z)$. 

%\begin{remark}$
% In what follows, we deal and discuss 
% %Under the appropriate assumptions, 
% % there are commutativities
% % $\tilde i^!L_{\A^1_{S}}$
% %we have 
% the commutativities
% $(\tilde i^!F)^{\Delta^\bullet_{S,Z}}\simeq \tilde i^!(F^{\Delta^\bullet_{S}})$, 
% $(\tilde i_*F)^{\Delta^\bullet_{S}}\simeq \tilde i_*(F^{\Delta^\bullet_{S,Z}})$, 
% where $\tilde i$ is the morphisms of sites defined 
% the functor $\SmAff_{S}\to\SmAff_{Z}; X\mapsto X^h_Z$,
% and
% $(i_*F)^{\Delta^\bullet_{S,Z}}\simeq i_*(F^{\Delta^\bullet_{Z}})$, 
% where $i$ relates to the functor $\SmAff_{S}\to\SmAff_{Z}; X\mapsto X^h_Z$,
% that  holds under the appropriate assumptions.
% %see \Cref{}, and \Cref{}.
% %The commutativity relations with the inverse and direct image functors are discussed later in the section.
%\end{remark}
\end{remark}
%\mathcal S
% $L_{\A^1}\colon\Pre_{\A^1}(\SmAff_{B,Z})\to\Pre_{\A^1}(\SmAff_{B,Z})$ is defined by \[L_{\A^1}F(X)=F((\Delta^\bullet_S\times_S X)^h_Z),\] 
% similarly for $\Pre^\fr_{\A^1}(\SmAff_{B,Z})$.

%Let $\tau$ be a topology in the category $\Sch_S$ that contains the trivial fibre ($\tf$) topology.

%For an open immersion $U\to V$ in $\SmAff_S$ and a presheaf $F$ on $\SmAff_S$ denote $$

% \begin{definition}\label{def:CorrHU}
% %Given a schemes $S\in\Sch$, a closed subschemes
% Let $S\in \Sch$, and $Z$ be a closed subscheme of $S$.
% A category of correspondences $\Corr_S$ on $\EssSmAff_S$ has the property (HU) if 
% $\Corr(U^h_Z,X^h_Z)\simeq \Corr(U^h_Z,X)$
% for all $U,X\in\EssSmAff_S$.

% A family of $\infty$-categories of correspondences $\Corr_{(-)}$ on $\EssSmAff_{(-)}$ 
% satisfies the property (HU), whenever $\Corr_{S}$ satisfies (HU) for each base scheme $S$ in the family,
% and any closed subscheme $Z$.
% \end{definition}

\begin{lemma}\label{lm:CorrHU}%def:CorrHU
%Given a schemes $S\in\Sch$, a closed subschemes
% Let $\Corr_{(-)}$ be 
% a family of $\infty$-categories of correspondences on $\EssSmAff_{(-)}$ 
% over $\Aff$ 
% that 
% satisfies the property (AHP)(2). 
% Let $S\in \bbS$, and $Z$ be a closed subscheme of $S$.
% Then
There is the isomorphism
$\Corr_S(U^h_Z,X^h_Z)\simeq \Corr_S(U^h_Z,X)$
for any $U,X\in\EssSmAff_S$.
\end{lemma}
\begin{proof}
For any $U^h_Z$-scheme $E$ there is an isomorphism $E^h_{U\times_S Z}\simeq E^h_Z$, where $E$ is considered as $S$-scheme at the right side. So write $E^h_Z$ for $E^h_{U\times_S Z}$. 
Further, there is an isomorphism of $U^h_Z$-schemes
$(X^h_Z\times_S U^h_Z)^h_{Z}\simeq(X\times_S U^h_Z)^h_{Z}$.
Then due to (AHP)(2) 
\begin{equation*}\begin{array}{lcl}
\Corr_S(U^h_Z,X^h_Z)&\simeq&
\Corr_{U^h_Z}(U^h_Z,X^h_Z\times_S U^h_Z)\\&\simeq&
\Corr_{U^h_Z}(U^h_Z,(X\times_S U^h_Z)^h_Z)\\&\simeq& \Corr_{U^h_Z}(U^h_Z,X\times_S U^h_Z)\\
&\simeq&
\Corr_S(U^h_Z,X).
\end{array}\end{equation*}    
\end{proof}

% \todo{added}

\begin{lemma}\label{lm:VectXinSmat}
For any 
% $S\in\Aff$, $Z\hookrightarrow S$, 
$X\in\SmAff_{S,Z}$, 
there is a vector bundle $\widetilde X$ over $X$ 
such that $\widetilde X\in\Smat_{S,Z}$.
%
% Consequently, 
% for any $X$ in $\SmAff_{*}$, there is an $\A^1$-equivalence $\widetilde X\to X$ such that $\widetilde X$ is in $\widetilde X\in\Smat_*$.
\end{lemma}
\begin{proof}
    The claim follows by the argument of \cite[Lemma 5.4]{DKO:SHISpecZ}.
    Indeed, %since $X\in \SmAff_{*}$,
    there is %a vector bundle 
    $\widetilde X$ %over $X$ 
    such that
    $\widetilde X\times_X T_X\simeq {\bf 1}^{N}_X$, for some $N\in\mathbb Z_{\geq 0}$, 
    where $T_X$ and ${\bf 1}^{N}_X$ are tangent and trivial bundles.%
%    $T_X$ is a tangent bundle and ${\bf 1}^{N}_X$ is the trivial bundle of rank $N$
.
\end{proof}
\begin{corollary}\label{lm:PretrASmAffsimeqPretrASmat}
%For any 
%$S\in\Aff$, $Z\hookrightarrow S$,
% t
The adjunction $r\colon \SPreradtrA(\SmAff_{S,Z})\leftrightarrows \SPreradtrA(\Smat_{S,Z})\colon l$, 
where $l\dashv r$, is an equivalence.
%$\PretrA(\SmAff_{S,Z})\simeq \PretrA(\Smat_{S,Z})$.
\end{corollary}
\begin{proof}
    Since the functor $\Corrtr(\Smat_{S,Z})\to \Corrtr(\SmAff_{S,Z})$ is fully faithful,
    the adjunction $l\dashv r$, is a coreflection.
    By \Cref{lm:VectXinSmat}
    for any $X$ in $\SmAff_{{S,Z}}$, there is an $\A^1$-equivalence $\widetilde X\to X$ such that $\widetilde X\in\Smat_{S,Z}$.
    Hence $r$ is conservative, and the claim follows.
\end{proof}
\begin{remark}\label{rem:PretrSmAffsimeqPretrSmat}
The argument of 
\cite[Proposition 5.8]{DKO:SHISpecZ}
proves even %the equivalence %for the categories of presheaves
$\Pretr(\SmAff_{S,Z})\simeq \Pretr(\Smat_{S,Z})$.
\end{remark}
% \todo{end added}

\subsection{}
We proceed with the localisation property for the homotopy categories of sheaves on $\SmAff_{S,Z}$, $\SmAff_{S}$, $\SmAff_{S-Z}$.
%and similarlty for \SmAff
%We start with the following lemma.
%
% Throughout this subsection we consider
% a family of $\infty$-categories of correspondences $\Corr(\EssSmAff_\Aff)$, 
% and a topology $\tau$ on $\EssSmAff_{\Aff}$ 
% such that 
% %$\Corr_{(-)}$ and $\tau$ 
% %are continuous %satisfy (FinE) 
% %in sense of \Cref{def:FinECorr,def:FinProp},
% %and 
% $\Corr_{(-)}$ satisfies (AHP)(2) in sense of \Cref{def:CorrAHP_second}. %Sch
% Then we consider some $S\in \Aff$\todo{because of (AHP)}, and a closed subscheme $Z$ in $S$. 
%\todo{added}
%
%
Throughout this section 
$\tau$ is a topology on $\EssSmAff_{\bbS}$,
and speaking about the topology $\tau$ on $\SmAff_{S,Z}$ we mean the topology $\wtau$.
%

% Recall that according to 
% \Cref{den:introductionnotation}
% $\SPreradtrd(-)=\SPreradd(\Corr(-))$ is the $\infty$-category of pointed radditive presheaves on the respective $\infty$-category of correspondences.

\subsubsection{}\label{sect:Loc_SZSS-Z:istarjsharp}
We prove localisation theorem for functors $\tilde i^*$, $\tilde i_*$, $j^*$, $j_\#$ on unpointed $\infty$-categories.

\begin{lemma}\label{lm:tildei**jsharp*preserveA1loctauloc:zero}

There are well defined adjunctions $\tilde i^*\dashv \tilde i_*$, $j_{\#}\dashv j^*$,
\begin{equation}\label{eq:PreSZSS-Z:istarjsharp}
\SPreradtr(\SmAff_{S,Z})
\rightleftarrows 
\SPreradtr(\SmAff_{S}) 
\rightleftarrows 
\SPreradtr(\SmAff_{S-Z})
\end{equation}
\[\begin{array}{ll}
\tilde i^*(h^\tr(X))=h^\tr(X^h_Z) &
\tilde i_*(F)(X)=F(X^h_Z) \\
j^*(F)(X)=F(X) &
j_{\#}(h^\tr(X))=h^\tr(X)
.\end{array}\]
\end{lemma}
\begin{proof}
The adjunction $\tilde i^*\dashv \tilde i_*$
is given by the inverse and the direct image funcotrs along 
the functor $\Corr(\SmAff_{S})\to \Corr(\SmAff_{S,Z})$; $X\mapsto X^h_Z$.

The adjunction $j_{\#}\dashv j^*$
is given by the inverse and the direct image funcotrs along 
the functor $\Corr(\SmAff_{S-Z})\to \Corr(\SmAff_{S})$; $X\mapsto X$.
\end{proof}

% The functor $\SmAff_{S}\to \SmAff_{S,Z}; X\mapsto X^h_Z$ induces the adjunction $\tilde i^*\dashv \tilde i_*$
% \begin{equation}\label{eq:istarul:SPreradtrSSZ}
% \tilde i^*\colon \SPreradtr(\SmAff_{S}) \rightleftarrows \SPreradtr(\SmAff_{S,Z})\colon \tilde i_*,
% \end{equation}
% such that $\tilde i^*(h^\tr_S(X))=h^\tr_{S,Z}(X^h_Z)$, $\tilde i_*(F)(U)\simeq F(U^h_Z)$.
Since $\tilde i^*$ in \eqref{eq:PreSZSS-Z:istarjsharp} preserves $\tau$-local equivalences, and 
$\tilde i_*$ in \eqref{eq:PreSZSS-Z:istarjsharp} preserves $\tau$-sheaves,
they induce the adjunction 
\[\tilde i^*\colon \Shv_{\Sigma,\wtau}^\tr(\SmAff_{S}) \rightleftarrows \Shv_{\Sigma,\tau}^\tr(\SmAff_{S,Z})\colon \tilde i_*.\]
% that form the adjunction 
% $\tilde i^*\dashv\tilde i_*$ such that $\tilde i^*(h^\tr_S(X))=h^\tr_{S,Z}(X^h_Z)$.

\begin{lemma}\label{lm:ilsXhZ}
Suppose $\tau\supset\tf$.
Let $X\in\SmAff_S$. 
Then 
%correctedaftersubmision
%\todo{:corrected:2052024:words:}
there is a $\tau$-local equivalence
% \[
% \begin{array}{lcl}
% j_# j^*(h^\tr(X))
% \simeq 
% h^\tr(X-X_Z)
% \\
% \tilde i_*\tilde i^*(h^\tr(X))
% &\simeq& 
% \cofib(h^\tr(X-X_Z)\to h^\tr(X)))
% \end{array}\]
\[
\tilde i_*\tilde i^*(h^\tr(X))
\simeq
\cofib(h^\tr(X-X_Z)\to h^\tr(X)))
.\]
% in $\Shv^\tr_{\Sigma,\tau}(\SmAff_S)$.
\end{lemma}
\begin{proof}
For any $U\in\SmAff_{S-Z}$, % such that $U\times_{S}{Z}\cong\emptyset$, 
there is the isomorphism $U^h_Z\cong\emptyset$, 
and consequently, 
\[\tilde i_*\tilde i^* h^\tr(X)(U)\simeq *.\]
Hence the unit 
$h^\tr(X)\to\tilde i_*\tilde i^*(h^\tr(X))$
of the adjunction $\tilde i^*\dashv\tilde i_*$
induces the morphism
\begin{equation}\label{eq:cofhXXZiisulX}\cofib(h^\tr(X-X_Z)\to h^\tr(X)))\to\tilde i_*\tilde i^*(h^\tr(X)),\end{equation}
and moreover, \eqref{eq:cofhXXZiisulX} is an isomorphism on $\SmAff_{S-Z}$.
Let $U\in\SmAff_S$ be such that $U\times_S Z\not\cong\emptyset$, then $h^\tr(X-X_Z)(U^h_Z)\simeq \emptyset,$
and consequently,
\[\begin{array}{lcl}
\cofib(h^\tr(X-X_Z)\to h^\tr(X)))(U^h_Z)&\simeq&
h^\tr(X)(U^h_Z)\\&\simeq&
\tilde i^*(h^\tr(X)) (U^h_Z)\\&\simeq&
\tilde i_*\tilde i^*(h^\tr(X))(U^h_Z).
\end{array}\]
% \[
% \cofib(h^\tr(X-X_Z)\to h^\tr(X)))(U^h_Z) \simeq h^\tr(X)(U^h_Z)
% .\]
So \eqref{eq:cofhXXZiisulX} is an isomorphism on non-empty objects of $\SmAff_{S,Z}$.
% \[\begin{array}
% \cofib(h^\tr(X-X_Z)\to h^\tr(X)))(U)&\simeq&
%
% \tilde i_*\tilde i^*(h^\tr(X))(U)\\&\simeq&
% \tilde i^*(h^\tr(X)) (U^h_Z)\\&\simeq&
% h^\tr(X) (U^h_Z)
% \simeq 
% \end{arra\]
%
% On the other side,
% For any $U\in\SmAff_S$,
% \[h^\tr(X-X_Z)(U^h_Z)\simeq \emptyset,\]
% since 
% either $U^h_Z=\emptyset$, when $U\times_S Z\cong\emptyset$,
% or 
% when $U$ is such that $U\times_S Z\not\cong\emptyset$,
% % For any $U\in\SmAff_S$ such that $U\times_S Z\not\cong\emptyset$,
% % \[h^\tr(X-X_Z)(U^h_Z)\simeq \emptyset,\]
% % and consequently, for any $U$
%
% \[\begin{array}
% \cofib(h^\tr(X-X_Z)\to h^\tr(X)))(U)&\simeq&

% \tilde i_*\tilde i^*(h^\tr(X))(U)\\&\simeq&
% \tilde i_* h^\tr(X^h_Z) (U)\\&\simeq&
% h^\tr(X^h_Z) (U^h_Z)
% \simeq 
% \end{arra\]
% consequently,
% .
Thus since $\tau\supset\tf$, it follows that \eqref{eq:cofhXXZiisulX} 
is a $\tau$-local equivalence.
% Then the unit of the adjunction $\tilde i^*\dashv\tilde i_*$,
% gives the commutative diagram %morphism 
% \begin{equation*}\label{eq:htrtauXiusilshtrtauX}
% \xymatrix{
% h^\trd(X-X_Z)\ar[r]\ar[d]& h^\trd(X)\ar[d] 
% \\
% \phantom{.}*\ar[r] & \tilde i_*\tilde i^*(h^\trd(X)).
% }
% \end{equation*}
% that leads to the morphism
% \begin{equation}\label{eq:cof_S-Z_Xtr->iissXtr}\cofib(h^\trd(X-X_Z)\to h^\trd(X))\to \tilde i_*\tilde i^*(h^\trd(X))\end{equation}
% in $\Shv^\trd_\tau(\SmAff_S)$.
% Denote $G=\tilde i^*(h^\trd(X))$, and denote by $iG$ the left side of \eqref{eq:cof_S-Z_Xtr->iissXtr}.
% Suppose $U_Z=\emptyset$;
% then $iG(U)$ is contractible.
% Then since $\tau\supset\tf$ by the assumption, 
% it follows that $iG(U)\cong iG(U^h_Z)$ for any $U\in\SmAff_S$.
% So, when $U_Z\neq \emptyset$, $iG(U)\cong iG(U^h_Z)\cong h^\trd_\tau(X)(U^h_Z)\cong h^\trd_\tau(X^h_Z)(U^h_Z)$
% by \Cref{lm:CorrHU}.
% Thus $iG(U)\cong \tilde i_*(G)(U)$ for all $U\in\SmAff_S$, and \eqref{eq:cof_S-Z_Xtr->iissXtr} is the isomorphism. 
\end{proof}

\begin{lemma}\label{lm:tildei**jsharp*preserveA1loctauloc}

% (0)
% There are well defined adjunctions 
% \begin{equation}\label{eq:PreSZSS-Z:istarjsharp}
% \SPreradtr(\SmAff_{S,Z})
% \rightleftarrows 
% \SPreradtr(\SmAff_{S}) 
% \rightleftarrows 
% \SPreradtr(\SmAff_{S-Z})
% \end{equation}
% \[\begin{array}{ll}
% \tilde i^*(h^\tr(X))=h^\tr(X^h_Z) &
% \tilde i_*(F)(X)=F(X^h_Z) \\
% j^*(F)(X)=F(X) &
% j_{\#}(h^\tr(X))=h^\tr(X_{S-Z})
% ,\end{array}\]
% $\tilde i^*\dashv \tilde i_*$,
% $j_#\dashv j^*$.

% (1) 
% The functors $\tilde i^*$, $j^*$, $j_#$ in \eqref{eq:PreSZSS-Z:istarjsharp}
% preserve $\tau$-local equivalences in sense of $\wtau$-topology on $\SmAff_{S,Z}$;
% and if $\tau\supset \tf$, $\tilde i_*$ preserves $\tau$-local equivalences, 
% so there are functors 
% \begin{equation}\label{eq:ShvSZSS-Z:istarjsharp}
% \Shv_{\Sigma,\wtau}^\tr(\SmAff_{S,Z}) \rightleftarrows \Shv_{\Sigma,\tau}^\tr(\SmAff_S)\rightleftarrows \Shv_{\Sigma,\tau}^\tr(\SmAff_{S-Z}),
% \end{equation}
% induced by \eqref{eq:PreSZSS-Z:istarjsharp} 
% via the localisation with respect to $\tau$-local equivalences.

% (2) 
% The functors $\tilde i^*$, $j^*$, $j_#$ in \eqref{eq:PreSZSS-Z:istarjsharp}
% preserve $\A^1$-equivalences. If $\tau\supset \tf$,
% then the functor $\tilde i_*$ in \eqref{eq:ShvSZSS-Z:istarjsharp} preserves $\A^1$-equivalence.

Suppose $\tau\supset\tf$.

(0)
%todo:corrected:2052024:superindexes:tr
%todo:corrected:2052024:=:\cong
There are the natural equivalences
\[\begin{array}{ll}
\tilde i^*(h^\tr(X))\simeq h^\tr(X^h_Z) &
% \tilde i_*(h^\tr(X^h_Z))\simeq\cofib(h^\tr(X-X_Z)\to h^\tr(X)) 
\\
j^*(h^\tr(X))\simeq h^\tr(X-X_Z) &
j_{\#}(h^\tr(X))\simeq h^\tr(X)
.\end{array}\]

(1) 
The functors $\tilde i_*$, $\tilde i^*$, $j^*$, $j_{\#}$ in \eqref{eq:PreSZSS-Z:istarjsharp}
preserve $\tau$-local equivalences.
% and there are adjunctions 
% \begin{equation}\label{eq:ShvSZSS-Z:istarjsharp}
% \Shv_{\Sigma,\wtau}^\tr(\SmAff_{S,Z}) \rightleftarrows \Shv_{\Sigma,\tau}^\tr(\SmAff_S)\rightleftarrows \Shv_{\Sigma,\tau}^\tr(\SmAff_{S-Z}),
% \end{equation}
% induced by \eqref{eq:PreSZSS-Z:istarjsharp} 
% via the localisation with respect to $\tau$-local equivalences.

(2) 
The functors 
$\tilde i^*$, $j^*$, $j_{\#}$ 
in \eqref{eq:PreSZSS-Z:istarjsharp}
preserve $\A^1$-equivalences. 
The functor induced by
$\tilde i_*$ 
% in \eqref{eq:PreSZSS-Z:istarjsharp}
on $\Shv^\tr_{\Sigma,\tau}(-)$. %\SmAff_S
preserve $\A^1$-equivalences. 

% There are adjunctions 
% \begin{equation}\label{eq:ShvSZSS-Z:istarjsharp}
% \Shv_{\Sigma,\wtau}^\tr(\SmAff_{S,Z}) \rightleftarrows \Shv_{\Sigma,\tau}^\tr(\SmAff_S)\rightleftarrows \Shv_{\Sigma,\tau}^\tr(\SmAff_{S-Z}),
% \end{equation}
% induced by \eqref{eq:PreSZSS-Z:istarjsharp} 
% via the localisation with respect to $\tau$-local equivalences.

% By \Cref{lm:tildei*!j**preserveA1loctauloc}
Consequently, 
there are adjunctions 
\begin{equation}\label{eq:ShvSZSS-Z:istarjsharp}
\HHtr_{\Sigma,\wtau}(\SmAff_{S,Z}) 
\rightleftarrows 
\HHtr_{\Sigma,\tau}(\SmAff_S)
\rightleftarrows 
\HHtr_{\Sigma,\tau}(\SmAff_{S-Z}),
\end{equation}
induced by \eqref{eq:PreSZSS-Z:istarjsharp} 
via the localisation with respect to $\tau$-local equivalences, and $\A^1$-equivalences.

\end{lemma}
\begin{proof}
Point (0) is provided by \Cref{lm:tildei**jsharp*preserveA1loctauloc:zero,lm:ilsXhZ}.
%todo:corrected:2052024:argument:secondparagraphisinserted
% Points (1) and (2) follows from (0).

Points (1) and (2) for $\tilde i^*$, $j^*$, $j_{\#}$
follows by Point (0). %\eqref{}.

We proceed with Point (1) for $\tilde i_*$.
The functor 
$\tilde i_*$ 
preserves $\tf$-equivalences
because the functor $X\mapsto X^h_Z$
preserves $\tf$-points.
Hence 
$\tilde i_*$ induces functor
$\HHtr_{\Sigma,\wtau}(\SmAff_{S,Z}) 
\rightarrow 
\HHtr_{\Sigma,\tau}(\SmAff_S)$
via the localisation with respect to $\tf$-local equivalences.
The induced functor preserves $\tau$-local equivalences by \Cref{lm:ilsXhZ}.
Since $\tau\supset \tf$, it follows that 
$\tilde i_*$ preserves $\tau$-local equivalences.

Point (2) for $\tilde i_*$ 
follows from %Point (2) for 
% $\tilde i^*$, $j^*$, $j_{\#}$ and
\Cref{lm:ilsXhZ}.
\end{proof}

% By \Cref{lm:tildei*!j**preserveA1loctauloc}
% there are adjunctions 
% \begin{equation}\label{eq:ShvSZSS-Z:istarjsharp}
% \HHtr_{\Sigma,\wtau}(\SmAff_{S,Z}) 
% \rightleftarrows 
% \HHtr_{\Sigma,\tau}(\SmAff_S)
% \rightleftarrows 
% \HHtr_{\Sigma,\tau}(\SmAff_{S-Z}),
% \end{equation}
% induced by \eqref{eq:PreSZSS-Z:istarjsharp} 
% via the localisation with respect to $\tau$-local equivalences, and $\A^1$-equivalences.

\begin{proposition}\label{prop:HtsSmBcZSmBSmU:istartjsharp}
% Consider the pair of adjunctions $\tilde i_*\dashv \tilde i^!$, $j^*\dashv j_*$, 
% \[
% \HH_{\Sigma,\wtau}^\trd(\SmAff_{S,Z}) \rightleftarrows \HH_{\Sigma,\tau}^\trd(\SmAff_S)\rightleftarrows \HH_{\Sigma,\tau}^\trd(\SmAff_{S-Z})
% ,\]
% provided by \Cref{lm:tildei*!j**preserveA1loctauloc,lm:adjunctionilsiush}.
For any $F\in \HHtr_{\Sigma,\tau}(\SmAff_S)$,
the square
\begin{equation*}\label{eq:SZlocproppbsq_sharp}\xymatrix{
\tilde i_*\tilde i^*F& F\ar[l]\\
{*}\ar[u]\ar[r] & j_\#j^*F\ar[u]\ar[l]
}\end{equation*}
is a pushout.
\end{proposition}
\begin{proof}
The claim follows by \Cref{lm:ilsXhZ,lm:tildei**jsharp*preserveA1loctauloc}.
% For any $F\in \Shv^\trd_{\Sigma,\tau}(\SmAff_S)$, 
% \[\tilde i_*\tilde i^!F \simeq 
% \mathrm{fib}(F(X^h_Z)\to F(X^h_Z-X_Z))\simeq 
% \mathrm{fib}(F(X)\to F(X-X_Z))\simeq 
% \mathrm{fib}(F(X)\to j_*j^*F(X)).\]
% Hence the corresponding square is pullback in $\Shv^\trd_{\Sigma,\tau}(\SmAff_S)$. 
% The claim for $\HH^\trd_{\Sigma,\tau}(\SmAff_S)$ follows by \Cref{lm:tildei*!j**preserveA1loctauloc}.
\end{proof}

\subsubsection{}
We prove localisation theorem for functors $\tilde i^!$, $\tilde i_*$, $j^*$, $j_*$ on pointed $\infty$-categories.

Recall that according to 
\Cref{den:introductionnotation}
$\SPreradtrd(-)=\SPreradd(\Corr(-))$ is the $\infty$-category of pointed radditive presheaves on the respective $\infty$-category of correspondences.
Then
\Cref{lm:ilsXhZ} holds for the $\infty$-category of pointed radditive $\tau$-sheaves $\Shv^\trd_{\Sigma,\tau}(\SmAff_S)$ as well.

\begin{lemma}\label{lm:tildei*!j**preserveA1loctauloc}
%\todo{Affine $S$ . notation in the diagram} 

% Let $\Corr_{(-)}$ be a family of $\infty$-categories of correspondences on $\EssSmAff_{(-)}$ over $\Aff_S$, 
% and $\tau$ be a topology on $\EssSmAff_S$ 
% such that $\Corr_{(-)}$ and $\tau$ satisfy (FinE) in sense of \Cref{def:FinECorr,def:FinProp},
% and $\Corr_{(-)}$ satisfies (AHP)(2) in sense of \Cref{def:CorrAHP_second}. %Sch
% Consider any $S\in \Aff$, and a closed subscheme $Z$ in $S$. 
(0)
% For any $S\in Sch$, a closed subscheme $Z$ in $S$, 
% an $\infty$-category of correspondences $\Corr_S$ on $\EssSmAff_S$, 
% and a topology $\tau$ on $\EssSmAff_S$, %Sch
There are well defined functors 
%consider the functors
\begin{equation}\label{eq:PreSZSS-Z}\SPretrd(\SmAff_{S,Z})\rightleftarrows \SPretrd(\SmAff_{S}) \rightleftarrows \SPretrd(\SmAff_{S-Z})\end{equation}
\[\begin{array}{ll}\tilde i^!(F)(X^h_Z)=\fib (F(X^h_Z)\to F(X^h_Z-X_Z)) &
\tilde i_*(F)(X)=F(X^h_Z) \\
j^*(F)(X)=F(X) &
j_*(F)(X)=F(X_{S-Z})
\end{array}\]
% \begin{equation}\label{eq:PreSZSS-Z}\Pre^\tr(\SmAff_{S,Z})\rightleftarrows \Pre^\tr(\SmAff_{S}) \rightleftarrows \Pre^\tr_(\SmAff_{S-Z})\end{equation}
% \[\begin{array}{ll}\tilde i^!(F)(X^h_Z)=\fib (F(X^h_Z)\to F(X^h_Z-X_Z)) &
% \tilde i_*(F)(X)=F(X^h_Z) \\
% j^*(F)(X)=F(X) &
% j_*(F)(X)=F(X_{S-Z})
% .\end{array}\]
that preserve the subcategories $\SPreradtrd(-)$.

(1) 
All the functors $\tilde i^!$, $\tilde i_*$, $j^*$, $j_*$
in \eqref{eq:PreSZSS-Z}
preserve 
$\tau$-sheaves; %in sense of $\wtau$-topology on $\SmAff_{S,Z}$;
so there are %the adjunctions of 
functors %$\SmAff_{S,Z}$
\begin{equation}\label{eq:ShvSZSS-Z}
\Shv_{\Sigma,\wtau}^\trd(\SmAff_{S,Z}) \rightleftarrows \Shv_{\Sigma,\tau}^\trd(\SmAff_S)\rightleftarrows \Shv_{\Sigma,\tau}^\trd(\SmAff_{S-Z}),
\end{equation}
%todo \Sh \Shv
induced by \eqref{eq:PreSZSS-Z} by the restriction to the subcategories of sheaves.
%If $\tau\supset\tf$, the latter functors from the adjunctions $\tilde{i}_*\dashv\tilde{i}^!$, $j^*\dashv j_*$.

(2) 
%If, in addition, $\tau\supset \tf$,
%then 
The functors $\tilde i_*$, $j^*$, $j_*$ 
in \eqref{eq:PreSZSS-Z}
%todo \Sh \Shv
% induced by \eqref{eq:PreSZSS-Z} by the restriction to the subcategories of sheaves
% $\Shv_{\wtau}(\SmAff_{S,Z})$, $\Sh_\tau(\SmAff_S)$, $\Sh_{\tau}(\SmAff_{S-Z})$
preserve $\A^1$-invariant presheaves. %objects
If $\tau\supset \tf$,
then the functor $\tilde i^!$ in \eqref{eq:ShvSZSS-Z} preserves $\A^1$-invariant shaves.
% , %objects
% %inducing by the restriction in its turn the adjunctions
% and there are adjunctions
% \[
% \HH_\wtau^\tr(\SmAff_{S,Z}) \rightleftarrows \HH_\tau^\tr(\SmAff_S)\rightleftarrows \HH_{\tau}^\tr(\SmAff_{S-Z})
% \]
% induced by \eqref{eq:PreSZSS-Z} by the restriction.
%
% \[
% \Shv_{\wtau,\A^1}^\tr(\SmAff_{S,Z}) \rightleftarrows \Shv_{\tau,\A^1}^\tr(\SmAff_S)\rightleftarrows \Shv_{\tau,\A^1}^\tr(\SmAff_{S-Z})
% .\]
%(preserve $\A^1$-equivalences since commute with $L_{\A^1}$)
\end{lemma}
\begin{proof}
Functionalities of the scheme operations in the definition
imply that the formulas in point (0) define functors $\tilde{i}_*$, $\tilde{i}^!$, $j^*$, $j_*$.
Further, the claims (1) and (2) for $j^*$ and $j_*$ are well known.

We are going to prove (1) for $\tilde i^!$.
% that %To prove that 
% $\tilde i^!$ preserves $\tau$-sheaves. 
Let %consider 
$F\in \Shv_{\Sigma,\tau}^\trd(\SmAff_S)$. 
By definition the topology $\wtau$ %$\widetilde{X^h_Z}\to X^h_Z$ 
is generated by coverings in the category $\SmAff_{S,Z}$ of the form $\widetilde{X}^h_Z\to X^h_Z$
for a $\tau$-covering $\widetilde{X}\to X$ in $\SmAff_S$. 
Since for any $\tau$-covering $\widetilde{X}\to X$,
$\widetilde X\to X$ and $\widetilde X_{S-Z}\to X_{S-Z}$ are $\tau$-coverings,
it follows that $i^!(F)$ is $\wtau$-sheaf.
%todo
Thus
% that %To prove that 
$\tilde i^!$ preserves $\tau$-sheaves. 

% Since $\tf$-topology is contained in $\tau$, it follows that 
% for any \'etale affine neighbourhood $\widetilde X^\prime\to \widetilde X$ of $\widetilde X_Z$ the morphism 
% $\widetilde X^\prime\amalg \widetilde X_{S-Z}\to \widetilde X$ is a $\tau$-covering, and hence
% $R_{}$
% \[\Pre^\fr( R_{\widetilde X\to X}, 
% \fib(F(\widetilde X)\to F(\widetilde X\times_{S-Z}), \fib(F(\)\]

We proceed with (2) for $\tilde i^!$.
If $F$ is $\A^1$-invariant, and $\tf\subset \tau$, then for any $X\in \SmAff_S$
\[\begin{array}{lclc}
\fib(F((\A^1\times X)^h_Z)&\to& F((\A^1\times X)^h_Z\times_S (S-Z))&\simeq\\ 
\fib(F(\A^1\times X)&\to& F((\A^1\times X)\times_S (S-Z))&\simeq \\
\fib(F(X)&\to& F(X\times_S (S-Z))&\simeq\\
\fib(F(X^h_Z)&\to& F(X^h_Z\times_S (S-Z)).
\end{array}\]
Thus $i^!F$ is $\A^1$-invariant.

%We proceed with 
To prove that claims (1) and (2) for $\tilde i_*$ 
consider $F\in \Shv^\trd_{\Sigma,\wtau}(\SmAff_{S,Z})$. 
The presheaf $\tilde i_*(F)\in \SPreradtrd(\SmAff_S)$ is a $\tau$-sheaf, % on $\SmAff_S$ 
because for any $\tau$-covering $\widetilde X\to X$ in $\SmAff_S$ the morphism $\widetilde X^h_Z\to X^h_Z$ is a $\tau$-covering in $\SmAff_{S,Z}$.
If $F$ is $\A^1$-invariant, then for any $X\in \SmAff_S$,
$\tilde i_*F(\A^1\times X)=F((\A^1\times X)^h_Z)\simeq F(X^h_Z)=\tilde i_*F(X)$.
%
% To complete the proof we have to show that there is an adjunction $\tilde{i}_*\dashv\tilde{i}^!$ in \eqref{eq:ShvSZSS-Z}, when $\tau\supset\tf$.
% Let $G\in\Shv_\tau^\tr(\SmAff_{S,Z})$ be the sheaf associated with the presheaf represented by $X^h_Z\in\SmAff_{S,Z}$, where $X\in\SmAff_S$.
% \todo{do we need that $\tau$ is subcanonical or something like that}
% Denote by $iG = \cofib(h^\tr_S(X-X_Z)\to h^\tr(X))\in \Shv_\tau^\tr(\SmAff_{S})$.
% Then $iG(U)$ is contractible for any $U$ such that $U_Z=\emptyset$, 
% and 
% since $iG$ is a $\tau$-sheaf by the definition, and $\tau\supset\tf$ by the assumption, 
% it follows that $iG(U)\cong iG(U^h_Z)$ for any $U\in\SmAff_S$.
% Further, when $U_Z\neq \emptyset$, $iG(U)\cong iG(U^h_Z)\cong h^\tr_S(X)(U^h_Z)\cong h^\tr_S(X^h_Z)(U^h_Z)$.
% Summarising we conclude using (FinE) for $\Corr_{(-)}$ that for any $U\in \SmAff_S$, 
% \[
% iG(U)\cong \begin{cases}
% h^\tr_S(X^h_Z)(U^h_Z)\cong G(U^h_Z) , \quad U_Z\neq\emptyset,\\
% *\cong G(U^h_Z), \quad U_Z=\emptyset
% \end{cases}
% .\]
% Thus $\tilde i_*(G)\cong iG$. 
% Since $\tau\supset\tf$, it follows that
% $\tilde i^!F(X^h_Z)\cong \fib(F(X)\to F(X-X_Z))\cong \Map(iG,F)$.
% for any $F\in \Shv_\tau(\SmAff_S)$,
% and there are the isomorphisms
% \[\Map(\tilde i_* G,F)\cong \tilde i^!F(X^h_Z) \cong \Map(G,\tilde i^!F).\]
% Since representable objects
% generate 
% the category $\Shv_\tau(\SmAff_S)$, we get
% the adjunction.
%
% by the sequence of isomorphisms
% \[
% \Map(\tilde{i}(F),G)\simeq
% \Map(\tilde{i}(F),G)
% \]
\end{proof}

\begin{lemma}\label{lm:adjunctionilsiush}
If $\tau\supset\tf$, the functors \eqref{eq:ShvSZSS-Z} 
from the adjunctions $\tilde{i}_*\dashv\tilde{i}^!$, $j^*\dashv j_*$,
and consequently, 
there are the adjunctions
\[
\HH_{\Sigma,\wtau}^\trd(\SmAff_{S,Z}) \rightleftarrows \HH_{\Sigma,\tau}^\trd(\SmAff_S)\rightleftarrows \HH_{\Sigma,\tau}^\trd(\SmAff_{S-Z})
\]
induced by \eqref{eq:PreSZSS-Z} and \eqref{eq:ShvSZSS-Z} by the restriction.
\end{lemma}
\begin{proof}
%To complete the proof we have to show that there is an adjunction $\tilde{i}_*\dashv\tilde{i}^!$ in \eqref{eq:ShvSZSS-Z}, when $\tau\supset\tf$.
Let $G\in\Shv_{\Sigma,\tau}^\trd(\SmAff_{S,Z})$ be the sheaf associated with the presheaf represented by $X^h_Z\in\SmAff_{S,Z}$, where $X\in\SmAff_S$.
% \todo{do we need that $\tau$ is subcanonical or something like that}
Denote %\todo{by}
$iG = \cofib(h^\trd(X-X_Z)\to h^\trd(X))\in \Shv_{\Sigma,\tau}^\trd(\SmAff_{S})$.
% Then $iG(U)$ is contractible for any $U$ such that $U_Z=\emptyset$, 
% and 
% since $iG$ is a $\tau$-sheaf by the definition, and $\tau\supset\tf$ by the assumption, 
% it follows that $iG(U)\cong iG(U^h_Z)$ for any $U\in\SmAff_S$.
% Further, when $U_Z\neq \emptyset$, $iG(U)\cong iG(U^h_Z)\cong h^\tr_S(X)(U^h_Z)\cong h^\tr_S(X^h_Z)(U^h_Z)$.
% Summarising we conclude using (FinE) for $\Corr_{(-)}$ that for any $U\in \SmAff_S$, 
% \[
% iG(U)\cong \begin{cases}
% h^\tr_S(X^h_Z)(U^h_Z)\cong G(U^h_Z) , \quad U_Z\neq\emptyset,\\
% *\cong G(U^h_Z), \quad U_Z=\emptyset
% \end{cases}
% .\]
By \Cref{lm:ilsXhZ} 
$\tilde i_*(G)\cong iG$. 
Since $\tau\supset\tf$, it follows that
$\tilde i^!F(X^h_Z)\cong \fib(F(X)\to F(X-X_Z))\cong \Map(iG,F)$
for any $F\in \Shv_{\Sigma,\tau}(\SmAff_S)$,
and there are the isomorphisms
\[\Map(\tilde i_* G,F)\cong \tilde i^!F(X^h_Z) \cong \Map(G,\tilde i^!F).\]
Since representable objects
generate 
the category $\Shv_{\Sigma,\tau}(\SmAff_S)$, we get
the adjunction.
\end{proof}
%\todo{new text : end}

\begin{proposition}\label{prop:HtsSmBcZSmBSmU}
% Given a base scheme $S$ and a topology $\tau$ on $\Aff_S$ such that $\tf\subset\tau$, 
% Given $\Corr_{(-)}$ like in \Cref{lm:tildei*!j**preserveA1loctauloc},
% a scheme $S\in\Sch$, a closed subscheme $Z$, and a topology $\tau$ on $\EssSmAff_S$ such that $\tf\subset\tau$,
Consider the pair of adjunctions $\tilde i_*\dashv \tilde i^!$, $j^*\dashv j_*$, 
% \[
% \Prefr_{\wtau,\A^1}(\SmAff_{S,Z})\rightleftarrows \Prefr_{\tau,\A^1}(\SmAff_{S}) \rightleftarrows \Prefr_{\tau,\A^1}(\SmAff_{S-Z})
% \]
\[
\HH_{\Sigma,\wtau}^\trd(\SmAff_{S,Z}) \rightleftarrows \HH_{\Sigma,\tau}^\trd(\SmAff_S)\rightleftarrows \HH_{\Sigma,\tau}^\trd(\SmAff_{S-Z})
,\]
% \[\Hfr_{\wtau,\A^1}(\SmAff_{S,Z})\rightleftarrows \Hfr_{\tau,\A^1}(\SmAff_{S}) \rightleftarrows \Hfr_{\tau,\A^1}(\SmAff_{S-Z})\]
provided by \Cref{lm:tildei*!j**preserveA1loctauloc,lm:adjunctionilsiush}.
For any $F\in \HH^\trd_{\Sigma,\tau}(\SmAff_S)$,
%\Prefr_{\tau,\A^1}
the square
\begin{equation*}\label{eq:SZlocproppbsq}\xymatrix{
\tilde i_*\tilde i^!F\ar[d]\ar[r]& F\ar[d]\\
{*}\ar[r] & j_*j^*F
}\end{equation*}
is a pullback.
\end{proposition}
\begin{proof}
For any $F\in \Shv^\trd_{\Sigma,\tau}(\SmAff_S)$, 
\[\tilde i_*\tilde i^!F \simeq 
\mathrm{fib}(F(X^h_Z)\to F(X^h_Z-X_Z))\simeq 
\mathrm{fib}(F(X)\to F(X-X_Z))\simeq 
\mathrm{fib}(F(X)\to j_*j^*F(X)).\]
Hence the corresponding square is pullback in $\Shv^\trd_{\Sigma,\tau}(\SmAff_S)$. 
The claim for $\HH^\trd_{\Sigma,\tau}(\SmAff_S)$ follows by \Cref{lm:tildei*!j**preserveA1loctauloc}.
% 
% Taking sections on a scheme $X\in \Sm_S$ in the square
% we get the sequre 
% \[\xymatrix{
% F(X^h_Z/(X^h_Z-X_Z)))& F(X)\\
% * & F(X\times(S-Z))
% }\]
% where $X_Z=X\times_S Z$, and $F(X^h_Z/(X^h_Z-X_Z)))
% %\simeq 
% $ stands for $\mathrm{fib}(F(X^h_Z)\to F(X^h_Z-X_Z))$
\end{proof}

% $\Corr^\A1\fr_{S,Z}(X_{Z}, E_Z)\times_{ \Corr^\A1\fr_{S,Z}(W_{Z}, E_Z) }\Corr^\A1\fr_{S,Z}(W^h_{Z}, E^h_Z)
% \simeq \Corr^\A1\fr_{S,Z}(X_{Z}\cup W^h_{Z}, E^h_Z)\to \Corr^\A1\fr_{S,Z}(X^h_Z, E^h_Z)$

\subsection{}\label{subsect:SZA1Z}

Throughout this subsection \begin{itemize}
\item 
$\mathcal S_\bbS$ is $\Smat_\bbS$ or $\SmAff_\bbS$,
and 
\item
$\mathcal S_*$ denotes a category from the list: $\mathcal S_S$, $\mathcal S_Z$, $\mathcal S_{S,Z}$.
\end{itemize}

\begin{lemma}\label{cor:LA1trivfibhenspairsmoothLRwAHP}
%Under the assumptions of \Cref{cor:DeltatrivfibhenspairsmoothLRwAHP} 
Suppose a discrete presheaf of sets $c$ on $\Aff_S$
satisfies lifting property with respect to affine henselian pairs in sense of \Cref{def:RLwAHP}, 
and satisfies closed gluing on $\Aff_S$.
Denote by $c^h_Z$ and $c_Z$ the restrictions of $c$ to the categories $\SmAff_{B,Z}$ and $\SmAff_Z$.

Let $U\in \AffSm_{S}$, and $Y\subset U$ be a closed subscheme.
Then the morphism of simplicial sets
%\[c((\Delta^\bullet_S\times_S U)^h_{Z})\to c((\Delta^\bullet_S\times_S U)_Z).\]
\[ L_{\A^1}c^h_Z(U^h_{Z})\to L_{\A^1}c_Z(U_Z)\]
is a trivial fibration.
% More precisely, the latter morphismre is the trivial fibration
% \[c((\Delta^\bullet\times_S U)^h_{Z})\to c(Y).\]
\end{lemma}
\begin{proof}
By the construction of $L_{\A^1}$ on $\Pre_\Sigma(\SmAff_{S,Z})$ and $\Pre_\Sigma(\SmAff_{Z})$, see \Cref{subsect:SmSZ},
there are weak equivalences of simplicial sets
$L_{\A^1}c^h_Z(U^h_{Z})\simeq c((\Delta^\bullet_S\times_S U)^h_Z)$,
$L_{\A^1}c_Z(U_{Z})\simeq c(\Delta^\bullet_Z\times_Z U)$.
Since
$(\Delta^\bullet_S\times_S U)^h_Z\cong (\Delta^\bullet_U)^h_Z\cong (\Delta^\bullet_U)^h_Y$,
$\Delta^\bullet_Z\times_S U\cong \Delta^\bullet_Y$,
where $Y = U\times_S Z$,
the claim follows
by \Cref{cor:DeltatrivfibhenspairsmoothLRwAHP}.
\end{proof}

\begin{corollary}\label{lm:LA1RLwAHPhenspairweakeq}
For any 
$E\in \SmAff_S$, 
$U\in \SmAff_S$,
and closed subscheme $Z\subset S$,
the morphism 
$L_{\A^1}h^\tr(E)(U^h_Z)\to L_{\A^1}h^\tr(E)(U_Z)$ 
is a weak equivalence. 
\end{corollary}
\begin{proof}
The claim follows by \Cref{def:CorrAHP_first} and \Cref{cor:LA1trivfibhenspairsmoothLRwAHP}.
% By \cite[Corollary 2.3.25]{ehksy} there is an equivalence 
% $L_{\A^1}h^\fr(E)\simeq L_{\A^1}\Fr(-,E)$ on the category $\SmAff_S$.
% The then claim follows by \Cref{cor:LA1trivfibhenspairsmoothLRwAHP} and \eqref{ex:LRwAHP}(2).
% Similar argument holds with the use \Cref{ex:LRwAHP}(3),(4), applying \cite[Proposition 3]{SmModelSpectrumTP} for $h^{\pfr}(E)$. 
%
% By \cite[Corollary 2.3.25]{ehksy} and by \cite[Proposition 3]{SmModelSpectrumTP} there are equivalences 
% $L_{\A^1}h^\tfr(E)\simeq L_{\A^1}h^\nfr(E)\simeq L_{\A^1}\Fr(E,-)\simeq L_{\A^1}h^{\pfr}(E)$ on the category $\SmAff_S$.
% So summarising $L_{\A^1}h^\tfr(E)\simeq L_{\A^1}h^\fr(E)$ on the category $\SmAff_S$, for any $h^\fr\in\{\Fr,h^\nfr,h^\pfr\}$.
% The then claim follows by \Cref{cor:LA1trivfibhenspairsmoothLRwAHP}.
\end{proof}

\begin{definition}\label{def:CorrA}
%todo notation h_X h_\widetile X h(X)
% Let $\mathcal S_*$ denote a category from the list: 
% $\Sch_S$, $\Sm_S$, and $\Aff_{S,Z}$, $\SmAff_{S,Z}$, $\Smat_{S,Z}$.
% Let $\mathcal S_*$ denote $\mathcal S_S$, or $\mathcal S_Z$, or $\mathcal S_{S,Z}$.
Define the $\infty$-category 
$\CorrAtr(\mathcal S_*)$ 
as the full subcategory in
$\SPreradtrA(\mathcal S_*)$ spanned by the objects of the form 
$L_{\A^1}h^\tr(X)$, $X\in \mathcal S_*$.
Denote by 
\begin{equation}\label{eq:lCorrfrCorrAfr}l\colon \Corr(\mathcal S_*)\to \Corr^{\A}(\mathcal S_*)\end{equation}
the functor $X \mapsto L_{\A^1}h^\tr(X)$. We write $X\in \Corr^{\A}(\mathcal S_*)$, where $X$ denotes $L_{\A^1}h^\tr(X)$, so we may write $l(X)\cong X$.
% defined by the functor
% \[v\colon \Corr(\mathcal S_*)\to \PretrA(\mathcal S_*); X \mapsto L_{\A^1}h^\tr(X).\]
Denote by $h^{\A\tr}(X)= \CorrAtr(-,X)$ the representable presheaves.
\end{definition}
\begin{remark}
The representable presheaves $h^{\A\tr}(X)$ being restricted to $\mathcal S_*$ 
are equivalent to the presheaves
%\[
$h^\tr(X)^{\Delta^\bullet_*}\simeq \Corr(-\times\Delta^\bullet_*,X).$
%\]
\end{remark}

% \begin{definition}
% For any category $\mathcal S_*$ from the list: 
% $\Sch_S$, $\Sm_S$, and $\Aff_{S,Z}$, $\SmAff_{S,Z}$, $\Smat_{S,Z}$,
% define the $\infty$-category $\Corr^{\A\tr}(\mathcal S_*)$ as follows.

% Let 
% \[v\colon \Corr^\fr(\mathcal S_*)\to \Pre^\fr_{\A^1}(\mathcal S_*) \]
% be the functor of $\infty$-categories
% given by $X\mapsto L_{\A^1}h^\fr(X)$. 
% Consider a model for the functor $v$ in the category of simplicially enriched categories, 
% and define the simplicially enriched category $\Corr^{\A\tr}(\mathcal S_*)$ with the same objects as in 
% $\Corr^\fr(\mathcal S_*)$ and hom-spaces given by $\Corr^{\A\tr}(X_1,X_2)=\Pre^\fr_{\A^1}(v(X_1),v(X_2))$.
% Then there is the canonical functor of simplicially enriched categories 
% \begin{equation}\label{eq:lCorrfrCorrAfr}l\colon \Corr^\fr(\mathcal S_*)\to \Corr^{\A\tr}(\mathcal S_*).\end{equation}
% Finally we consider %$\Corr^\fr$ 
% the $\infty$-category and the functor associated 
% with $\Corr^{\A\tr}(\mathcal S_*)$ 
% and the functor $l$.
% \end{definition}
% % \begin{remark}
% % There is the canonical functor 
% % $\Corr^\fr(\mathcal S_*)\to \Corr^{\A\tr}(\mathcal S_*)$.
% While the functor \eqref{eq:lCorrfrCorrAfr} is a kind of $\A^1$-localisation functor,
% we do not formulate, and do not prove this, 
% using the following lemma instead.
% % \end{remark}

\begin{lemma}\label{lm:HA1CorrfrCorrAfr}
% Let $S$ be a scheme, $Z$ be a closed subscheme, 
% $\Corr_S$ be an $\infty$-category of correspondences on $\Sch_S$ that representable presheaves satisfy (AHP)(1) and the closed gluing on affine schemes.
The adjunction of $\infty$-categories
\begin{equation}\label{eq:ls*PtrPAtrlu*}
l^*\colon \SPretrrad(\mathcal S_{S,Z})\rightleftarrows  \SPreradAtr(\mathcal S_{S,Z})\colon l_*
\end{equation}
%\todo{$\Sigma$ not necessary}
restricts to the  equivalence
\[\SPreradtrA(\mathcal S_{S,Z})\simeq \SPreradAtr(\mathcal S_{S,Z}).\]
% where $\mathcal S_{S,Z}$ is $\SmAff_{S,Z}$ or $\Smat_{S,Z}$.
% \[\PreA^\tr(\mathcal S_{S,Z})\simeq \Pre^{\A\tr}(\mathcal S_{S,Z})\colon l_*%
% Let $S$ be a scheme, $Z$ be a closed subscheme, $\Corr_S$ be an $\infty$-category of correspondences on $\Sch_S$ that representable presheaves satisfy (RLwAHP) and the closed gluing on affine schemes.
% The functor of $\infty$-categories
% \begin{equation}\label{eq:PtrPAtrl*}\Pre^\tr(\mathcal S_{S,Z})\leftarrow \Pre^{\A\tr}(\mathcal S_{S,Z})\colon l_*
% \end{equation}
% preserves $\A^1$-invariant objects and induces the  equivalence
% \[\PreA^\tr(\mathcal S_{S,Z})\leftarrow \Pre^{\A\tr}(\mathcal S_{S,Z})\colon l_*%
% \[\PreA^\tr(\mathcal S_{S,Z})\leftarrow \PreA^{\A\tr}(\mathcal S_{S,Z})\colon l_*%
% $l^*\colon \HH_{\A^1}(\Corr^\fr(\SmAff_{S,Z}))\rightleftarrows \HH_{\A^1}(\Corr^{\A\tr}(\SmAff_{S,Z}))\colon l_*$
%is an equivalence
% ,\]
% where $\mathcal S_{S,Z}$ is $\SmAff_{S,Z}$ or $\Smat_{S,Z}$.
\end{lemma}
\begin{proof}
Since the $\A^1$-localisation functors on 
$\SPreradtr(\mathcal S_{S,Z})$ and $\SPreradAtr(\mathcal S_{S,Z})$
are both given by $L_{\A^1}\cong(-)^{\Delta_{S,Z}^\bullet}$,
and
the functor $l$ commutes with the functors $-\times\Delta^n_{S,Z}$, i.e.
\begin{equation}\label{eq:LAcomutetrAtrl}l_*L_{\A^1}\simeq L_{\A^1}l_*\colon \SPreradAtr(\mathcal S_{S,Z})\to \SPreradtr(\mathcal S_{S,Z}).\end{equation}

Further, by the definition $l_*h^{\A\tr}(E)\simeq L_{\A^1}h^\tr(E)$, and $l^* h^\tr(E)\simeq h^{\A\tr}(E)$.
Hence 
\begin{equation}\label{eq:llhtrLa1htr}l_*l^*h^\tr(E)\simeq L_{\A^1}h^\tr(E)\in \SPreradtr(\mathcal S_{S,Z}),\end{equation}
and 
\begin{equation}\label{eq:usPtrPAtr}L_{\A_1}l_*h^{\A\tr}(E)\simeq L_{\A^1}h^\tr(E)\simeq l_*h^{\A\tr}(E).\end{equation}
So by \ref{eq:llhtrLa1htr} the unit of the adjunction \eqref{eq:ls*PtrPAtrlu*}
% \begin{equation}\label{eq:l*trAtr}l^*\colon \Pre^\tr(\mathcal S_{S,Z})\rightleftarrows \Pre^{\A\tr}(\mathcal S_{S,Z})\colon l_*\end{equation}
is equivalent to the endofunctor $L_{\A^1}$ on $\SPreradtr\mathcal S_{S,Z})$.
Since the functor $\SPreradAtr(\mathcal S_{S,Z})\to \SPrerad(\mathcal S_{S,Z})$
is conservative, $l_*$ in \eqref{eq:ls*PtrPAtrlu*}
is conservative,
and hence \eqref{eq:usPtrPAtr} implies that the counit 
is equivalent to the identity endofunctor on $\SPreradAtr(\mathcal S_{S,Z})$.
Thus \eqref{eq:ls*PtrPAtrlu*} is a reflection, and 
induces the equivalence of $\SPreradtrA(\mathcal S_{S,Z})$ and $\Pre^{\A\tr}_\Sigma(\mathcal S_{S,Z})$.

% Since the $\A^1$-localisation functors on 
% $\Pre^\tr(\mathcal S_{S,Z})$ and $\Pre^{\A\tr}(\mathcal S_{S,Z})$
% both are given by $L_{\A^1}\cong(-)^{\Delta_{S,Z}^\bullet}$,
% and
% the functor $l$ commutes with the functors $-\times\Delta^n_{S,Z}$,
% \[l_*L_{\A^1}\simeq L_{\A^1}l_*\colon \Pre^{\A\tr}(\mathcal S_{S,Z})\to \Pre^\tr(\mathcal S_{S,Z}),\] 
% and hence $l_*$ preserves $\A^1$-invariant objects.

% Further, by the definition $l_*h^{\A\tr}(E)\simeq L_{\A^1}h^\tr(E)$, and $l^* h^\tr(E)\simeq h^{\A\tr}(E)$.
% Hence 
% \begin{equation}\label{eq:llhtrLa1htr}l_*l^*h^\tr(E)\simeq L_{\A^1}h^\tr(E)\colon \Pre^\tr(\mathcal S_{S,Z})\to \Pre^\tr(\mathcal S_{S,Z}).\end{equation}

% The adjunction
% \begin{equation}\label{eq:l*trAtr}l^*\colon \Pre^\tr(\mathcal S_{S,Z})\rightleftarrows \Pre^{\A\tr}(\mathcal S_{S,Z})\colon l_*\end{equation}
% induces the adjunction
% \[L_{\A^1}l^*\colon \Pre^\tr(\mathcal S_{S,Z})\rightleftarrows \Pre^{\A\tr}(\mathcal S_{S,Z})\colon l_*,\]
% that restricts in its turn to the adjunction
% \begin{equation}\label{eq:PA1trPA1Atr}\Pre^{\A\tr}(\mathcal S_{S,Z})\rightleftarrows \Pre^{\A\tr}(\mathcal S_{S,Z})\colon l_*,\end{equation}
% because $l_*$ preserves $\A^1$-invariant objects be the above.
% By \eqref{eq:llhtrLa1htr} it follows that the unit of the adjunction is equivalent to the identity endofunctor on $\Pre^{\A\tr}(\mathcal S_{S,Z})$.
% %Thus $l_*l^*\simeq_{\A^1} \Id$.
% Since the functor $l_*$ in \eqref{eq:l*trAtr} is conservative, the functor $l_*$ in \eqref{eq:PA1trPA1Atr} is conservative.
% Thus the adjunction is the equivalence.
\end{proof}

\begin{lemma}\label{eq:AfrSZAsimeqZ}
The functor 
$\CorrAtr(\Smat_{S,Z})\rightarrow \CorrAtr(\Smat_Z)$
% in the first row of \eqref{eq:CorrfrAfrSZZ}
% $\Corr^{\A\tr}_{S,Z}\to \Corr^{\A\tr}_{Z}$
% \begin{equation*}% \label{eq:CorrfrSZZ}
% \Corr^{\A\tr}(\SmAff_{S,Z})\to \Corr^{\A\tr}(\SmAff_Z); X^h_Z\mapsto X_Z,\end{equation*}
% where $X_Z=X\times_S Z$, 
is an equivalence.
\end{lemma}
\begin{proof}
The claim follows because morphisms
$\CorrAtr_{S,Z}(X_1,X_2)\to \CorrAtr_{Z}(X_1\times_S Z,X_2\times_S Z)$
are weak equivalences 
%trivial fibrations
by \Cref{lm:LA1RLwAHPhenspairweakeq},
and for any scheme $X\in\Smat_Z$ there is a scheme $\tilde X\in \Smat_{S,Z}$ such that $\tilde X\times_S Z\cong X$.
\end{proof}

% The 
% base change along the immersion $i\colon Z\to S$
% leads to the functor 
% The base change along $i\colon Z\to S$
The mapping $X^h_Z\mapsto X_Z$ % induces 
defines the functor 
$\overline{i}^*_\Corr\colon \Corrtr(\mathcal S_{S,Z})\rightarrow\Corrtr(\mathcal S_Z)$
%in the second row of \eqref{eq:CorrfrAfrSZZ}
% \begin{equation}\label{eq:CorrfrSZZ}\Corr^\fr(\SmAff_{S,Z})\to \Corr^\fr(\SmAff_Z); X^h_Z\mapsto X_Z
% ,\end{equation} where $X\in \SmAff_S$, $X_Z=X\times_S Z\in \SmAff_Z$,
that induces the functors %of $\infty$-categories of presheaves
\begin{equation}\label{eq:PrefrSZZ}
\overline i^*\colon 
\SPretrrad(\mathcal S_{S,Z})\rightleftarrows \SPreradtr(\mathcal S_Z)
\colon \overline i_*
\end{equation} %$H^\fr$
given by $\overline i^*(h^\tr(X^h_Z))=h^\tr(X_Z)$, $\overline i_*(F)(X^h_Z)=F(X_Z)$.
%and similarly for $\Smat_\bbS$.
Since $\overline{i}^*_\Corr$
takes $(\A^1\times X)^h_Z$ to $\A^1\times X_Z$,
the functor $\overline i_*$ preserves $\A^1$-invariant objects, 
and $\overline i^*$ preserves $\A^1$-equivalences,
and \eqref{eq:PrefrSZZ} induces the adjunction
\begin{equation}\label{eq:HfrASZZ}
\overline i^*\colon \SPreradtrA(\mathcal S_{S,Z})\rightleftarrows \SPreradtrA(\mathcal S_Z)\colon \overline i_*.
\end{equation} %$H^\fr$
% because the horizontal functors in \eqref{eq:CorrfrAfrSZZ} below take $(\A^1\times X)^h_Z$ to $\A^1\times X_Z$.

\begin{proposition}\label{prop:HfrBcZsimeqHfrZ}
% Let $\mathcal S_*$ denote $\SmAff_*$ or $\Smat_*$, see \Cref{den:introductionnotation}. %, see \Cref{subsect:SZA1Z}
% Let $\mathcal S_\bbS$ be $\Smat_\bbS$ or $\SmAff_\bbS$.
% The functor $\overline i_*$
The adjunction \eqref{eq:HfrASZZ} is an equivalence.
% induce the
% induces equivalence of categories
% \[\PretrA(\mathcal S_{S,Z})\simeq \PretrA(\mathcal S_Z).\] %$H^\fr$
%
% The inverse and direct image functors 
% along %\eqref{eq:CorrfrSZZ} 
% the functor $\Corr^\fr(\SmAff_{S,Z})\to \Corr^\fr(\SmAff_Z)$
% %functor $i^*$ (and $i_*$) 
% %s
% induces equivalence of categories
% $\HfrA(\Sm_{S,Z})\simeq \HfrA(Z)$. %$H^\fr$
%Consequently, $\Pretr_{\stau,\A^1}(\Smat_{S,Z})\simeq \Pretr_{\tau,\A^1}(\Smat_Z)$ for a topology $\tau$ on $\Smat_Z$.
% The claim holds for the family of categories $\SmAff_*$ as well.
\end{proposition}
\begin{proof}
Consider the commutative square
\begin{equation}\label{eq:CorrfrAfrSZZ}
\xymatrix{
\CorrAtr(\Smat_{S,Z})\ar[r]& \CorrAtr(\Smat_Z)\\
\Corrtr(\Smat_{S,Z})\ar[u]\ar[r]& \Corrtr(\Smat_Z)\ar[u] 
,}\end{equation}
where the horizontal arrows take $X^h_Z$ to $X_Z$, and similar one for $\Smat_\bbS$.
% Consider the category
% $\Corr^{\A1\fr}_{S,Z}$
% with hom-spaces given by 
% $\Corr^{\A1\fr}_{S,Z}(X_1,X_2)\simeq L_{\A^1}h^{\fr}_{S,Z}(X_1)(X_2)$, where $h^{\fr}_{S,Z}(X_1)=h^{\fr}_{S,Z}(X_1)(-)$.
Then \eqref{eq:HfrASZZ} for $\mathcal S_* = \Smat_*$ is an equivalence %$\SPreradtrA(\Smat_Z)\simeq \SPreradtrA(\Smat_{S,Z})$ follows,
since 
by \Cref{lm:HA1CorrfrCorrAfr} and \Cref{eq:AfrSZAsimeqZ} there are the equivalences
\[
\xymatrix{
\SPreradAtr(\Smat_{S,Z})\ar[d]^{\simeq}& \SPreradAtr(\Smat_Z)\ar[d]^{\simeq}\ar[l]^{\simeq}\\
\SPreradtrA(\Smat_{S,Z})& \SPreradtrA(\Smat_Z) 
\ar[l],}
\] % The case of 

% So the claim is proved for $\Smat_\bbS$.
%
Then the claim for $\SmAff_\bbS$ follows because of \Cref{lm:PretrASmAffsimeqPretrASmat}. 
\end{proof}

\begin{lemma}\label{lm:SmatZSZ}
(1) For any $X\in \Smat_Z$ there is $\tilde X\in\Smat_{S,Z}$ such that $X=\tilde X\times_S Z$.
(2) For any given morphism $X\to Y\in \Smat_Z$, there is a morphism $\tilde X\to \tilde Y$ that goes to the given one along the functor $-\times_SZ$.
% $\widetilde X,\widetilde Y\in \Smat_{X,Z}$
% the morphism $\Smat_{S,Z}(\widetilde X,\widetilde Y)\to \Smat_Z(X,Y)$,
% where $X=\widetilde X\times_S Z$, $Y=\widetilde Y\times_S $,
% is surjective.
\end{lemma}
\begin{proof}
Point (1) follows because of \cite[Lemma A.11]{DKO:SHISpecZ}. Point (2) follows then by \cite[Lemma E.10]{DKO:SHISpecZ}.
\end{proof}
\begin{lemma}\label{lm:stauSmatoverAff}
Let $\tau$ be a topology on $\Smat_Z$, see \Cref{subsect:SZA1Z}.

(1)
Let $\tau$ be a topology on $\Smat_Z$.
A morphism $\widetilde X\to X$ is a $\tau$-covering 
if and only if there is a morphism $\tilde{\widetilde X}\to \tilde X$
such that $\tilde{\widetilde X}_Z\to \tilde X_Z$ is a $\tau$-covering.

(2)
A presheaf $F\in \SPreradtr(\Smat_Z)$ is a $\tau$-sheaf if and only if 
$\overline{i}_* F\in \SPreradtr(\Smat_{S,Z})$ is $\stau$-sheaf.
% A presheaf $F\in \Pre^\tr(\Smat_Z)$ is a $\tau$-sheaf if and only if 
% the presheaf $F(-\times_S Z)\in \Pre^\tr(\Smat_Z)$ is $\stau$-sheaf.
\end{lemma}
\begin{proof}
Point (1) follows from \Cref{lm:stau} because of \Cref{lm:SmatZSZ%
}%
% for any morphism $\widetilde X\to X$ in $\Smat_Z$ there is a morphism is $\Smat_{S,Z}$ that goes to the given one along the functor $\Smat_{S,Z}\to \Smat_Z$
.
We proceed with point (2). Since the functor $\Smat_{S,Z}\to \Smat_Z$ is continuous, the functor $\SPreradtr(\Smat_{S,Z})\to\SPreradtr(\Smat_Z)$ preserves the sheaves, and 
it detects the sheaves by point (1).
\end{proof}

\begin{proposition}\label{prop:HfrtauBcZsimeqHfrtau}
%Let $\mathcal S_\bbS$ be $\Smat_\bbS$ or $\SmAff_\bbS$.
% The functor $\overline i_*$
%
The adjunction \eqref{eq:HfrASZZ} %induce the
induces the equivalence %of $\infty$-categories
\begin{equation}\label{eq:HfrAstauSZZ}
\overline i^*\colon \HHtr_{\Sigma,\stau}(\mathcal S_{S,Z})\stackrel{\simeq}{\leftrightarrows} \HHtr_{\Sigma,\tau}(\mathcal S_Z)\colon \overline i_*
\end{equation} %$H^\fr$
for any topology $\tau$ on $\mathcal S_Z$.
%
%
% The adjunction \eqref{eq:HfrASZZ} %induce the
% induces the equivalence %of $\infty$-categories
% % \[\Pretr_{\stau,\A^1}(\mathcal S_{S,Z})\simeq \Pretr_{\tau,\A^1}(\mathcal S_Z)\] for a topology $\tau$ on $\mathcal S_Z$.
% \[\HHtr_\stau(\mathcal S_{S,Z})\simeq \HHtr_\tau(\mathcal S_Z)\] 
% for any topology $\tau$ on $\mathcal S_Z$.
% % The claim holds for the family of categories $\SmAff_*$ as well.
%
\end{proposition}
\begin{proof}
% A presheaf $F\in \Pre^\tr(\Smat_Z)$ is a $\tau$-sheaf if and only if 
% the presheaf $F(-\times_S Z)$ is $\stau$-sheaf by \Cref{lm:stauSmatoverAff}(2).
Let $\mathcal S_\bbS$ be $\Smat_\bbS$.
The equivalence proved in \Cref{prop:HfrBcZsimeqHfrZ} restricts to the equivalence on the subcategories of sheaves because of \Cref{lm:stauSmatoverAff}(2). 

By \Cref{lm:stau,lm:stauSmatoverAff}
a morphism in $\Smat_{S,Z}$ is a $\stau$-covering if and only if it is an $\stau$-covering in $\SmAff_{S,Z}$.
So the functor $\SPrerad(\SmAff_{S,Z})\to \SPrerad(\Smat_{S,Z})$ preserves $\stau$-sheaves,
and the claim for $\mathcal S_\bbS=\SmAff_\bbS$ follows.
%
% The proved equivalence for the categories of presheaves on $\Smat_*$ in \Cref{prop:HfrBcZsimeqHfrZ}
% restricts to the equivalence on the subcategories of sheaves because of \Cref{lm:stauSmatoverAff}(2). 
%
% The claim for $\SmAff_*$ follows because the functor $\Pre(\SmAff_{S,Z})\to \Pre(\Smat_{S,Z})$ preserves $\stau$-sheaves,
% because a morphism in $\Smat_{S,Z}$ is a $\stau$-covering if and only if it is an $\stau$-covering in $\SmAff_{S,Z}$ by \Cref{lm:stau,lm:stauSmatoverAff}.
\end{proof}

\subsection{}\label{sect:LocThConclusion} We conclude the main results of the section.
Let $\tau$ be a %subcanonical 
family of topologies on $\EssSmAff_\bbS$.
Consider the functor 
\[i_*\colon \HHtr_{\Sigma,\stau}(\SmAff_{Z})\to\HHtr_{\Sigma,\tau}(\SmAff_{S});\; i_*F(X)=F(X_Z).\]
Then $i_*\simeq\tilde i_*\overline i_*$, 
and the functor $i^!:=\overline i^*\tilde i^!$
is the right adjoint, 
where 
%$\overline i^*\colon \HHtrd_{\tau}(\SmAff_{S,Z})\to \HHtrd_{\stau}(\SmAff_{Z})$ 
$\overline i^*$ is from \eqref{eq:HfrAstauSZZ}.
The letter adjunctions induce the ones for the pointed $\infty$-categories $\HHtr_*(-)$.
% where $\overline i^*=(\overline i_*)^{-1}$ is the inverse functor to 
% the equivalence $\overline i_*\colon \HHtrd_{\tau}(\SmAff_{S})\to \HHtrd_{\stau}(\SmAff_{Z})$ given by \Cref{prop:HfrBcZsimeqHfrZ}.
% Consider the functor $i^*\colon \Pretr_{\A^1,\stau}(\SmAff_{Z})\to\Pretr_{\A^1,\tau}(\SmAff_{S})$; $i_*F(X^h_Z)=F(X_Z)$.
% Then $i_*\simeq\tilde i_*\overline i_*$, and $i_*$ has the right adjoint $i^!:=\overline i^*\tilde i^!$, where $\overline i^*=(\overline i_*)^{-1}$ is the inverse functor to the equivalence $\overline i_*\colon \Pretr_{\A^1,\tau}(\SmAff_{S})\to \Pretr_{\A^1,\stau}(\SmAff_{Z})$ given by \Cref{prop:HfrBcZsimeqHfrZ}.

\begin{theorem}\label{th:tautfLoc}
%Let $S$ be a scheme, $Z$ be a closed subscheme, 
% Let $i\colon Z\to S$ be a closed immersion of affine schemes.

% Let $\Corr_{(-)}$ be a family of an $\infty$-categories of correspondences on $\EssSm_{(-)}$ over $\Sch$ 
% that 
% % is continuous
% % %satisfies (FinE) 
% % in sense of \Cref{def:FinECorr}, and 
% satisfies (AHP) in sense of \Cref{def:CorrAHP}.

% Let $\tau$ be a %subcanonical 
% % topology on $\Sch_S$ such that 
% family of topologies on $\EssSmAff_\bbS$ such that 
% $\tau$ 
% is continuous
% %satisfies (FinE) 
% in sense of \Cref{def:FinProp},
Suppose that
$\tau\supset \tf$, and 
%such that 
$\wtau=\stau$ on $\SmAff_{S,Z}$, see \Cref{def:wtaustau}.
Consider the pair of adjunctions
\[\HHtrd_{\Sigma,\tau}(\SmAff_{Z})\rightleftarrows \HHtrd_{\Sigma,\tau}(\SmAff_{S})\rightleftarrows\HHtrd_{\Sigma,\tau}(\SmAff_{S-Z})\]
% \[\Pretr_{\A^1,\tau}(\SmAff_{Z})\rightleftarrows \Pretr_{\A^1,\tau}(\SmAff_{S})\rightleftarrows\Pretr_{\A^1,\tau}(\SmAff_{S-Z})\]
% \[\Hfr_{\A^1,\tau}(\SmAff_{S})\rightleftarrows \Hfr_{\A^1,\tau}(\SmAff_{S})\rightleftarrows\Hfr_{\A^1,\tau}(\SmAff_{Z})\]
given by $i_*\dashv i^!$, $j^*\dashv j_*$.
Then for any $F\in \HHtrd_{\Sigma,\tau}(\SmAff_{S})$, there is a pullback square
%Then for any $F\in \Pretr_{\A^1,\tau}(\SmAff_{S})$, there is a pullback 
\[\xymatrix{
i_* i^! F\ar[r]\ar[d] & F\ar[d]\\
{*}\ar[r] & j_* j^* F.
}\]
\end{theorem}
\begin{proof}
The claim follows by \Cref{prop:HtsSmBcZSmBSmU,prop:HfrtauBcZsimeqHfrtau}. %lm:LA1hfrhenspairweakeq}
% for the case of $\tau$ such that
% $\wtau=\stau$ on $\SmAff_{S,Z}$.
% In particular, the claim holds for $\tf$.
% For a general $\tau$ such that $\tau\supset\tf$ the claim follows form the case of $\tf$ and \Cref{lm:tildei*!j**preserveA1loctauloc}.
\end{proof}

\begin{theorem}\label{th:tautfLoc:istartjsharp}
Suppose that
$\tau\supset \tf$, and 
$\wtau=\stau$ on $\SmAff_{S,Z}$, see \Cref{def:wtaustau}.
Consider the pair of adjunctions
\[\HHtr_{\Sigma,\tau}(\SmAff_{Z})\leftrightarrows \HHtr_{\Sigma,\tau}(\SmAff_{S})\leftrightarrows\HHtr_{\Sigma,\tau}(\SmAff_{S-Z})\]
given by $i_*\vdash i^*$, $j^*\vdash j_\#$.
Then for any $F\in \HHtr_{\Sigma,\tau}(\SmAff_{S})$, there is a pushout square
\begin{equation*}\xymatrix{
\tilde i_*\tilde i^*F& F\ar[l]\\
{*}\ar[u]\ar[r] & j_\#j^*F\ar[u]\ar[l]
.}\end{equation*}
\end{theorem}
\begin{proof}
The claim follows by \Cref{prop:HtsSmBcZSmBSmU:istartjsharp,prop:HfrtauBcZsimeqHfrtau}. %lm:LA1hfrhenspairweakeq}
\end{proof}

\begin{corollary}\label{cor:tautfLoc}
The claims of \Cref{th:tautfLoc,th:tautfLoc:istartjsharp} hold for $\bbS=\Afffns$, and the topology $\tau=\nuf$ 
for any family of subtopologies $\nu$ of the \'etale topology, see \Cref{def:nuf}.
\end{corollary}
\begin{proof}
The property 
$\stau=\wtau$ holds by \Cref{lm:niszftfstauwtau}. 
\end{proof}

\begin{example}
The required properties on $\Corr_{(-)}$ in \Cref{th:tautfLoc} are satisfied for the category %$\Corr_{(-)}$ being 
$\SmAff_{(-)}$, and the $\infty$-category $\Corr^\fr_{(-)}$.
\end{example}

\begin{example}
% Recall that by \Cref{lm:niszftfstauwtau} 
% the property $\stau=\wtau$ holds for 
% the topology $\tau=\nuf$ 
% for any family of subtopologies $\nu$ of the Nisnevich topology, see \Cref{def:nuf}.

The example of the topologies $\nu$ in \Cref{cor:tautfLoc} are the trivial topology, Zariski topology, and the Nisnevich, or \'etale topologies. Then $\tau$ is the trivial fibre, Zariski fibre, and the Nisnevich, or \'etale topology respectively.

% So the example of the topologies $\tau$ in \Cref{th:tautfLoc} are 
% the trivial topology on $\Aff_{(-)}$ over $\Aff$, Zariski topology, and the Nisnevich topology. 
% The equality $\stau=\wtau$ for the Nisnevich topology is provided by \cite[Proposition 4.2]{DKO:SHISpecZ}, and the claim for Zariski topology and trivial fibre topology follows similarly. 
\end{example} 

\section{\'Etale excision and motivic localisation.}

%Zariski and Nisnevich localisations, and  for motivic spaces \subsection{Motivic localisation of presheaves with transfers.}
\label{sect:EtExLocA1niszar}

\subsection{General lemma on the excision property.}\label{sect:pointexcisive(pre)sheaf}

\begin{definition}\label{def:eX}
For an $\infty$-category $\calS$ and an object $X\in \calS$
denote by $\calS_X$ the comma category over $X$.
Define the functor
\[e_X\colon\SPre(\calS_X)\to \SPre(\calS)\]
that takes a presheaf $F$ on $\calS_X$
to a presheaf $G$ on $\calS$
\[G(W)=\coprod_{W\to X\in \calS_X}F(W\to X).\]
\end{definition}

\begin{lemma}\label{lm:ConservativeeXSXS}
For an $\infty$-category $\calS$ and $X\in \calS$ the functor $e_X\colon\SPre(\calS_X)\to \SPre(\calS)$ is conservative.
\end{lemma}
\begin{proof}
Let $F\to F^\prime$ be a morphism in $\SPre(\calS_X)$.
Let $w\colon W\to X\in \calS_X$.
If $e_X(F)\simeq e_X(F^\prime)$
then the canonical morphism
$\coprod_{s\colon W\to X\in \calS_X}F(s)\to\coprod_{s\colon W\to X\in \calS_X}F^\prime(s)$
is an equivalence.
Since the coproduct functor in $\mathrm{Spc}$ is conservative 
the equivalence $F(w)\simeq F^\prime(w)$ follows.
\end{proof}

%changed
\begin{definition}
Let $Z$ be a Grothendieck topology on a category $\calS$.
We say that an object $F$ of $\SPre(\calS)$ is a $Z$-sheaf 
if for any $Z$-covering sieve $c\in \Pre(\calS)$ of $X\in \calS$ the map 
$F(X)\to \lim\limits_{X^\prime\in c} F(X^\prime)$
is an equivalence, where $c$ in the subscript the fibrant category over $\calS$ corresponding to $c$.
% Let $Z$ be a Grothendieck topology on a category $\calS$.
% We say that an object $F$ of $\SPre(\calS)$ is a $Z$-sheaf 
% if for any $Z$-covering $\tilde X\to X$ the map 
% $F(X)\to \lim\limits_{\Delta} F(\check C_X(\tilde X))$
% is an equivalence.
\end{definition}
For a Grothendieck topology $Z$ on a category $\calS$ denote by the same symbol the induced topology on $\calS_X$.
\begin{lemma}\label{lm:eXZsheaves}
The functor $e_X$ from Def. \ref{def:eX} preserves $Z$-sheaves. 
%local equivalences
\end{lemma}
\begin{proof}
The claim follows form the definitions.
\end{proof}

For a category $\calS$ denote by $\mathrm{pro-}\calS$ the category of pro-objects.
For any $F\in \SPre(\calS)$, we denote by the same symbol the corresponding continuous presheaf on $\mathrm{pro-}\calS$.
% We identify the category  $\SPre(\calS)$ with the category of continuous presheaves
% on $\mathrm{pro-}\calS$

\begin{lemma}\label{lm:pointproobjectsSX}
% Let $\mathcal T$ be a family of pro-objects in $\calS$ that define enough set of points for 
% a 
% hypercomplete
% Grothendieck topology $Z$ on $\calS$.
% Then for $X\in \calS$ the 
% family $\mathcal T_X$ of pro-objects in $\calS_X$ given by morphisms of the form $T\to X$, $T\in \mathcal T$,
% gives enough set of points for 
% the 
% hypercomplete 
% topology $Z$ on $\calS_X$.
Let $\mathcal T$ be a family of pro-objects in $\calS$ that define enough set of points for a Grothendieck topology $Z$ on $\calS$ such that the topos of sheaves $\Shv_Z(\mathcal T)$ is hypercomplete.
Then for $X\in \calS$ the 
family $\mathcal T_X$ of pro-objects in $\calS_X$ given by morphisms of the form $T\to X$, $T\in \mathcal T$,
gives enough set of points for 
the topology $Z$ on $\calS_X$,
and 
the topos of sheaves $\Shv_Z(\calS_X)$ is
hypercomplete.
\end{lemma}
\begin{proof}
Since a pro-object $T\in \mathcal T$ defines a point of $Z$ on $\calS$,
for any $Z$-covering $\widetilde W\to W$, $W\in \calS$, the morphism $c_W\colon \widetilde W\times_W T\to T$ has a right inverse $l_W$.
Since $c_W\circ l_W=\mathrm{id}_W$, 
for any morphism $W\to X$ the morphism $l_W$ defines the morphism in $\calS_X$.
Then it follows that any pro-object $T\to X\in \mathcal T_X$ 
is a point of the topology $Z$ on $\calS_X$.

We are going to show that the set of points $\mathcal T_X$ is enough for the topology $Z$ on $\calS_X$, and the topos is hypercomplete.
Let $F\to F^\prime$ be any morphism of $Z$-sheaves on $\calS_X$ such that 
\begin{equation}\label{eq:FeqFp(t)}
F(t)\simeq F^\prime(t)
\end{equation} for any $t\colon T\to X\in \mathcal T_X$.
By \Cref{lm:eXZsheaves} $e_X(F)\to e_X(F^\prime)$ is a morphism of $Z$-sheaves on $\calS$.
By the above %for any $T\in \mathcal T$ and any morphism $t\colon T\to X$, 
$t\colon T\to X\in \mathcal T_X$ defines a $Z$-point in $\calS_X$.
Hence for any $T\in \mathcal T$, and $F\to F^\prime$ as above, 
there are isomorphisms 
\begin{equation*}\label{eqqq}
e_X(F)(T)\simeq \coprod_{t\colon T\to X} F(t)\stackrel{\eqref{eq:FeqFp(t)}}{\simeq} \coprod_{t\colon T\to X} F^\prime(t) \simeq e_X(F^\prime)(T)
\end{equation*} 
where the middle one follows by \eqref{eq:FeqFp(t)} because 
for any $T\in \mathcal T$ and $t\colon T\to X$, we have $t\in \mathcal T_X$.
Since $\mathcal T$ is an enough set of $Z$-points for $Z$ on $\calS$,
it follows that $e_X(F)\simeq e_X(F^\prime)$,
and consequently $F\simeq F^\prime$ by \Cref{lm:ConservativeeXSXS}.
% We are going to show that the set of points $\mathcal T_X$ is enough for the topology $Z$ on $\calS_X$.
% Let $F\to F^\prime$ be a morphism of $Z$-sheaves on $\calS_X$ such that $F(T\to X)\simeq F^\prime(T\to X)$ for any $T\to X\in \mathcal T_X$.
% By \Cref{lm:eXZsheaves} $e_X(F)\to e_X(F^\prime)$ is a morphism of $Z$-sheaves on $\calS$.
% For any $T\in \mathcal T$ and morphism $t\colon T\to W$, 
% by the above $t$ define a $Z$-point in $\calS_X$, 
% hence $e_X(F)(T)\to e_X(F^\prime)(T)$ by assumption of the morpihsm $F\to F^\prime$.
% Then $e_X(F)\simeq e_X(F^\prime)$, since $\mathcal T$ is an enough set of $Z$-points in $\calS$.
% Thus $F\simeq F^\prime$ by \Cref{lm:ConservativeeXSXS}.
\end{proof}

% Let $N$ be a Grothendieck topology on a category $\calS$ and $Z$ be a subtopology of $N$.
% Assume that $N$ is complete decomposible, and 
% $Z$ admits enough set of points given by a family of pro-objects in $\calS$.

% For any pro-object $V$ in $\calS$ and, a morphism $V\to X$, 
% and an square $R$ 
% \begin{equation}\label{eq:NsquareWX}\xymatrix{
% W^\prime\ar[d]\ar[r] & X^\prime\ar[d]\\
% W\ar[r] & X
% }\end{equation}
% in $\calS$, 
% the fibred product $R\times_X V$ 
% is a square in the category of pro-objects.
% The squares of the form $R\times_X V$ an $N$-square over $V$.
% %We extent the $cd$-structure of the topology $N$ to the category of pro-objects.

% We consider a presheaf $F$ on $\calS$ being extended to the the category of pro-objects in $\calS$ as continuous presheaf.
% For any (pro)object $V$ in $\calS$ and an $N$-square $R$ as above 
% % \[
% % W^\prime\ar[d]\ar[r] & V^\prime\ar[d]\\
% % W\ar[r] & V
% % \]
% over $V$ we say that $F$ satisfies excision on $R\times_X V$ if 
% \begin{equation}\label{eq:hofibFwodetldeR}F(V)\simeq F( (W\amalg V^\prime)\times_X V)\to F(W^\prime\times_X V)).\end{equation}

\begin{definition}\label{def:excisive}
Let $R$ be a square
\begin{equation}\label{eq:NsquareWX}\xymatrix{
W^\prime\ar[d]\ar[r]^{i^\prime} & X^\prime\ar[d]^p\\
W\ar[r]^i & X
}\end{equation}
in $\calS$.
We say that a presheaf $F\in \SPre(\calS)$ is \emph{$R$-excisive} 
if
\begin{equation}\label{eq:hofibFwodetldeR}
F(X)\simeq F( W )\times_{F(X)} F( X^\prime ).\end{equation}

Given a morphism $V\to X$ in $\calS$,
we say that $F$ is \emph{$R$-excisive over $V$} 
if $F$ is excisive with respect to the square $R\times_X V$.
\end{definition}

Given a $\mathrm{cd}$-structure $N$ on a category $\calS$,
denote by the same symbol $N$ the completely decomposable topology $N$ on $\calS$
defined by the $\mathrm{cd}$-structure, and call the squares of $N$ by $N$-squares. 
Any $N$-sheaf is $N$-excisive and 
under the set of assumptions of the cirterion \cite[Theorem 3.2.5]{10.1215/00127094-0000014X}
a presheaf $F\in \SPre(\calS)$ is an $N$-sheaf if and only if 
it is $N$-excisive for each $N$-square. %%changed

\begin{lemma}[Main lemma] \label{lm:NexZpoints}
Let $N$ be a Grothendieck topology on a category $\calS$ and $Z$ be a subtopology of $N$.
Assume that $N$ is complete decomposible, and 
$Z$ admits enough set of points given by a family of pro-objects in $\calS$
and the associated topos is hypercomplete.

Let $F$ be a $Z$-sheaf on $\calS$ such that 
for each $Z$-point $U$ 
the presheaf $F$ is $N$-excisive over $U$.
Then $F$ is $N$-excisive, and it is an $N$-sheaf under the assumptions of \cite[Theorem 3.2.5]{10.1215/00127094-0000014X}.
\end{lemma}
\begin{proof}
Let $R$ be an $N$-square \eqref{eq:NsquareWX} in $\calS$ over the scheme $X$.
% Let $\calS_X$ be the comma category over $X$ in $\calS$, and 
% denote by the same symbols $Z$, $N$, and $F$ the topologies and the presheaf induced on $\calS_X$.
% Let $\mathcal T$ be a family of pro-objects in $\calS_X$ that defines enough set of points for the topology $Z$, 
% and $\mathcal T_X$ be the family of pro-objects in $\calS_X$ of the form $T\to X$, $T\in \mathcal T$.
% 
Consider the presheaf $F^\prime$ on the category of pro-objects in $\calS_X$ defined 
by the assignment
\begin{equation*}
\{V\to X\}\mapsto F( W\times_X V )\times_{F(X\times_X V)} F( X^\prime\times_X V ).\end{equation*}
Since finite %homomtopy 
limits commute with filtered colimits, $F^\prime$ is continuous, and since finite %homotopy 
limits commute with %homotopy 
limits, $F^\prime$ is a $Z$-sheaf on $\calS_X$.

It follows by the assumption on $F$ and Lemma~\ref{lm:pointproobjectsSX}, that 
$F$ is $R$-excisive over any $T\in\mathcal T_X$.
% for the square $R\times_X T$
% for any $T\in\mathcal T_X$. 
So we have equivalences $F(T)\simeq F^\prime(T)$ for each $T\in \mathcal T_X$.
Hence by Lemma~\ref{lm:pointproobjectsSX} the canonical morphism $F\to F^\prime$ is $Z$-local equivalence. 
Thus $F\simeq F^\prime$ since both presheaves are $Z$-sheaves. %n\colon 
Since $F^\prime$ is $R$-excisive by the definition, $F$ is $R$-excisive.

Thus $F$ is $R$-excisive for each $N$-square $R$ in $\calS$.
\end{proof}

\subsection{\'Etale excision theorem for framed presheaves.}\label{subsect:EtExFrCorr}

Let $\Corr_k$ be an additive $\infty$-category of correspondences over a field $k$ in the sense of \Cref{def:correspondencesCorrS},
and suppose that the canonical functor $\Sm_k\to \Corr_k$ passes through the graded classical category of framed correspondences $\Fr_+(k)$ defined in \cite{GP14}, and the functor $\Fr_+(k)\to \Corr_k$ takes the morphisms $\sigma_X$ for all $X\in \Sm_k$ to the elements in $\Map_{\Corr_k}(X,X)$ equivalent to the identity morphisms. %changed 
\begin{proposition}[\'Etale excision]\label{prop:NisExcLoc}
Given a smooth scheme $X$ over a field $k$,
a Nisnevich square $R$ of the form \eqref{eq:NsquareWX},
and a point $x\in X$, we have that
any $\A^1$-invariant presheaf $F\in \SPre^\tr_\Sigma(k)$
is $R$-excisive over the local scheme $X_x$.
i.e. the map 
\[
F(X_x) \to 
F(W\times_X X_x) \times_{F( W^\prime\times_{X^\prime} X^\prime_{p^{-1}(x)} )} F(X^\prime_{p^{-1}(x)})
\]
is an equivalence. 
\end{proposition}
\begin{proof}
It suffices to show that the map 
\[
\pi_i
\Fib (F(X_x) \stackrel{i^*}\to F( W\times_X X_x )) 
\to 
\pi_i
\Fib(F(X^\prime_{p^{-1}(x)}) \stackrel{i^{'*}}\to F(W^\prime\times_{X^\prime} X^\prime_{p^{-1}(x)})
\]
%todo
is an isomorphism for any $i\in\mathbb Z$.
%todo morphism(eq:NsquareWX)anddegree(abovedyspaly)
% $$\pi_i\Fib (F(X_x) \stackrel{i^*}\to F(U_{i^{-1}(x)})) \to \pi_i\Fib(F(V_{p^{-1}(x)}) \stackrel{i^{'*}}\to F((U \times_X V)_{(p'i)^{-1}(x)}))$$
% is an isomorphism for any $i$. 
% There are the long exact sequences
% \[\begin{array}{lclclclcl}
% \cdots
% &
% \to
% &
% \pi_{i+1}F(W\times_X X_x)
% &
% \to
% &
% \pi_i
% \Fib (F(X_x) \stackrel{i^*}\to F(W^\prime\times_{X^\prime} X_x )) 
% &
% \to
% &
% \pi_i(F(X_x))
% &
% \to
% &
% \cdots
% \\
% \cdots
% &
% \to
% & 
% \pi_{i+1}F(W^\prime\times_{X^\prime} X^\prime_{p^{-1}(x)})
% &
% \to
% &
% \pi_i
% \Fib (F(X^\prime_{p^{-1}(x)}) \stackrel{i^*}\to F( W^\prime\times_{X^\prime} X^\prime_{p^{-1}(x)} )) 
% &
% \to
% &
% \pi_i(F(X^\prime_{p^{-1}(x)}))
% &
% \to
% &
% \cdots
% \end{array}\]
There is a long exact sequence
\begin{equation}\label{eq:hglesF(WX)}
\cdots
\to
\pi_{i+1}F(W\times_X X_x)
\to
\pi_i
\Fib \big(F(X_x) \stackrel{i^*}\to F(W^\prime\times_{X^\prime} X_x )\big) 
\\
\to
\pi_i(F(X_x))
\to
\cdots
\end{equation}
Note that we are implicitly using here that taking fibres commutes 
with taking filtered colimits. 
The long exact sequence \eqref{eq:hglesF(WX)}
together with the similar one for 
$X^\prime$ instead of $X$,
$X^\prime_{p^{-1}(x)}$ instead of $X_x$,
and $W^\prime$ instead of $W$,
imply
%Hence there are 
the short exact sequences
\[\begin{array}{lclcl}
\Coker (\pi_{i+1}(i^*))
&
\to
&
\pi_i
\Fib (F(X_x) \stackrel{i^*}\to F( W\times_X X_x )) 
&
\to
&
\Ker (\pi_i(i^*))
\\
\Coker (\pi_{i+1}(i^{'*}))
&
\to
&
\pi_i
\Fib (F(X^\prime_{p^{-1}(x)}) \stackrel{i^*}\to F( W^\prime\times_{X^\prime} X_x )) 
&
\to
&
\Ker (\pi_i(i^{'*}))
\end{array}\]
% Since taking homotopy groups commutes with filtered colimits hidden in the definition of values of a presheaf on essentially smooth schemes, the left-hand side is an extension of $\Ker (\pi_i(i^*))$ by $\Coker (\pi_{i+1}(i^*))$ 
% and the right-hand is an extension of 
% $\Ker (\pi_iF(i^{'*}))$ by $\Coker (\pi_{i+1}F(i^{'*}))$. 
The kernels in the above short exact sequences are trivial by \cite[Theorem~2.15(3)]{GP-HIVth} and \cite[Corollary~3.6(2)]{DrKyl}, 
and the maps 
\[
\Coker (\pi_{i+1}(i^*))
\to 
\Coker (\pi_{i+1}(i^{'*}))
\]
are isomorphisms by \cite[Theorem~2.15(5)]{GP-HIVth} in combination with the main result of \cite{DruzhPanin-SurjEtEx} and \cite[Corollary~3.6(4)]{DrKyl}. 
\end{proof}
%todo change
% \begin{corollary}\label{cor:NisExcLoc}
% Given a smooth scheme $X$ over a field $k$,
% a Nisnevich square $R$ of the form \eqref{eq:NsquareWX},
% and a point $x\in X$. 
% Then any $\A^1$-invariant additive presheaf of abelian groups $F$ on $\Sm_S$
% the sequence
% \[
% 0 \to F(X_x) \to 
% F(W\times_X X_x) \oplus F(X^\prime_{p^{-1}(x)})\to 
% F(W^\prime\times_{X^\prime} X^\prime_{p^{-1}(x)})
% \to 0
% \]
% is short exact. 
% \end{corollary}
% \begin{proof}
% By Proposition \ref{prop:NisExcLoc} 
% there is a quasi-isomorphism of complexes
% \[
% F(X_x) \to \mathrm{coCone}(
% F(W\times_X X_x) \oplus F(X^\prime_{p^{-1}(x)})\to 
% F( W^\prime\times_{X^\prime} X^\prime_{p^{-1}(x)} )
% ).
% \]
% So the claim follows.
% \end{proof}

% \begin{proposition}\label{prop:SmZar=NIs}
% % Given a category of correspondences $\Corr_k$ over a field $k$ that satisfies (\'EtEx),
% %Any essentially smooth scheme $X$ over a spectrum of a field $S$ is Zar-Nis motivically trivial. 
% %In other words 
% Any group-like $\A^1$-invariant Zariski sheaf $F \in \SPrefr_\Sigma(S)$ is a Nisnevich sheaf. %on $\Et_X$.
% \end{proposition}
% \begin{proof}
% % It suffices to prove that for any Nisnevich square of smooth schemes~\ref{eq:NissqXUV} 
% % the map $F(X) \to F(U) \times_{F(U \times_X V)} F(V)$ is an equivalence. 
% Combining Proposition~\ref{prop:NisExcLoc} and 
% Lemma~\ref{lm:NexZpoints} % Lemma~\ref{lm:ZarLocCov} 
% we get the claim. %see that it is an equivalence. 
% \end{proof}

\subsection{Zariski and Nisnevich motivic localisations}\label{subsect:EtExLocA1niszar}% of \'etale excisive presheaves
Let $k$ be a field.

\begin{definition}\label{def:etex}
An $\infty$-category of correspondences $\Corr_k$ over $k$, see \Cref{def:preaddcorrespondencescategoryCorrS}%def:corrcat
, satisfies  \emph{the property (\'EtEx)} if
any $\A^1$-invariant presheaf $F\in \SPretr_\Sigma(S)$ is $R$-excisive over the essetially smooth local scheme $X_x$, for any Nisnevich square $R$ of the form \eqref{eq:NsquareWX}  over $k$. See \Cref{def:excisive}.
\end{definition}

\begin{example}\label{ex:CorrfrEtEx}
\Cref{prop:NisExcLoc} equivalently claims that
the $\infty$-category $\Corr^{\gp,\fr}(k)$ equipped with the composite functor $\Sm_k\to \Corr^\fr(k)\to \Corr^{\gp,\fr}(k)$ satisfies (\'EtEx).
\end{example}
% \Cref{prop:SmZar=NIs} can be formulated in general in the following form.
% 
% Combining Proposition~\ref{prop:NisExcLoc} and 
% Lemma~\ref{lm:NexZpoints} % Lemma~\ref{lm:ZarLocCov} 
% we get the claim. %see that it is an equivalence. 

\begin{lemma}\label{lm:CorrSmZar=NIs}
Given an $\infty$-category of correspondences $\Corr_k$ over %a
$k$, see \Cref{def:preaddcorrespondencescategoryCorrS}, 
%and a functor $\Sm_k\to \Corr_k$ 
that satisfies (\'EtEx),
%Any essentially smooth scheme $X$ over a spectrum of a field $S$ is Zar-Nis motivically trivial. 
%In other words 
% any $\A^1$-invariant Zariski sheaf $F \in \SPretr(S)$ is a Nisnevich sheaf. %on $\Et_X$.
then
an $\A^1$-invariant presheaf $\mathcal F\in \SPre^\mathrm{tr}_\Sigma(k)$ is a Nisnevich sheaf if and only if $F$ is a Zariski sheaf.
In particular, any group-like $\A^1$-invariant Zariski sheaf $F \in \SPrefr_\Sigma(S)$ is a Nisnevich sheaf. %on $\Et_X$.
\end{lemma}
\begin{proof}
% It suffices to prove that for any Nisnevich square of smooth schemes~\ref{eq:NissqXUV} 
% the map $F(X) \to F(U) \times_{F(U \times_X V)} F(V)$ is an equivalence. 
The first claim follows by \Cref{def:etex} from 
Lemma~\ref{lm:NexZpoints}, the second by \Cref{ex:CorrfrEtEx}. % Lemma~\ref{lm:ZarLocCov} 
 %see that it is an equivalence. 
\end{proof}

% \begin{corollary}\label{cor:ZarTrNis}
% Let $\Corr_k$ be a category of correspondences over a field $k$ that satisfies (\'EtEx).
% % An $\A^1$-invariant presheaf $\mathcal F\in \SPre(\mathrm{Corr}^\mathrm{tr}_k)$ is Nisnevich sheaf iff $F$ is a Zariski sheaf.
% % An $\A^1$-invariant presheaf $\mathcal F\in \SPre(\mathrm{Corr}^\mathrm{tr})$ is Nisnevich acyclic iff it is Zariski acyclic.
% \end{corollary}

% \section{Categorical result}

% The subcategories of $\SPre_\Sigma(S)=\SPre_\Sigma(\Sm_S)$, $\SPre^\tr_\Sigma(S)=\SPre_\Sigma(\mathrm{Corr}_S)$ consisting of 
% $\A^1$-invariant, Zariski local, Nisnevich local objects, and their intersection 
% are reflective subcategories in $\SPre_\Sigma(S)$ or $\SPretr_\Sigma(S)$.
% We denote the corresponding localisation functors by 
% $L_{\A^1}, L_{\zar}$ (resp. $L_{\nis})$, and $L_{\A^1\zar}$ (resp. $L_{\A^1\nis})$. 

%\begin{theorem}
%Consider the category $\SPre_{\A^1}(\mathrm{Corr}^{\mathrm{fr})$ that is the $\A^1$-localisation of $\SPre(\mathrm{Corr}^{\mathrm{fr})$. 
%Let $L_\mathrm{nis}^{\A^1}$ and $L_\mathrm{Zar}^{\A^1}$ denote the functors of Nisnevich and Zariski local replacements on $\SPre_{\A^1}(\mathrm{Corr}^{\mathrm{fr})$.
%Then $L_\mathrm{nis}\simeq L_\mathrm{Zar}$ 
%\end{theorem}
%
\begin{theorem}\label{th:SHZar=Nis}
Let $k$ be a field, and $\Corr_k$ be an $\infty$-category of correspondences over $k$ that satisfies (\'EtEx).
There is an equivalence of $\infty$-categories 
$\mathbf{H}^{\tr}_{\nis}(k)\simeq \mathbf{H}_{\zar}^{\tr}(k)$.
In particular,
$\mathbf{H}^{\fr,\gp}_{\nis}(k)\simeq \mathbf{H}_{\zar}^{\tr,\gp}(k)$,
$\SH^{\fr}_\nis(k)\simeq \SH^{\fr}_\zar(k)$.
\end{theorem}
\begin{proof}
The first claim follows from \Cref{lm:CorrSmZar=NIs}, since 
the $\infty$-categories $\mathbf{H}_{\nis}^{\tr}(k)$ and $\HH_{\zar}^{\tr}(k)$ are 
the full subcategories of in $\SPretr_\Sigma(k)$ whose objects are group-like $\A^1$-invariant Nisnevich and Zariski sheaves, respectively.
The case of framed presheaves follows because of \Cref{ex:CorrfrEtEx} and the equivalence $\SPre^{\fr,\gp}(k)\simeq \SPre(\Corr^{\fr,\gp}(k))$ provided by \cite[Lemma 1.6]{UnivMultSp}.
Note that the case of $\mathbf{H}^{\tr}(k)$ implies the equivalences for $\mathbf{SH}^{S^1,\tr}(k)$, $\mathbf{SH}^{\tr}(k)$.
\end{proof}

% \begin{remark}\label{rem:LfrzarLfrnisLzarLnis}
% Let us note that
% the Nisnevich localisations on categories $\SPrefrgp_\Sigma(S)$ agrees with the one on $\SPre_\Sigma(S)$, 
% and the same holds for Zariski fibre on the subcategories of $\A^1$-invariant presheaves,
% but the statement for $L_\zar$ is probably not true,
% though these three statements are not discussed in the text.

% In terms of cohomologies of abelian presheaves the first statement claims the isomorphism of
% the Nisnevich cohomology groups $H_\nis^i(X,F)$ for a presheaf with framed transfers $F$ over a scheme $S$ and the groups $\mathrm{Ext}_{\Sh^\mathrm{tr}_\nis(k)}(\mathbb{Z}_\mathrm{tr}(X),F[i])$ in the category of Nisnevich sheaves with transfers.
% The argument in \cite{Voe-motives} holds for any field $k$, and a presheaf with $\mathrm{Cor}_k$-transfers, and moreover the same source contains the proof for the Zariski case over a field.
% The statement for framed transfers and the Nisnevich topology is contained \cite{GP-HIVth}.
% \end{remark}

\subsection{Result over base schemes}\label{subsect:ResultingSummary}

% \subsection{Excision theorem}

% In the subsection we prove
% \Cref{prop:relaitve:SHfrS1ZarSHfrSNis} proved in this subsection proves \Cref{th:SHtrzarSHtrnis}, and 
% in view of \Cref{ex:tf} and \Cref{ex:zf} it implies \Cref{th:SHfrzarSHfrnis}.

% Let $\Corr_S$ be a family of preadditive $\infty$-categories of correspondences, see \Cref{def:preaddcorrespondencesCorrS}, that satisfies properties (Embed), (FinE), and (\'EtEx), \Cref{def:FinECorr,def:EmbedCorr,def:etex}.

% We use easy fact to reduce the case of $\Sm_S$ to $\SmAff_S$.
\begin{lemma}\label{lm:widetildeSzarSConservativity}
Let $\nu$ be the Zariski or the Nisnevich topology on $\Sm_S$ over a noetherian separated scheme $S\in\Sch$ of finite Krull dimension.
Then 
\begin{itemize}
\item[(1)]
the canonical restriction 
\[\Shv^\tr_\nu(S)\to \Shv^\tr_\nu(\SmAff_S)\]
is an equivalence;
\item[(2)]
for a Zariski covering $\widetilde S\to S$
the inverse image functor
\[\Shv^\tr_\nu(S)\to \Shv^\tr_\nu(\widetilde S)\]
is conservative.
\end{itemize}
\end{lemma}
\begin{proof}
(1) Any scheme in $\Sm_S$ has a Zariski covering $v\colon \widetilde S\to S$ in $\SmAff_S$, and consequently, there is a $\nu$-local equivalence
$h(S)\simeq_\nu \colim_{[n]\in\Delta}h(\widetilde S^{\times n})$.
Applying $\gamma^*\colon \Pre(S)\to\Pre^\tr(S)$,
we get $h^\tr(S)\simeq_\nu \colim_{[n]\in\Delta}h^\tr(\widetilde S^{\times n})$.
Hence the fully faithful Kan extension functor $\Shv^\tr_\nu(\SmAff_S)\to \Shv^\tr_\nu(S)$ is essentially surjective.
So the claim follows.

(2)
A morphism $F\to G \in \Shv^\tr_\nu(S)$ is an equivalence 
if the induced map
$\Map(h^\tr(X), F) \to \Map(h^\tr(X), G)$ 
is an equivalence for 
any $X \in \Sm_S$. 
Using the equivalence shown above, this induced map can be rewritten 
as
\[
\operatorname{lim}_{[n]\in\Delta}F(X\times_S\widetilde S^{\times n}) \to \operatorname{lim}_{[n]\in\Delta}G(X\times_S\widetilde S^{\times n}).
\]
Each map in the simplicial diagram is an equivalence by assumption.
\end{proof}

\begin{theorem}\label{prop:relaitve:SHfrS1ZarSHfrSNis}
Let $\Corr_{(-)}$ be a family of preadditive $\infty$-categories of correspondences on 
% $\EssSm_{(-)}$, 
$\Sch_{\Schfns}$,
see \Cref{def:correspondencesCorrS}, 
that 
is continuous,
%satisfies properties %(Embed), 
%(FinE), 
see \Cref{def:FinECorr}, 
satisfies 
the closed gluing on affine schemes, satisfies (AHP), see \Cref{def:RLwAHP}, and 
satisfies (\'EtEx) for each field $k$, see \Cref{def:etex}. %,def:EmbedCorr

Then the canonical functor
\begin{equation*}%\label{}
\Htrgp_{\zf}(S)\to \Htrgp_{\nis}(S),
\end{equation*} %^\fr_{\A^1,\nis}
is an equivalence for any neotherian separated scheme $S$ of finite Krull dimension.
% In particular, the equivalence holds for $\zf$-topology.
\end{theorem}
\begin{proof}
By Theorem \ref{th:SHZar=Nis} the claim holds over fields.
Let $S$ be affine scheme of finite Krull dimension.
Let $F\in \Htrgp_{\zf}(\SmAff_S)$. 
The claim is that $F\in \Htrgp_{\nis}(\SmAff_S)$.
% Let $F\in \SH^{S^1,\tr}_{\zar\cup\tau}(\SmAff_S)$, we are going to show that $F\in \SH^{S^1,\tr}_{\nis\cup\tau}(\SmAff_S)$, i.e. it is a Nisnevich sheaf.
By \Cref{th:SHZar=Nis} $F$ goes to the Nisnevich sheaf under the functor
\[\Htrgp_{\zf}(\SmAff_S)\to \prod_{z\in S}\Htrgp_{\zar}(\SmAff_z),\]
By \Cref{lm:Loc(Aff)detectssheaves} applied to the topologies $\zf$ and $\nis$ in view of \Cref{ex:embedZarNistfFinEmbed} the claim follows. 
So the claim holds for
the $\infty$-categories 
$\Htrgp_{\tau}(\SmAff_S)$, 
where $\tau=\zf,\nis$.

By the first point of \Cref{lm:widetildeSzarSConservativity}
the claim for the $\infty$-categories 
$\Htrgp_{\zf}(\SmAff_S)$ 
follows.
The claim for an arbitrary scheme $S$ of finite Krull dimension 
follows by the second point of \Cref{lm:widetildeSzarSConservativity}
applied to a Zariski covering $\coprod_{\beta} S_\beta\to S$ with affine schemes $S_\beta$. 
\end{proof}
\begin{remark}
    It is expected that \Cref{prop:relaitve:SHfrS1ZarSHfrSNis} holds for all qcqs schemes for the appropriate definition of the Zariski topology.
\end{remark}
   
% \end{corollary}

\printbibliography

\end{document}